\documentclass[11pt]{amsart}
\usepackage{amssymb}
\usepackage{enumerate}
\usepackage{enumitem}
\usepackage[margin=1in]{geometry}
\usepackage{graphicx}
\usepackage{hyperref}
\usepackage{mathtools}


\newtheorem{theorem}{Theorem}[section]
\newtheorem{lemma}[theorem]{Lemma}
\newtheorem{corollary}{Corollary}

\theoremstyle{definition}
\newtheorem{definition}[theorem]{Definition}
\newtheorem{example}[theorem]{Example}

\theoremstyle{remark}
\newtheorem{remark}[theorem]{Remark}

\numberwithin{equation}{section}

\newcommand{\st}{\,\vert\,}
\newcommand{\bigst}{\,\bigg|\,}
\newcommand{\isom}{\cong}
\newcommand{\abs}[1]{\left| #1 \right|}
\newcommand{\bigabs}[1]{\bigg| #1 \bigg|}
\newcommand{\hh}{\mathbb{H}}
\newcommand{\rr}{\mathbb{R}}
\newcommand{\nn}{\mathbb{N}}
\newcommand{\zz}{\mathbb{Z}}
\newcommand{\pp}{\mathbb{P}}
\newcommand{\qq}{\mathbb{Q}}
\newcommand{\cc}{\mathbb{C}}
\newcommand{\bb}[1]{\mathbb{#1}}
\newcommand{\eye}{\mathbf{I}}
\newcommand{\mc}[1]{\mathcal{#1}}
\newcommand{\restrict}{\big|}
\newcommand{\boundary}{\delta}
\newcommand{\PGL}{\mathrm{PGL}}
\newcommand{\PSL}{\mathrm{PSL}}

\newcommand{\Fricke}[1]{\Gamma_0^*(#1)}

\newcommand{\Ex}{\mathrm{Ex}}
\newcommand{\lmod}[2]{#1 \backslash #2}

\newcommand{\LeftQuotientByStab}[1]{\lmod{#1_\infty}{#1}}
\newcommand{\Dmat}[1]{\left\lbrack #1\right\rbrack}
\renewcommand{\emptyset}{\varnothing}
\renewcommand{\phi}{\varphi}
\renewcommand{\Im}{\mathrm{Im}\,}
\renewcommand{\Re}{\mathrm{Re}\,}
\renewcommand{\setminus}{\smallsetminus}
\DeclareMathOperator{\divides}{\mid}
\DeclareMathOperator{\ndivides}{\nmid}
\DeclareMathOperator{\edivides}{\parallel}

\begin{document}

\title{Zeros of Replicable Functions}
\author{Ben Toomey}
\address{Oregon State University}
\email{toomeyb@oregonstate.edu}


\date{\today}

\begin{abstract}
	Following the work of Asai, Kaneko, and Ninomiya for Faber polynomials associated to $\PSL_2(\zz)$, and Bannai, Kojima, and Miezaki's partial proof for the case of $\Fricke{2}$, we show that the zeros of certain modular functions associated to some low-level genus zero groups are all located on the boundary of certain natural fundamental domains for $\Gamma$. The groups considered are $\Fricke{2}$, $\Fricke{3}$, $\Gamma_0(2\edivides 2)$, $\Fricke{5}$, $\Gamma_0(6)+$, $\Fricke{7}$, $\Gamma_0(4\edivides 2)+$, $\Gamma_0(3\edivides 3)$, and $\Gamma_0(10)+$.
\end{abstract}

\maketitle

\tableofcontents

\section{Introduction}\label{Section:Introduction}

In 1998, Asai, Kaneko, and Ninomiya \cite{AKN} located the zeros of a certain basis for the space of weakly holomorphic modular functions for $\PSL_2(\zz)$, using the action of Hecke operators. A decade later, Bannai, Kojima, and Miezaki \cite{BKM} suggested that the twisted Hecke operators appearing in the Monstrous Moonshine correspondence might be used to locate zeros of weakly holomorphic modular functions for other groups.

Let \begin{equation}\label{Def:SetS} \mc{S} = \{ \Fricke{2}, \Fricke{3}, \Fricke{5}, \Gamma_0(6){+}, \Fricke{7}, \Gamma_0(4\edivides 2)+, \Gamma_0(3\edivides 3), \Gamma_0(10){+} \}. \end{equation} 
These groups (see \S\ref{Subsection:Groups} for a description) have the important property that for each $\Gamma\in\mc{S}$, the associated modular surface $Y(\Gamma)$ is genus zero, and has only one cusp (see \S\ref{Subsection:ModularSurfaces}). As a consequence, the $\cc$-vector space of weakly holomorphic modular functions for $\Gamma$ has a basis of the form $\{ F_n(\tau) = q^{-n} + O(q) \st n\geq 0 \}$, where each $F_n(\tau)$ is uniquely determined. We define a particular fundamental domain $\mc{D}(\Gamma)$ for each $\Gamma\in\mc{S}$ (see \S\ref{Subsection:FundamentalDomains}), and locate the zeros of this basis $\{ F_n(\tau) \}$ in the given fundamental domain $\mc{D}(\Gamma)$.

\begin{theorem}\label{Theorem:MainTheorem}
	Let $\Gamma\in\mc{S}$ and let $n\in\nn$. Then all $n$ zeros for the unique modular function $F_n(\tau) = q^{-n}+O(q)$ for $\Gamma$ are on the lower boundary of the fundamental domain $\mc{D}(\Gamma)$, except for the cases of $n=2$ and $\Gamma = \Fricke{5},\Fricke{7}$, where one zero is on the lower boundary and one zero is on the side boundary of $\mc{D}(\Gamma)$.	
\end{theorem}

We demonstrate the particular cases of $\Fricke{2}$, $\Gamma_0(6){+}$, and $\Gamma_0(3\edivides 3)$ below in Section \ref{Section:Cases}. The other cases are similar in nature. While there are several other genus zero groups with one cusp, the groups in $\mc{S}$ have the property that the functions $F_{n}(\tau)$ are real-valued on the lower boundary of $\mc{D}(\Gamma)$, allowing us to use the intermediate value theorem to locate zeros.

Bannai, Kojima, and Miezaki \cite{BKM} obtained a partial proof in the case of $\Fricke{2}$, locating at least $n-1$ of the $n$ zeros for each function $F_n(\tau)$ in the fundamental domain $\mc{D}(\Fricke{2})$. There are some minor errors in the write up of their proof, having to do with the bookkeeping involved with twisted Hecke operators (their formula has the replicate function depending on $n$, when it instead depends on $a$, where $a\divides n$). In order to avoid this bookkeeping, and the inevitable errors that would creep in, we develop a technique to bypass this part of the proof method.

For the particular cases of $\Fricke{2}$ and $\Fricke{3}$, our results are previously known, due to results of Choi and Im \cite{ChoiIm} and Hanamoto and Kuga \cite{HanamotoKuga}, respectively. Their results cover a broader class of modular forms, but since the methods used are sufficiently dissimilar, we still offer our own proofs below.

\section{Motivation and description of method}\label{Section:Motivation}

We briefly summarize the strategies of the prior work mentioned in Section \ref{Section:Introduction}, which motivate our method. Let \[ \mc{H}_n := \left\{ \begin{psmallmatrix} a & b \\ 0 & d \end{psmallmatrix} \bigst a,b,d\in\zz,\, ad = n,\, 0\leq b< d \right\}, \] which we call the Hecke set of level $n$ (see Def. \ref{Def:HeckeSet}). As noted above, Asai, Kaneko, and Ninomiya \cite{AKN} proved the result of Theorem \ref{Theorem:MainTheorem} for the group $\Gamma = \PSL_2(\zz)$. Let \[ T_{1A}(\tau) = j(\tau) - 744 = q^{-1} + 196884q + 21493760q^q + \cdots \] where $j(\tau)$ is the classical $j$-invariant. To locate the zeros of $F_n(\tau)$, they first observe that the classical Hecke operators $T_n$ for modular forms (see, e.g. \cite{SERRE}) provide the identity\footnote{The conflicting notations $T_{1A}$ and $T_n$ are unfortunately rather standard here.} \[ F_n(\tau) = q^{-n} + O(q) = n\left(T_{1A}(\tau)\mid T_n\right) = \sum_{H\in\mc{H}_n} T_{1A}(H\tau). \] They also observe that when $\tau$ is `near infinity' --- that is, has sufficiently large imaginary part --- then $T_{1A}(\tau) \approx q^{-1}$, and in fact, using the non-negativity of the Fourier coefficients of $T_{1A}(\tau)$, they show $\abs{T_{1A}(\tau) - e^{-2\pi i\tau}} < 1335$ whenever $\tau$ is in the commonly-used fundamental domain for $\PSL_2(\zz)$, bounded below by the unit circle and to the left and right by $\abs{\Re\tau}\leq \frac{1}{2}$. We denote this fundamental domain by $\mc{D}(\PSL_2(\zz))$.

So, for each $H\in\mc{H}_n$, let $G_H\in\PSL_2(\zz)$ be such that $G_HH\tau \in\mc{D}(\PSL_2(\zz))$. Then
\begin{align*}
	\abs{ F_n(\tau) - \sum_{H\in\mc{H}_n} e^{-2\pi iG_HH\tau} } 
	&\leq \sum_{H\in\mc{H}_n} \abs{ T_{1A}(H\tau) - e^{-2\pi iG_HH\tau} } \\
	&= \sum_{H\in\mc{H}_n} \abs{ T_{1A}(G_HH\tau) - e^{-2\pi iG_HH\tau} } \tag{modularity of $T_{1A}$} \\
	&< 1335n^2,
\end{align*}
since there are $\sigma_1(n) \leq n^2$ elements in $\mc{H}_n$. The problem then becomes one of approximating the left-hand sum above. By a consideration of cases, and restricting to $\tau$ on the lower boundary of $\mc{D}(\PSL_2(\zz)$ (that is, to the unit circle), they show that for most $H\in\mc{H}_n$, one has $\abs{e^{-2\pi iG_H H\tau}} \leq e^{\pi n\Im\tau}$, but there are three exceptions. By bounding one of these three exceptional terms, the remaining two serve as an approximation of $F_{n}(\tau)$. After some algebra, this gives a bound (see \cite{AKN} for details) \[ \abs{F_n(\tau)e^{-2\pi n\Im\tau} - 2\cos(2\pi n\Re\tau)} < Mn^2e^{-2\pi n\Im\tau} + n^2e^{-\pi n\Im\tau} + 1 < 2, \] for $n\geq 2$ and $\tau$ on the lower boundary of $\mc{D}(\PSL_2(\zz))$. Since $F_n(\tau)$ is real-valued on the lower boundary, the approximation by cosine above shows $F_n(\tau)$ changes sign $n+1$ times along the unit circle with real part in the interval $[0, \frac{1}{2}]$, we deduce that $F_n(\tau)$ has $n$ zeros on the lower boundary of $\mc{D}(\Gamma)$.

Bannai, Kojima, and Miezaki \cite{BKM} sought to extend this technique to families of modular functions for groups other than $\PSL_2(\zz)$. Their key observation is that twisted Hecke operators from Monstrous Moonshine (see \S\ref{Subsection:MonstrousMoonshine} below) could replace the classical Hecke operators above. Briefly, with twisted Hecke operators, one has a family of functions $\{ f^{(a)}(\tau), a\in\nn \}$, with each $f^{(a)}$ a normalized Hauptmodul for a genus zero group $\Gamma^{(a)}$ (analogous to the role of $T_{1A}(\tau)$ and $\PSL_2(\zz)$ above). Letting $f^{(H)} = f^{(a)}$ for $H = \begin{psmallmatrix} a & b \\ 0 & d \end{psmallmatrix}\in\mc{H}_n$, this family of functions satisfies the \emph{twisted Hecke relations}, \[ F_n(\tau) = q^{-n}+O(q) = \sum_{H\in\mc{H}_n} f^{(H)}(H\tau), \] where now $F_n(\tau)$ is the unique weakly-holomorphic modular function for $\Gamma^{(1)}$ having a pole of order $n$ at infinity and holomorphic everywhere else. In particular, letting $f^{(a)} = T_{1A}$ for all $a\in\nn$, the twisted Hecke operators include the case of $\PSL_2(\zz)$ above as a special case.

In Bannai, Kojima, and Miezaki \cite{BKM}, they consider the case of $\Gamma^{(1)} = \Fricke{2}$, where the functions $f^{(a)}$ may either be $T_{1A}(\tau)$ or $T_{2A}(\tau) = q^{-1} + 4372q + 96256q^2 + \ldots$ \cite{CN}, the unique normalized Hauptmodul for $\Fricke{2}$, depending on the parity of $a$. This significantly complicates the bookkeeping, but a consideration of cases was sufficient for them to locate $n-1$ zeros on the lower boundary. With some extra care in our bounding procedure, we are able to locate all $n$ zeros. 

In order to directly extend \cite{AKN}, we will require that $\Gamma = \Gamma^{(1)}$ be a group having only one cusp. This ensures there exists a fundamental domain $\mc{D}(\Gamma)$ which is bounded away from the real line, similar to $\mc{D}(\PSL_2(\zz))$. Somewhat surprisingly, when $\Gamma^{(1)}$ has only one cusp, then $\Gamma^{(a)}$ has one cusp for all $a\in\nn$, and moreover, each normalized Hauptmodul $f^{(a)}$ has non-negative Fourier coefficients (\ref{Corollary:SingleCuspAllPositiveFourierCoefficients}). Taken together, this is sufficient to produce a bound \[ \abs{ F_n(\tau) - \sum_{H\in\mc{H}_n} e^{-2\pi iG_HH\tau} } < Mn^2, \] for some $M>0$, analogous to the case of $\PSL_2(\zz)$ above. Unlike for $\PSL_2(\zz)$, however, here each $G_H\in\Gamma^{(H)}$, which depends on $H$, so that a consideration of cases becomes significantly more error-prone. 

We avoid this brute force approach by proving some group theoretic results relating to twisted Hecke operators in Section \S\ref{Section:Groups}. In Section \S\ref{Section:Functions} we review some needed facts about modular functions and replication. We then apply these results to approximate the modular functions $F_{n}(\tau)$ for groups like those in the set $\mc{S}$ above, in Section \S\ref{Section:ZerosOfFaberPolynomials}, culminating with Theorem \ref{Theorem:MainBound}, which is the analogue of the approximation produced in \cite{AKN} for $\PSL_2(\zz)$. Finally, in Section \S\ref{Section:Cases}, we prove a few cases to demonstrate the method, beginning with $\Fricke{2}$, the simplest nontrivial case involving twisted Hecke operators. We also present the case of $\Gamma_0(6){+}$, which demonstrates handling a case where the lower boundary consists of more than one arc, as well as $\Gamma_0(3\edivides 3)$, where we use additional results (\S\ref{Subsection:Zeros:ConjugateGroups}) involving `harmonics,' as they were dubbed in \cite{CN}. Additional cases may be found in our upcoming thesis \cite{Toomey}.

\section{Groups}\label{Section:Groups}

\subsection{Fractional linear transformations}\label{Subsection:FLTs}

The material here is substantially similar to the exposition given by Duncan and Frenkel \cite{DuncanFrenkel}, with some minor differences in notation.

\begin{definition}\label{Def:OmegaEtc}
	Let
	\begin{align*}
		\Omega 
		&:= \PGL_2^+(\qq) \\
		&= \{ \begin{psmallmatrix} a & b \\ c & d \end{psmallmatrix} \st a,b,c,d\in\qq,\, ad-bc > 0 \} \big{/} \{ t\eye \st t\in\qq\setminus\{0\} \},
	\end{align*}
	where $\eye = \begin{psmallmatrix} 1 & 0 \\ 0 & 1 \end{psmallmatrix}$ is the identity matrix.
	
	We also define the following notation for elements of $\Omega$. Let $x,y\in\qq$ with $y>0$. Then \[ T^x := \begin{psmallmatrix} 1 & x \\ 0 & 1 \end{psmallmatrix},\quad \Dmat{y} := \begin{psmallmatrix} y & 0 \\ 0 & 1 \end{psmallmatrix},\quad \text{ and } S := \begin{psmallmatrix} 0 & -1 \\ 1 & 0 \end{psmallmatrix}. \]
	
	We further define the following subgroups of $\Omega$: \[ \Omega_\infty := \{ \begin{psmallmatrix} a & b \\ c & d \end{psmallmatrix}\in\Omega \st c = 0 \}, \]
	\begin{align*}
		\Omega_\infty^T 
		&:= \{ \begin{psmallmatrix} a & b \\ 0 & d \end{psmallmatrix}\in\Omega_\infty \st a = d \} \\
		&= \{ T^x \st x\in\qq \},
	\end{align*}
	and
	\begin{align*}
		\Omega_\infty^D 
		&:= \begin{psmallmatrix} a & b \\ 0 & d \end{psmallmatrix}\in\Omega_\infty \st b = 0 \} \\
		&= \{ \Dmat{y} \st y\in\qq, y>0 \}.
	\end{align*}	
\end{definition}

That is, $\Omega$ is the group of all rational $2\times 2$ matrices with positive determinant, up to scalar equivalence. For elements of $\Omega$, we generally choose a matrix for coset representative having integral entries, by simultaneously scaling the entries when necessary. In particular, for $\frac{p}{q}\in\qq$, we often write \[ T^{\frac{p}{q}} = \begin{psmallmatrix} q & p \\ 0 & q \end{psmallmatrix}, \qquad \Dmat{\frac{p}{q}} = \begin{psmallmatrix} p & 0 \\ 0 & q \end{psmallmatrix}, \] with the latter expression in $\Omega$ only for $\frac{p}{q} > 0$.

For any $M = \begin{psmallmatrix} a & b \\ c & d \end{psmallmatrix}$, we may take as coset representative for $M^{-1}$ the matrix $\begin{psmallmatrix} -d & b \\ c & -a \end{psmallmatrix}$, having the same determinant as the coset representative for $M$. Since an involution $W\in\Omega$ satisfies $W^{-1} = W$ (up to scalar multiplication), from the form of the inverse above we may deduce that involutions are precisely the elements of the form $W = \begin{psmallmatrix} a & b \\ c & -a \end{psmallmatrix}$ for some $a,b,c\in\zz$. That is, involutions are the trace zero elements of $\Omega$.

\begin{lemma}\label{Lemma:MatrixCommutations}
	Let $s,t,x,y\in\qq$ with $x,y>0$. Then
	\begin{enumerate}
		\item $T^sT^t = T^{s+t}$, 
		\item $(T^s)^t = T^{st}$,
		\item $\Dmat{x}\Dmat{y} = \Dmat{xy}$, and in particular $\Dmat{x}^{-1} = \Dmat{x^{-1}}$,
		\item $S^{-1} = S$,
		\item $\Dmat{y}T^t = T^{ty}\Dmat{y}$,
		\item $\Dmat{y}S = S\Dmat{y}^{-1}$,
		\item $ST^tS = \begin{cases} \eye & t = 0 \\ T^{-\frac{1}{t}}\Dmat{\frac{1}{t^2}}ST^{-\frac{1}{t}} & t\neq 0	\end{cases}$.
	\end{enumerate}
\end{lemma}
\begin{proof}
	In each case, one computes each side as matrices and compares. Note the identity $(T^s)^t = T^{st}$ involves an appropriate choice of root when $t\in\qq$ is not an integer, that is, we are implicitly defining $(T^s)^\frac{1}{n} = T^\frac{s}{n}$ for any nonzero $n\in\zz$ (one may check that this is the unique $n$th root of $T^s$ in $\Omega_\infty$).
\end{proof}

In preparation for Theorem \ref{Theorem:InvolutoryDecomposition} below, we introduce some further notation:

\begin{definition}\label{Def:PiRhoSquaredSigmaTheta}
	For any $M = \begin{psmallmatrix} a & b \\ c & d \end{psmallmatrix}\in \Omega$, we define the following functions from $\Omega$ to $\pp^1(\qq)$:
	\begin{align*}
		\pi(M) &:= -\frac{d}{c}, \\
		\rho^2(M) &:= \frac{ad-bc}{c^2}, \\
		\sigma(M) &:= \begin{cases} \frac{a}{d} & M\in\Omega_\infty \\ 1 & M\notin\Omega_\infty \end{cases}, \\
		\theta(M) &:= \begin{cases} \frac{b}{d} & M\in\Omega_\infty \\ \frac{a}{c} & M\notin\Omega_\infty \end{cases}.		
	\end{align*}
	We further define\footnote{We note that Duncan and Frenkel define the function $\varrho$, where $\varrho(M)=\rho^2(M)$ (not $\rho(M)$).} $\rho:\Omega\to\pp^1(\rr)$ by $\rho(M) = \sqrt{\rho^2(M)}$.
\end{definition}

We now define an important decomposition of elements of $\Omega$.

\begin{theorem}[Involutory Decomposition]\label{Theorem:InvolutoryDecomposition}
	Let $M = \begin{psmallmatrix} a & b \\ c & d \end{psmallmatrix}\in\Omega$. Then \[ M = \begin{cases} T^{\theta(M)}\Dmat{\sigma(M)} & M\in\Omega_\infty \\ T^{\theta(M)}\Dmat{\rho^2(M)}ST^{-\pi(M)} & M\notin\Omega_\infty. \end{cases} \] Moreover, this factorization is unique, in that if $M = T^{x}\Dmat{y}ST^{z}$ (resp. $M = T^{x}\Dmat{y}$) then $x = \theta(M)$, $y = \rho^2(M)$, and $z = -\pi(M)$ (resp. $x = \theta(M)$, $y = \sigma(M)$).
\end{theorem}
\begin{proof}
	First, suppose $M\in\Omega_\infty$, so \[ M = \begin{psmallmatrix} a & b \\ 0  & d \end{psmallmatrix} = \begin{psmallmatrix} 1 & \frac{b}{d} \\ 0 & 1 \end{psmallmatrix}\begin{psmallmatrix} a & 0 \\ 0 & d \end{psmallmatrix} = T^{\theta(M)}\Dmat{\sigma(M)}. \] Moreover, if $M = T^{\theta(M)}\Dmat{\sigma(M)} = T^x\Dmat{y}$, then $T^{x - \theta(M)} = \Dmat{\sigma(M)y^{-1}}$, and since $\Omega_\infty^T\cap\Omega_\infty^D = \{\eye \}$ we must have $T^{x-\theta(M)} = \eye = \Dmat{\sigma(M)y^{-1}}$, so $x=\theta(M)$ and $y=\sigma(M)$, and the given factorization is unique.
	
	Now suppose $M = \begin{psmallmatrix} a & b \\ c & d \end{psmallmatrix}\notin\Omega_\infty$. First, we compute that 
	\begin{align*} 
		T^{\theta(M)}\Dmat{\rho^2(M)}ST^{-\pi(M)} 
		&= \begin{psmallmatrix} 1  & \theta(M) \\ 0 & 1 \end{psmallmatrix} \begin{psmallmatrix} \rho^2(M) & 0 \\ 0 & 1 \end{psmallmatrix} \begin{psmallmatrix} 0 & -1 \\ 1 & 0 \end{psmallmatrix} \begin{psmallmatrix} 1 & -\pi(M) \\ 0 & 1 \end{psmallmatrix} \\
		&= \begin{psmallmatrix} \theta(M) & -\theta(M)\pi(M) - \rho^2(M) \\ 1 & -\pi(M) \end{psmallmatrix} \\
		&= \begin{psmallmatrix} \frac{a}{c} & \frac{b}{c} \\ 1 & \frac{d}{c} \end{psmallmatrix} \tag{since $-\theta(M)\pi(M) - \rho^2(M) = \frac{ad}{c^2} - \frac{ad-bc}{c^2}$}\\
		&= M,
	\end{align*}
	as desired. Now, suppose $T^{x}\Dmat{y}ST^{z} = T^{\theta(M)}\Dmat{\rho^2(M)}ST^{-\pi(M)}$, so that by rearranging, we find \[ \Dmat{\frac{1}{\rho^2(M)}}T^{x-\theta(M)}\Dmat{y} = ST^{-\pi(M)-z}S. \] Since the left hand side is in $\Omega_\infty$, the right hand side must be as well. But since $S\infty = S^{-1}\infty = 0$, we have $ST^{-\pi(M)-z}S\infty = \infty$ if and only if $T^{-\pi(M)-z}(0) = 0$, that is, if and only if $z = -\pi(M)$. But then $\Dmat{\frac{1}{\rho^2(M)}}T^{x-\theta(M)}\Dmat{y} = \eye$, so that $x = \theta(M)$ and $y = \rho(M)$ follows immediately, and the decomposition given is unique.
\end{proof}

One reason to call this factorization the `involutory decomposition' is that when $M\notin\Omega_\infty$, the factorization above includes the involution $S$.   

The following identities are useful in computations involving the involutory decomposition:
\begin{lemma}\label{Lemma:PiRhoComputations}
	Let $M\in\Omega$, and let $H\in\Omega_\infty$. Then
	\begin{enumerate}
		\item $\pi(M) = M^{-1}\infty$,
		\item $\pi(HM) = \pi(M)$, and $\pi(MH) = H^{-1}\pi(M)$,
		\item $\rho^2(M) = \rho^2(M^{-1})$,
		\item $\rho^2(HM) = \sigma(H)\rho^2(M)$, and $\rho^2(MH) = \sigma(H)^{-1}\rho^2(M)$,
		\item $\theta(HM) = \theta(H) + \sigma(H)\theta(M)$.
	\end{enumerate}
\end{lemma}
\begin{proof}
	The proofs are all direct computations.
\end{proof}
An important corollary for our purposes is the following.
\begin{corollary}\label{Corollary:PiRhoOnCosets}
	Let $\Gamma\subset \Omega$ be a subgroup such that \[ \Gamma_\infty := \{ K\in\Gamma \st K\infty = \infty \} \subset \Omega_\infty^T. \] Then $\pi$ and $\rho$ descend to $\LeftQuotientByStab{\Gamma}$, that is, $\pi,\rho^2:\LeftQuotientByStab{\Gamma}\to\pp^1(\qq)$ are well-defined functions (by abuse of notation, we use the same name irrespective of domain).
\end{corollary}
\begin{proof}
	Any coset representative for $[K]\in\LeftQuotientByStab{\Gamma}$ has the form $HK$ for some $H\in\Gamma_\infty\subset\Omega_\infty^T$. Then $\pi(HK) = \pi(K)$ and $\rho^2(HK) = \sigma(H)\rho^2(K) = \rho^2(K)$, since $\sigma(H) = 1$ for all $H\in\Omega_\infty^T$, so $\pi$ and $\rho^2$ are constant on cosets, as desired.
\end{proof}

Now, let 
\begin{align*} 
	\hh = \{\tau \in \cc \st \Im \tau > 0 \}, \\ 
	\widehat{\hh} &= \hh\cup\qq\cup\{\infty\},
\end{align*}  
with the topology on $\widehat{\hh}$ given by open sets bounded by horocycles centered at the points in $\qq\cup\{\infty\}$ (see, e.g., \cite{Shimura}). Then $\Omega$ acts on the projective rational line $\pp^1(\qq) = \qq\cup\{\infty\}$ and on the upper half-plane $\hh$ by linear fractional transformations, \[ \begin{psmallmatrix} a & b \\ c & d \end{psmallmatrix}\tau = \frac{a\tau+b}{c\tau+d}, \] and this action is continuous on $\widehat{\hh}$. The remainder of this subsection summarizes the geometry of this action, and in particular what the involutory decomposition in Thm. \ref{Theorem:InvolutoryDecomposition} tells us.

The subgroup $\Omega_\infty$ is the stabilizer of $\infty$ under this action, while $\Omega_\infty^T$ and $\Omega_\infty^D$ act on $\widehat{\hh}$ by translations and dilations, respectively. With this geometric point of view, many of the algebraic identities of Lemma \ref{Lemma:MatrixCommutations} describe well-known facts, for example, translating by $t$ then $s$ is the same as translating by $s+t$ ($T^sT^t = T^{s+t}$), and scaling by $y$ then $x$ is the same as scaling by $xy$ ($\Dmat{x}\Dmat{y} = \Dmat{xy}$). From the involutory decomposition in Thm. \ref{Theorem:InvolutoryDecomposition}, we see that elements of $\Omega_\infty$ are uniquely described by some scaling in $\Omega_\infty^D$ followed by a horizontal translation in $\Omega_\infty^T$. Using an identity from Lemma \ref{Lemma:MatrixCommutations}, the statement $T^{x}\Dmat{y} = \Dmat{y}T^{xy}$ witnesses this transformation as a (different) horizontal translation followed by (the same) scaling. We note that the restriction of the map $\sigma$ to $\Omega_\infty$ is a homomorphism with kernel $\Omega_\infty^T$, witnessing $\Omega_\infty$ as a semi-direct product, $\Omega_\infty^T\rtimes \Omega_\infty^D$.

On the other hand, for $M\notin\Omega_\infty$, the involutory decomposition contains the involution $S$. The action of $S$ is given algebraically by \[ S\tau = \frac{-1}{\tau} = -\frac{\overline{\tau}}{\abs{\tau}^2}. \] Geometrically, $S$ performs a circle inversion about the unit circle (that is, $\tau \mapsto \frac{\tau}{\abs{\tau}^2}$) and a horizontal reflection about the imaginary axis ($\tau \mapsto -\overline{\tau}$) (and these operations may be applied in either order). The involutory decomposition then states that the action of $M = T^{\theta(M)}\Dmat{\rho^2(M)}ST^{-\pi(M)}$ is ``translate by $-\pi(M)$, perform the involution $S$, scale by $\rho^2(M)$, then translate by $\theta(M)$.'' More frequently, we will use the equivalent point of view, ``perform an involution about the circle with center $\pi(M)$ and radius $\rho(M)$ followed by translation by $\theta(M)-\pi(M)$.'' We will make frequent use of the circle about which $M\in\Omega\setminus\Omega_\infty$ performs an involution.

\begin{definition}\label{Def:Arcs}
	Let $M\in\Omega$, then the \emph{arc}, \emph{isometric circle}, or \emph{isometric locus} for $M$ is the set \[ \mc{A}(M) = \{ \tau\in\hh \st \Im M\tau = \sigma(M)\Im\tau \}. \]
	If $\Im M\tau < \sigma(M)\Im\tau$, we say $\tau$ is \emph{above} $\mc{A}(M)$, and if $\Im M\tau > \sigma(M)\Im\tau$, we say $\tau$ is \emph{below} $\mc{A}(M)$.
	
	Furthermore, for any subset $\mc{S}\subset\Omega$, let \[ \mc{A}(\mc{S}) = \{\mc{A}(K) \st K\in\mc{S} \}. \]
\end{definition}

This set goes by many names, but we will most frequently refer to it as simply the `arc for $M$,' as it is called by Ferenbaugh\footnote{Ferenbaugh writes $\mathrm{arc}(c,r)$, with $c=\pi(M)$ and $r=\rho(M)$} \cite{Ferenbaugh}. Katok \cite{Katok} calls this set the isometric circle, and Duncan and Frenkel use isometric locus \cite{DuncanFrenkel}.

\begin{lemma}\label{Lemma:Arcs}
	For any $M\in\Omega$,
	\[ \mc{A}(M) = \{ \tau\in\hh \st \abs{\tau - \pi(M)}^2 = \rho^2(M) \}. \]
\end{lemma}
\begin{proof}
	Let $M = \begin{psmallmatrix} a & b \\ c & d \end{psmallmatrix}$, then if $c = 0$ (that is, $M\in\Omega_\infty$), \[ \Im M\tau = \frac{ad-bc}{\abs{c\tau+d}^2}\Im\tau = \frac{a}{d}\Im\tau = \sigma(M)\Im\tau, \] that is, $\mc{A}(M) = \hh$. Interpreting $\abs{\tau - \infty} = \infty$ as true, this agrees with the statement of the lemma, since $\pi(M) = \rho^2(M) = \infty$.
	
	If $c\neq 0$, then $M\in \Omega\setminus\Omega_\infty$, so $\sigma(M) = 1$, and we have \[ \Im M\tau = \frac{ad-bc}{\abs{c\tau+d}^2}\Im\tau = \frac{\frac{ad-bc}{c^2}}{\abs{\tau - \frac{-d}{c}}^2}\Im\tau = \frac{\rho^2(M)}{\abs{\tau - \pi(M)}^2}\Im\tau. \] Thus, in this case as well, $\mc{A}(M)$ consists of all $\tau\in\hh $ such that $\abs{\tau - \pi(M)}^2 = \rho^2(M)$.
\end{proof}

The next corollary is useful for describing the arc associated to a product of elements of $\Omega_\infty$ and $\Omega\setminus\Omega_\infty$.

\begin{corollary}\label{Corollary:ArcsUnderOmegaInfinity}
	Let $M\in\Omega\setminus\Omega_\infty$ and $H\in\Omega_\infty$. Then $\mc{A}(MH)$ is a half-circle with center $H\pi(M)$ and radius $\sigma(H)^{-1}\rho^2(M)$, while $\mc{A}(HM)$ is a half-circle centered at $\pi(M)$ with radius $\sigma(H)\rho^2(M)$.
\end{corollary}
\begin{proof}
	By Lemma \ref{Lemma:Arcs}, the arc for $M\notin\Omega_\infty$ is centered at $\pi(M)$ with radius $\rho(M)$, both non-infinite. Applying Lemma \ref{Lemma:PiRhoComputations} then gives the desired centers and radii of $\mc{A}(HM)$ and $\mc{A}(MH)$.
\end{proof}

Whenever a group $\Gamma\subset\Omega$ satisfies $\Gamma_\infty\subset\Omega_\infty^T$, it is better to consider arcs as associated to cosets in $\LeftQuotientByStab{\Gamma}$.

\begin{corollary}\label{Corollary:ArcsAreLeftQuotByStab}
	Let $\Gamma\subset\Omega$ be a subgroup with $\Gamma_\infty \subset \Omega_\infty^T$. Then $\mc{A}:\Gamma\to\mc{A}(\Gamma)$ descends to a bijection $\mc{A}:\LeftQuotientByStab{\Gamma}\to\mc{A}(\Gamma)$.
\end{corollary}
\begin{proof}
	By Corollary \ref{Corollary:PiRhoOnCosets}, when $\Gamma_\infty\subset\Omega_\infty^T$, both $\pi$ and $\rho^2$ descend to functions on $\LeftQuotientByStab{\Gamma}$. Since by Lemma \ref{Lemma:Arcs}, $\mc{A}(M)$ depends only on $\pi(M)$ and $\rho^2(M)$, we find $\mc{A}:\LeftQuotientByStab{\Gamma}\to\mc{A}(\Gamma)$ is well defined. Since $\mc{A}$ certainly surjects onto its image, we need only show injectivity. so that $\mc{A}(HK) = \mc{A}(K)$ for any $H\in\Gamma_\infty$.
	
	Now suppose $K_1,K_2\in\Gamma$ with $\mc{A}(K_1)=\mc{A}(K_2)$. If both arcs are all of $\hh$, then $K_1,K_2\in \Gamma_\infty$, so $[K_1]=[K_2]$ in $\LeftQuotientByStab{\Gamma}$. Otherwise, both arcs are half-circles with center $\pi(K_1)=\pi(K_2)\in\qq$ and squared radius $\rho^2(K_1) = \rho^2(K_2)\in\qq^+$. By the involutory decomposition (Theorem \ref{Theorem:InvolutoryDecomposition}) and the identities of Lemma \ref{Lemma:MatrixCommutations}, 
	\begin{align*}\label{Eqn:PiRhoEqualIsSameCoset}
		K_1K_2^{-1} &=
		T^{\theta(K_1)}\Dmat{\rho^2(K_1)}ST^{-\pi(K_1)}\left(T^{\theta(K_2)}\Dmat{\rho^2(K_2)}ST^{-\pi(K_2)}\right)^{-1} \\
		&= T^{\theta(K_1)}T^{-\theta(K_2)},
	\end{align*}
	so that $K_1K_2^{-1}\in\Omega_\infty^T\cap\Gamma = \Gamma_\infty$, and again $[K_1] = [K_2]$ in $\LeftQuotientByStab{\Gamma}$. Thus $\mc{A}:\LeftQuotientByStab{\Gamma}\to\mc{A}(\Gamma)$ is a bijection.
\end{proof}

\subsection{Groups}\label{Subsection:Groups}

We now introduce the groups $\Gamma_0(mh\divides h){+e,f,g,\ldots}$, and give some basic properties. The bulk of this section may be found in Conway and Norton \cite{CN}, with slight notational modifications as noted below.

\begin{definition}
	A natural number $k$ is an \emph{exact divisor} of $m$ (written $k\edivides m$) if $k\divides m$ and $(k, m/k) = 1$, where $(a,b)$ denotes the greatest common divisor of $a$ and $b$. We denote by $\Ex(m)$ the abelian group of exponent two consisting of all exact divisors of $m$ under the group operation $a*b := \frac{ab}{(a,b)^2}$.
\end{definition}

Observe that if $k\in\Ex(m)$ then $\frac{m}{k}\in\Ex(m)$. Groups of exact divisors and their subgroups play important roles in constructing our groups and in the twisted Hecke relations of Monstrous Moonshine (below, \S\ref{Subsection:MonstrousMoonshine}). A few properties we will use repeatedly are given here.

\begin{lemma}\label{Lemma:ExactGroupIntersections}
	For any $m,n\in\nn$, 
	\begin{align*} 
		\Ex(m)\cap\Ex(mn) &= \{ k\in\Ex(m) \st (k,n) = 1 \} \\
		&= \{ k\in\Ex(mn) \st k\text{ divides }m \}.
	\end{align*}
\end{lemma}
\begin{proof}
	Any $k\in\Ex(m)$ divides $mn$, and $(k,\frac{mn}{k}) = (k, \frac{m}{k}n) = (k,n)$ since $(k,\frac{m}{k}) = 1$. On the other hand, any $k\in\Ex(mn)$ that divides $m$ satisfies $(k, \frac{m}{k}) = 1$ since $(k,\frac{mn}{k}) = 1$.
\end{proof}
\begin{lemma}\label{Lemma:SquareFreeIsAllDivisors}
	Let $m\in\nn$. Then $\Ex(m) = \{ k\st k\text{ divides }m \}$ if and only if $m$ is square-free.
\end{lemma}
\begin{proof}
	If $m$ is square-free, then any factorization $m = k\frac{m}{k}$ is into coprime parts, so every divisor $k$ is exact. Otherwise, let $m = ab^2$ with $b > 1$, then $ab\divides m$ but $(ab,\frac{m}{ab}) = (ab,b) = b$, so $\Ex(m)$ does not contain the divisor $ab$.
\end{proof}

\begin{lemma}\label{Lemma:ExactDivisorsContainSquares}
	Let $m,h\in\nn$. If $k\in\Ex(mh^2)$ is such that $h\divides k$, then $h^2\divides k$.
\end{lemma}
\begin{proof}
	Since $h\divides k\in\Ex(mh^2)$, we must have that $(h,\frac{mh^2}{k})$ divides $(k,\frac{mh^2}{k}) = 1$. Then since $1 = (h, \frac{mh^2}{k})$, we have $k = (kh, mh^2) = h^2(\frac{k}{h},m)$ (recall $h\divides k$), so $h^2\divides k$.
\end{proof}
\begin{lemma}\label{Lemma:ExactDivisorUniqueExpression}
	Let $[p:q]\in\pp^1(\qq)$ with $(p,q) = 1$. If $m\divides pq$, then there exists a unique $k\in\Ex(m)$ such that $k\divides p$ and $\frac{m}{k}\divides q$.
\end{lemma}
\begin{proof}
	Let $k = (p,m)$. Then $\frac{m}{k}\divides \frac{p}{k}q$ with $(\frac{p}{k},\frac{m}{k}) = 1$, so $\frac{m}{k}\divides q$, and by coprimality of $p$ and $q$, $(k,\frac{m}{k}) = 1$, proving existence. Suppose $\ell\in\Ex(m)$ is such that $\ell\divides p$ and $\frac{m}{\ell}\divides q$. Since $\ell\divides m$, we have $\ell\divides (p,m) = k$, and $\frac{k}{\ell}\divides p$. But then $\frac{m}{\ell} = \frac{m}{k}\frac{k}{\ell}\divides q$, so $\frac{k}{\ell}\divides q$ as well, implying $\frac{k}{\ell} = 1$, and $k = \ell$ is unique.
\end{proof}

For any group $G$ and any subset $S\subset G$, we denote by $\langle S\rangle$ the subgroup of $G$ generated by $S$. In particular, for any $e,f,g,\ldots\in\Ex(m)$, we denote by $\langle e,f,g,\ldots\rangle$ the subgroup of $\Ex(m)$ generated by $e,f,g,\ldots$.

\begin{definition}\label{Def:GroupsGamma0mhh}
	Let $m,h\in\nn$, and let $\langle e,f,g,\ldots\rangle$ be a subgroup of $\Ex(m)$. Then \[ \Gamma_0(mh\divides h){+e,f,g,\ldots} := \bigsqcup_{k\in\langle e,f,g,\ldots\rangle} \left\{ \begin{psmallmatrix} khw & x \\ mh^2y & khz \end{psmallmatrix} \bigst w,x,y,z\in\zz,\, kwz-\frac{m}{k}xy = 1 \right\}, \] where $\sqcup$ denotes a disjoint union. We define further notation for special cases: when $\langle e,f,g,\ldots\rangle = \Ex(m)$, we write $\Gamma_0(mh\divides h)+$, and when $\langle e,f,g,\ldots \rangle = \{ 1\}$ we write $\Gamma_0(mh\divides h)-$, or simply $\Gamma_0(mh\divides h)$. When $h=1$, we omit the `$h\divides h$' and write simply $\Gamma_0(m)+e,f,g,\ldots$, $\Gamma_0(m)+$, etc.
	
	For any $K\in \Gamma_0(mh\divides h){+e,f,g,\ldots}$, we call a matrix representative $\begin{psmallmatrix} khw & x \\ mh^2y & khz \end{psmallmatrix}$ with integral entries and determinant $k\edivides m$ as given above a \emph{canonical coset representative} for $K$.
\end{definition}

Note by taking $y\geq 0$ and $z > 0$ whenever $y=0$, we may specify a unique canonical coset representative for any $K\in\Gamma_0(mh\divides h){+e,f,g,\ldots}$.

\begin{remark}
	These groups are usually denoted $\Gamma_0(n\divides h){+e,f,g,\ldots}$, with $h\divides n$, see \cite{CN}, \cite{CMS}, and we write $mh$ for $n$. The notation `$h\divides h$' evokes the fact that $\Gamma_0(mh\divides h){+e,f,g,\ldots}$ is conjugate to $\Gamma_0(m){+e,f,g,\ldots}$ by $\Dmat{h}$, that is, \[ \Gamma_0(mh\divides h){+e,f,g,\ldots} = \Dmat{h^{-1}}(\Gamma_0(m){+e,f,g,\ldots})\Dmat{h}. \] Moreover, as we will see, $m$ is more important for our purposes here (see, e.g., Lemma \ref{Lemma:SingleCuspCriterion}).
\end{remark}

The disjoint union in the definition above is a partition into cosets in the quotient of the group $\Gamma_0(mh\divides h){+e,f,g,\ldots}$ by the normal subgroup $\Gamma_0(mh\divides h)$ \cite{CN}, that is, \[ \Gamma_0(mh\divides h){+e,f,g,\ldots}/\Gamma_0(mh\divides h)\isom \langle e,f,g,\ldots \rangle. \] 

\begin{example}\label{Example:GroupsGamma0mhdividesh+}
	We will consider a few particular examples of these groups throughout. First, we have the (projective) modular group \[ \PSL_2(\zz) = \Gamma_0(1) = \Gamma_0(1\divides 1){+} = \left\{\begin{psmallmatrix} w & x \\ y & z \end{psmallmatrix} \bigst wz - xy = 1 \right\}. \] Next, we will consider a so-called \emph{Fricke group}, 
	\begin{align*}
		\Fricke{2} &:= \Gamma_0(2){+} \\
		&= \left\{\begin{psmallmatrix} w & x \\ 2y & z \end{psmallmatrix} \bigst wz - 2xy = 1 \right\} \sqcup  \left\{\begin{psmallmatrix} 2w & x \\ 2y & 2z \end{psmallmatrix} \bigst 2wz - xy = 1 \right\} \\
		&= \Gamma_0(2) \sqcup \begin{psmallmatrix} 0 & -1 \\ 2 & 0 \end{psmallmatrix}\Gamma_0(2).
	\end{align*}
	As noted, this partition is the two cosets of $\Fricke{2} / \Gamma_0(2) \isom \Ex(2)$. More generally, for any $p$ prime, the Fricke group of level $p$ is \[ \Fricke{p} := \Gamma_0(p){+} = \Gamma_0(p)\sqcup \begin{psmallmatrix} 0 & -1 \\ p & 0 \end{psmallmatrix}\Gamma_0(p). \] The choice of the matrix $\begin{psmallmatrix} 0 & -1 \\ p & 0 \end{psmallmatrix}$ is standard, since as a linear fractional transformation, this exchanges $0$ and $\infty$, but any (integral) $\begin{psmallmatrix} pw & x \\ py & pz \end{psmallmatrix}$ with $pwz - xy = 1$ would suffice for constructing $\Fricke{p}$ as an extension of $\Gamma_0(p)$.  
	
	For a slightly more complicated example, we will consider
	\begin{align*}
		\Gamma_0(6){+} 
		&= \bigsqcup_{k\in\{1,2,3,6\}} \left\{ \begin{psmallmatrix} kw & x \\ my & kz \end{psmallmatrix} \bigst kwz-\frac{m}{k}xy = 1 \right\} \\
		&= \Gamma_0(6) \sqcup \begin{psmallmatrix} 2 & -1 \\ 6 & -2 \end{psmallmatrix}\Gamma_0(6) \sqcup \begin{psmallmatrix} 3 & -2 \\ 6 & -3 \end{psmallmatrix}\Gamma_0(6) \sqcup \begin{psmallmatrix} 0 & -1 \\ 6 & 0 \end{psmallmatrix}\Gamma_0(6).
	\end{align*}
	Here, one has four cosets in the quotient $\Gamma_0(6){+}/\Gamma_0(6)$, corresponding to the four elements of $\Ex(6) = \{1, 2, 3, 6 \}$. Note that $\Ex(6)$ is isomorphic to the Klein four-group (every element of $\Ex(m)$ is of order 2), and is generated by any two non-identity elements, so $\Gamma_0(6){+} = \Gamma_0(6){+2,3} = \Gamma_(6){+2,6} = \Gamma_0(6){+3,6}$. That is, beginning from $\Gamma_0(6)$, one may construct any of the three groups $\Gamma_0(6){+2}$, $\Gamma_0(6){+3}$, or $\Gamma_0(6)+{6}$ by adjoining the appropriate \emph{Atkin-Lehner involution} represented by the matrix above with determinant $2$, $3$, or $6$ respectively. Adjoining any further Atkin-Lehner involution gives all of $\Gamma_0(6){+}$.
	
	Lastly, we consider a case involving conjugation, 
	\begin{align*}
		\Gamma_0(3\divides 3) 
		&= \left\{ \begin{psmallmatrix} 3w & x \\ 9y & 3z \end{psmallmatrix} \bigst wz - xy = 1 \right\} \\
		&= \begin{psmallmatrix} 1 & 0 \\ 0 & 3 \end{psmallmatrix}\PSL_2(\zz)\begin{psmallmatrix} 3 & 0 \\ 0 & 1 \end{psmallmatrix}.
	\end{align*}
	The group $\Gamma_0(mh\divides h){+e,f,g,\ldots}$ is always conjugate to $\Gamma_0(m){+e,f,g,\ldots}$ by $\Dmat{h}$ in this way.
\end{example}

Let $\Gamma = \Gamma_0(mh\divides h){+e,f,g,\ldots}$. Observe that if $K\in\Gamma_\infty$, then $K=\begin{psmallmatrix} khw & x \\ 0 & khz\end{psmallmatrix}$ with $kwz = 1$, so $k = 1$, and $w = z = \pm 1$. By Thm. \ref{Theorem:InvolutoryDecomposition}, then, $K = T^{\frac{x}{h}}\Dmat{\frac{w}{z}} = T^{\frac{x}{h}}$. Since for all $x\in\zz$, $\begin{psmallmatrix} h & x \\ 0 & h \end{psmallmatrix} \in \Gamma_\infty$, we find \[ \Gamma_\infty = \langle T^{\frac{1}{h}}\rangle \subset \Omega_\infty^T. \] Thus, Corollaries \ref{Corollary:PiRhoOnCosets} and \ref{Corollary:ArcsAreLeftQuotByStab} apply, that is, for any $[K]\in\LeftQuotientByStab{\Gamma}$, each of $\pi([K])$, $\rho^2([K])$, and $\mc{A}([K])$ are well-defined, and $\mc{A}(\Gamma)$ is in bijection with $\LeftQuotientByStab{\Gamma}$. But in fact, we show that for these groups, $\rho(K)$ depends only on $m$, $h$, and $\pi(K)$, so that $\mc{A}(K)$ is the unique arc with center $\pi(K)$ in $\mc{A}(\Gamma)$.

\begin{definition}
	Let $m,h\in\nn$. Then 
	\begin{align*}
		r^2_{mh\divides h}: \pp^1(\qq) &\to \pp^1(\qq) \\
		\frac{p}{q} &\mapsto \frac{(mhp,q)(hp,q)}{mh^2q^2}.
	\end{align*}
\end{definition}
\begin{lemma}\label{Lemma:RadiusFromPi}
	Let $K = \begin{psmallmatrix} khw & x \\ mh^2y & khz \end{psmallmatrix}\in\Gamma_0(mh\divides h){+e,f,g,\ldots}$. Then \[ \rho^2(K) = r^2_{mh\divides h}\circ\pi(K). \]
\end{lemma}
\begin{proof}
	The element $K$ has determinant $kh^2$, and $(kv,\frac{m}{k}y) = 1$. We compute $\rho^2(K) = \frac{kh^2}{(mh^2y)^2} = \frac{k}{(mhy)^2}$. Since $\pi(K) = \frac{kv}{mhy}$, we have \[ r^2_{mh\divides h}\circ\pi(K) = \frac{(mhkv, mhy)(hkv, mhy)}{mh^2(mhy)^2} = \frac{k(kv,y)(v,\frac{m}{k}y)}{(mhy)^2} = \frac{k}{(mhy)^2}. \]
\end{proof}

\subsection{Additional groups}\label{Subsection:MoreGroups}

In addition to the groups $\Gamma_0(mh\divides h){+e,f,g,\ldots}$, we are interested in groups which are denoted $\Gamma_0(mh\edivides h){+e,f,g,\ldots}$. These are the groups in the set $\mc{S}$ given in the introduction (\ref{Def:SetS}). When it exists, $\Gamma_0(mh\edivides h){+e,f,g,\ldots}$ is the kernel of a surjective homomorphism $\lambda$ from $\Gamma_0(mh\divides h){+e,f,g,\ldots}$ to the $h$th roots of unity.

Conway, McKay, and Sebbar \cite{CMS} completely characterize when $\lambda$ exists, and therefore when the group $\Gamma_0(mh\edivides h){+e,f,g,\ldots}$ exists. In particular, they note that this homomorphism exists whenever $\Gamma_0(mh\divides h){+e,f,g,\ldots}$ appears in the Monstrous Moonshine correspondence, which includes all the groups $\Gamma\in\mc{S}$ (see \S\ref{Subsection:MonstrousMoonshine} below for more on Monstrous Moonshine).

\begin{lemma}\label{Lemma:LambdaProperties}
	Let $\Gamma_0(mh\divides h){+e,f,g,\ldots}$ be a group appearing in the Monstrous Moonshine correspondence. Then there exists a uniquely determined homomorphism $\lambda : \Gamma \to C^\times$, given by \begin{enumerate}\label{List:LambdaExistenceCriteria}
		\item $\lambda(K) = 1$ if $K\in\Gamma_0(mh^2)+$ with $\det K = k$ such that every prime dividing $k$ also divides $m$,
		\item $\lambda(T^{\frac{1}{h}}) = e^{-\frac{2\pi i}{h}}$, and
		\item $\lambda(ST^{-mh}S) = e^{\pm \frac{2\pi i}{h}}$, where the sign is $+$ if $m\in \langle e,f,g,\ldots\rangle$ and $-$ otherwise.
	\end{enumerate}
\end{lemma}
\begin{proof}
	See \cite{CMS}.
\end{proof}

Recall that the stabilizer of infinity in $\Gamma_0(mh\divides h){+e,f,g,\ldots}$ is $\langle T^{\frac{1}{h}}\rangle$, while $\lambda(T^{\frac{r}{h}}) = 1$ if and only if $T^{\frac{r}{h}}\in\langle T\rangle$. That is, since $\Gamma_0(mh\edivides h){+e,f,g,\ldots}$ is the kernel of $\lambda$, the stabilizer of infinity in $\Gamma_0(mh\edivides h){+e,f,g,\ldots}$ is $\langle T\rangle$.

Conway, McKay, and Sebbar's work builds on that of Ferenbaugh \cite{Ferenbaugh}, who determined when a larger class of homomorphisms which includes $\lambda$ exist, particularly in the case that $h=2$. We will let $\lambda$ denote any such homomorphism $\lambda:\Gamma_0(mh\divides h)\to \cc^*$ with, and let $\ker\lambda = \Gamma_0(mh\edivides h){+e,f,g,\ldots}$ in these cases as well (this appears to be somewhat standard practice). We will say more when discussing non-Monstous Moonshine cases, but for now we note this larger class of homomorphisms $\lambda$ still satisfy the second condition above, that is, $\lambda(T^{\frac{1}{h}}) = e^{-\frac{2\pi i}{h}}$. Thus, the stabilizer of infinity of $\Gamma_0(mh\edivides h){+e,f,g,\ldots}$ is still $\langle T\rangle$, even when $\Gamma_0(mh\divides h){+e,f,g,\ldots}$ does not appear in Monstrous Moonshine, and we have \[ \mbox{\Large \( ^{\Gamma_0(mh\divides h){+e,f,g,\ldots}}/_{\Gamma_0(mh\edivides h){+e,f,g,\ldots}} \)} \isom\; \mbox{\Large \( ^{\langle T^{\frac{1}{h}} \rangle}/_{\langle T\rangle} \)} \isom\; \mbox{\Large \( ^{\zz}/_{h\zz}. \)} \]

\begin{example}\label{Example:ConjugateGroup}
	We will consider the group $\Gamma_0(3\edivides 3) \subset \Gamma_0(3\divides 3)$, which is the kernel of a homomorphism $\lambda:\Gamma_0(3\divides 3)\to \langle e^{\frac{2\pi i}{3}}\rangle$, the group of third roots of unity as a subset of $\cc$. We know that $\Gamma_0(3\divides 3) = \Dmat{3}^{-1}\PSL_2(\zz)\Dmat{3}$, where $\PSL_2(\zz)$ is generated by $S$ and $T$ \cite{SERRE}, so that \[ \Gamma_0(3\divides 3) = \langle \Dmat{3}^{-1}S\Dmat{3}, \Dmat{3}^{-1}T\Dmat{3} \rangle = \langle \Dmat{\frac{1}{9}}S, T^{\frac{1}{3}} \rangle, \] using \ref{Lemma:MatrixCommutations}. What values does $\lambda$ take on these generators? From \cite{CMS}, we find that $\lambda:\Gamma_0(3\divides 3)\to \cc^\times$ takes the value $e^{-\frac{2\pi i}{3}}$ on $\begin{psmallmatrix} 1 & \frac{1}{3} \\ 0 & 1 \end{psmallmatrix} = T^{\frac{1}{3}}$. This behavior under translations will play an important role in the theory we develop below. For the other other generator, note that $\Dmat{\frac{1}{9}}S = \begin{psmallmatrix} 0 & -1 \\ 9 & 0 \end{psmallmatrix}$ is present, so $\lambda$ takes the value $e^{+\frac{2\pi i}{3}}$ on $\begin{psmallmatrix} 1 & 0 \\ 3 & 1 \end{psmallmatrix} = T^{\frac{1}{3}}\Dmat{\frac{1}{9}}ST^{\frac{1}{3}}$ (this is the involutory decomposition, Thm. \ref{Theorem:InvolutoryDecomposition}). Since $\lambda$ is a homomorphism, we find $\lambda(\Dmat{\frac{1}{9}}S) = 1$. (Alternately, we could have observed $\Dmat{\frac{1}{9}}S$ is an involution, so $\lambda$ maps this element of $\Gamma_0(3\divides 3)$ to an element with order dividing $2$.)
\end{example}

\subsection{Fundamental domains}\label{Subsection:FundamentalDomains}

Each of the groups $\Gamma = \Gamma_0(mh\divides h){+e,f,g,\ldots}$ is a discrete subgroup of $\PSL_2(\rr)$ commensurate with $\PSL_2(\zz)$ (\cite{CN}). We give a canonical choice of fundamental domain $\mc{D}(\Gamma)$. The construction of this fundamental domain is essentially the same as \cite{SERRE} for $\PSL_2(\zz)$, working coset-by-coset within $\Gamma_0(mh\divides h){+e,f,g,\ldots}/\Gamma_0(mh\divides h)$.

\begin{theorem}\label{Theorem:FundamentalDomains}
	Let $\Gamma = \Gamma_0(mh\divides h){+e,f,g,\ldots}$. Then there exists a fundamental domain for $\Gamma$, denoted $\mc{D}(\Gamma)$, given by choosing the unique point in each orbit satisfying
	\begin{enumerate}
		\item $\Im \tau\geq \Im K\tau$ for all $K\in\Gamma$
		\item $-\frac{1}{2h} < \Re\tau \leq \frac{1}{2h}$, and
		\item If $\tau$, $K\tau$ both satisfy the above two conditions for some $K\in\Gamma$, then $\Re\tau \geq \Re K\tau$.
	\end{enumerate}
\end{theorem}
\begin{proof}
	Let $\tau\in\hh$, and let $[\tau]$ denote the orbit of $\tau$ under $\Gamma$. Observe that if there exists a point $\tau'\in [\tau]$ satisfying all three conditions, this point is unique, having the imaginary part and maximal allowed real part among all points in $[\tau]$. We need only show such a point exists. 
	
	We first show there exists a $K\in\Gamma$ such that $\Im K\tau$ is maximal. The argument here is largely similar to that found in \cite{SERRE} for $\PSL_2(\zz)$, iterated for each $k\in\langle e,f,g,\ldots\rangle$. Recall for any $K = \begin{psmallmatrix} khw & x \\ mh^2y & khz \end{psmallmatrix}\in\Gamma$, where this is a canonical coset representative, so that $\left(kz, \frac{m}{k}y\right) = 1$, $y\geq 0$, $z>0$ when $y=0$, and \[ \Im K\tau = \frac{kh^2}{\abs{mh^2y\tau+khz}^2}\Im\tau = \frac{\rho^2(K)}{\abs{\tau - \pi(K)}^2}\Im\tau. \] For each $k\in\langle e,f,g,\ldots\rangle$, form the set \[ \mc{L}_k = \{ mh^2y\tau + khz \st (kz, \frac{m}{k}y) = 1, y\geq 0, z>0 \text{ when } y=0 \}. \] As a subset of the lattice $(mh^2\tau)\zz + (kh)\zz$, $\mc{S}_k$ has a (not necessarily unique) element of minimal length, say $mh^2y\tau+khz$, and this element is nonzero since $(y,z)=1$. We also note that there are only finitely many choices of pairs $(y,z) = 1$ such that the resulting lattice element $mh^2y\tau+khz$ has minimal length in $\mc{L}_k$. We may therefore define the finite set \[ \mc{T}_k = \{ (y,z)\in\zz^2 \st y,z\text{ as above } \}. \] 
	
	We associate to each $(y,z)\in\mc{T}_k$ an element $[K_{k,y,z}]\in\LeftQuotientByStab{\Gamma}$, by $\pi(K_{k,y,z}) = -\frac{kz}{mhy}$ and $\rho^2(K_{k,y,z}) = \frac{k}{(mhy)^2}$. Since $\Gamma_\infty = \langle T^{\frac{1}{h}}\rangle$, there is a unique coset representative $K_{k,y,z}$ in $[K]$ such that $-\frac{1}{2h} < \Re K_{(k,y,z)}\tau \leq \frac{1}{2h}$. Observe that for any $K' = \begin{psmallmatrix} khw' & x' \\ mh^2y' & khz' \end{psmallmatrix}\in\Gamma$, we have \[ \Im K_{k,y,z} = \frac{kh^2}{\abs{mh^2y\tau + khz}^2}\Im\tau \geq \frac{kh^2}{\abs{mh^2y'\tau + khz'}^2}\Im\tau = \Im K'\tau, \] by minimality of the associated lattice element in $\mc{L}_k$.
	
	Let \[ \mc{K} = \{ K_{k,y,z}\in\Gamma \st k\in\langle e,f,g,\dots\rangle, (y,z)\in\mc{T}_k \}, \] noting $\mc{K}$ is finite since each $\mc{T}_k$ is finite, and let \[ \mc{K}^* = \{ K\in\mc{K} \st \Im K\tau\geq \Im K'\tau \text{ for all }K'\in\mc{K} \}, \] so that for all $K\in\mc{K}^*$, the point $K\tau$ satisfies the first two conditions of the theorem. Moreover, if $K\in\Gamma$ satisfies the first two conditions, then $K$ is associated to an element of $\mc{L}_k$ of minimal length (where, e.g., $\det K = kh^2$ for a canonical coset representative determines $k$), and therefore to some $K\in\mc{K}^*$. That is, $\{ K\tau \st K\in\mc{K}^* \}$ is precisely the set of points in the orbit of $\tau$ satisfying the first two conditions, and since $\mc{K}^*$ is a finite set, there exists a unique point with maximal real part, so $\mc{D}(\Gamma)$ is well-defined.
\end{proof}

\begin{corollary}\label{Corollary:FundamentalDomainsExactGroups}
	Let $\Gamma = \Gamma_0(mh\edivides h){+e,f,g,\ldots}$. Then we may define a fundamental domain $\mc{D}(\Gamma)$,
	given by choosing the unique point in each orbit satisfying
	\begin{enumerate}
		\item $\Im \tau\geq \Im K\tau$ for all $K\in\Gamma$
		\item $-\frac{1}{2} < \Re\tau \leq \frac{1}{2}$, and
		\item $\Re\tau\geq \Re K\tau$ for all $K\in\Gamma$.
	\end{enumerate}
\end{corollary}
\begin{proof}
	Let $\tau\in\hh$, and let $G\in\Gamma' = \Gamma_0(mh\divides h){+e,f,g,\ldots}$ be such that $G\tau\in\mc{D}(\Gamma')$. Since $\Gamma'/\Gamma \isom \langle T^{\frac{1}{h}}\rangle / \langle T\rangle \isom \zz/ h\zz$, we may choose $T^{\frac{r}{h}}\in\Gamma'$ such that $\lambda(T^{\frac{r}{h}}) = \lambda(G)^{-1}$ (ensuring $T^{\frac{r}{h}}G\in\ker\lambda = \Gamma$) and $-\frac{1}{2} < T^{\frac{r}{h}}G\tau \leq \frac{1}{2}$. Thus there exists a point in each orbit in the given region. This point is necessarily unique by having maximal imaginary and (permitted) real parts.
\end{proof}

That is, the fundamental domain for $\Gamma_0(mh\edivides h){+e,f,g,\ldots}$ is $h$ copies of the fundamental domain for $\Gamma_0(mh\divides h){+e,f,g,\ldots}$ laid side to side. We may also relate fundamental domains for $\Gamma_0(mh\divides h){+e,f,g,\ldots}$ and $\Gamma_0(m){+e,f,g,\ldots}$.

\begin{corollary}\label{Corollary:FundamentalDomainConjugates}
	For all $m,h\in\nn$ and $\langle e,f,g,\ldots\rangle\subset\Ex(m)$, \[ \mc{D}(\Gamma_0(mh\divides h){+e,f,g,\ldots}) = \Dmat{h^{-1}}\mc{D}(\Gamma_0(m)+e,f,g,\ldots). \]
\end{corollary}
\begin{proof}
	Let $\mc{R} = \Dmat{h^{-1}}\mc{D}(\Gamma_0(m)+e,f,g,\ldots)$. For any $h\tau\in\mc{D}(\Gamma_0(m)+e,f,g,\ldots)$, $\Im K(h\tau) \leq \Im h\tau$ for all $K\in\Gamma_0(m)+e,f,g,\ldots$. Since $\tau\in\mc{R}$ if and only if $h\tau\in\mc{D}(\Gamma_0(m)+e,f,g,\ldots)$, we find that for any $K\in\Gamma$, \[ \Im \tau = \frac{1}{h}\Im(h\tau) \geq \frac{1}{h}\Im Kh\tau = \Im (\Dmat{h^{-1}}K\Dmat{h})\tau. \] By an entirely similar argument, $\tau\in\mc{R}$ implies $\Re \tau\geq \Re (\Dmat{h^{-1}}K\Dmat{h})\tau$ for any element $K\in\Gamma_0(m){+e,f,g,\ldots}$. Lastly, since $\mc{D}(\Gamma_0(m){+e,f,g,\ldots})$ has real part in the interval $(-\frac{1}{2},\frac{1}{2}]$, we see $\tau\in\mc{R}$ implies $h\tau \in \mc{D}(\Gamma_0(m){+e,f,g,\ldots})$, so that $-\frac{1}{2h} < \Re\tau \leq \frac{1}{2h}$. That is, we have found that $\mc{R} = \mc{D}(\Gamma_0(mh\divides h){+e,f,g,\ldots})$.
\end{proof}

\begin{definition}\label{Def:LowerBoundary}
	The \emph{lower boundary} of $\mc{D}(\Gamma)$ is the set \[ \mc{C}(\Gamma) = \{ \tau\in\mc{D}(\Gamma) \st \tau - i\epsilon \notin\mc{D}(\Gamma)\text {for all }\epsilon > 0 \}. \] 
\end{definition}

Note that we have defined the lower boundary $\mc{C}(\Gamma)$ to include only points in $\mc{D}(\Gamma)$, and not points in $\overline{\mc{D}(\Gamma)}$, although points in $\mc{C}(\Gamma)$ are indeed boundary points of $\mc{D}(\Gamma)$.

\begin{corollary}\label{Corollary:LowerBoundaryIsArcs}
	Let $\tau\in\mc{D}(\Gamma)$. Then $\tau\in\mc{C}(\Gamma)$ if and only if $\tau\in\mc{A}(K)$ for some $K\in\Gamma\setminus\Gamma_\infty$.
\end{corollary}
\begin{proof}
	Observe that $\sigma(K) = 1$ for all $K\in\Gamma$, since $\Gamma_\infty\subset\Omega_\infty^T$. Thus, since $\tau\in\mc{D}(\Gamma)$, we have $\sigma(K)\Im \tau = \Im\tau \geq \Im K\tau$ for all $K\in\Gamma$, that is, $\tau$ is on or above all arcs $\mc{A}(K)$ for $\Gamma$.
	
	Suppose $\tau\notin\mc{C}(\Gamma)$, so there exists some $\epsilon > 0$ such that $\tau - i\epsilon \in\mc{D}(\Gamma)$. Then $\tau - i\epsilon$ is on or above all arcs $\mc{A}(K)$, which are all circles with centers in $\qq$, apart from $\mc{A}(K) = \hh$, when $K\in\Gamma_\infty$. Thus if $\tau\notin\mc{C}(\Gamma)$ then $\tau$ is strictly above $\mc{A}(K)$ for any $K\in\Gamma\setminus\Gamma_\infty$, that is, $\tau\notin\mc{A}(K)$.
	
	On the other hand, if $\tau\in\mc{C}(\Gamma)$, then $\tau\in\mc{D}(\Gamma)$ and so is on or above all arcs $\mc{A}(K)$. But if $\tau$ is strictly above all arcs, then there exists an $\epsilon>0$ such that $\tau - i\epsilon$ is still strictly above all arcs, so that $\Im K(\tau-i\epsilon)<\Im(\tau - i\epsilon)$ for all $K\in\Gamma\setminus\Gamma_\infty$. Since $\Re (\tau-i\epsilon) = \Re \tau$, we conclude that $\tau-i\epsilon\in\mc{D}(\Gamma)$, contradicting that $\tau\in\mc{C}(\Gamma)$. Thus $\tau$ is on some arc $\mc{A}(K)$.
\end{proof}

\begin{example}
	Continuing with our running examples, we have the following fundamental domains for the groups $\Fricke{2}$, $\Gamma_0(6){+}$, and $\Gamma_0(3\divides 3)$. We will also need to make use of the fundamental domains for $\Fricke{3}$ and $\PSL_2(\zz)$.
	\begin{figure}
		\centering
		\includegraphics[width=0.5\linewidth]{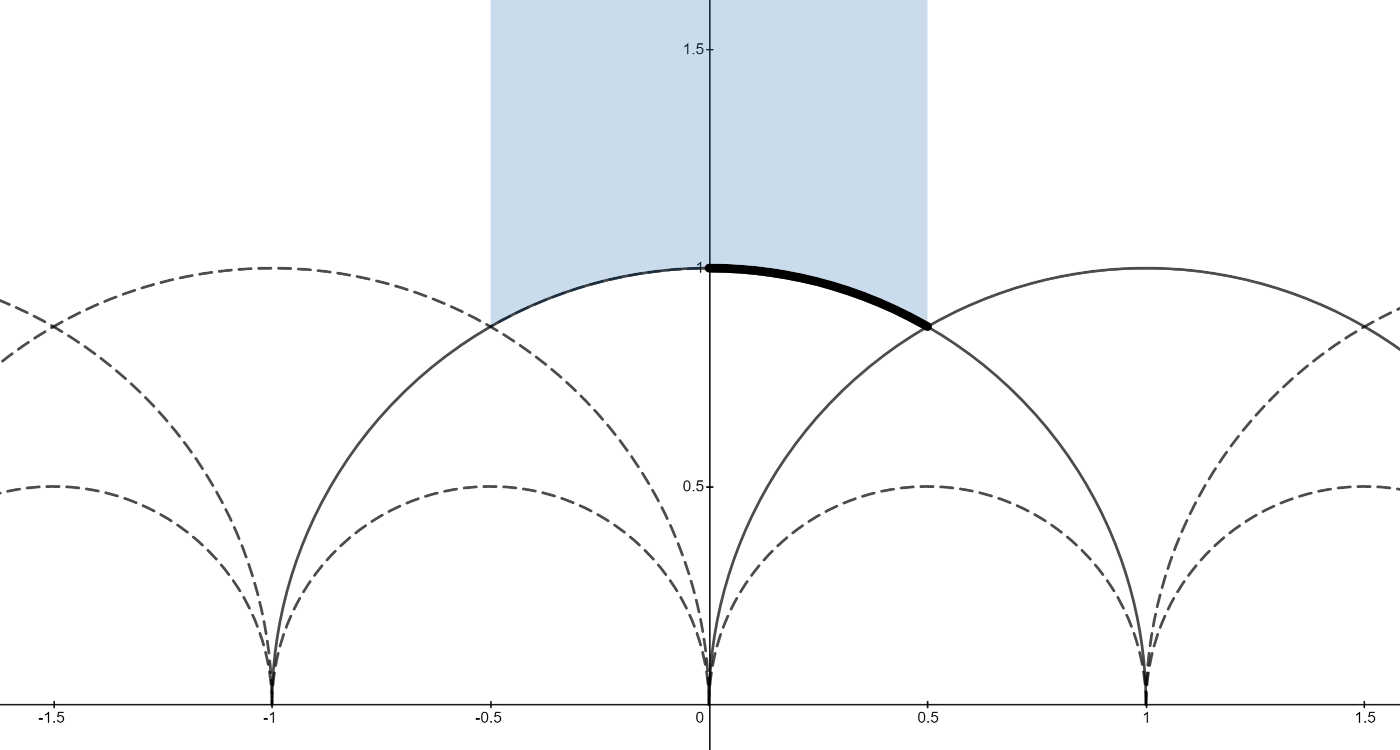}
		\caption{Fundamental Domain $\mc{D}(\PSL_2(\zz))$ for $\PSL_2(\zz)$ (shaded), and the lower arc $\mc{C}(\PSL_2(\zz))$ (in bold).}
		\label{Fig:FundamentalDomainPSL2Z}
	\end{figure}
	\begin{figure}
		\centering
		\includegraphics[width=0.5\linewidth]{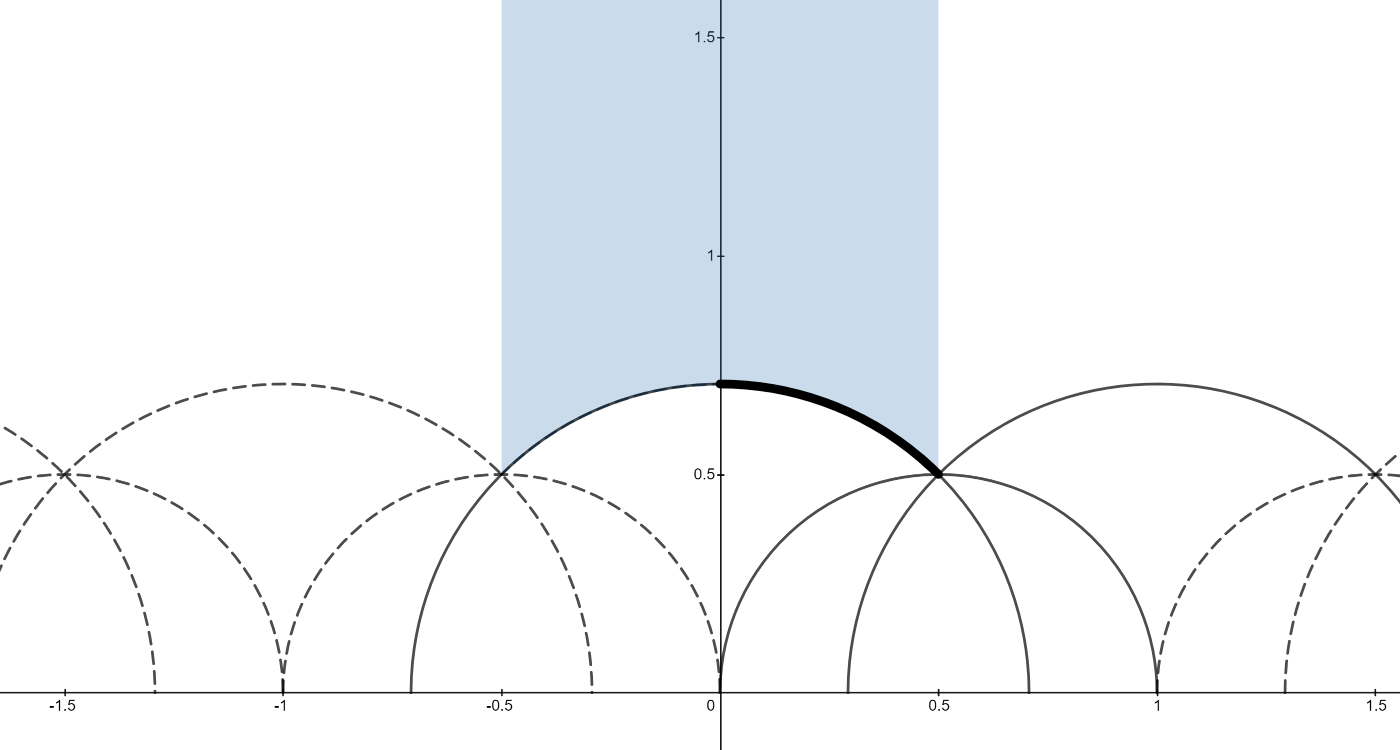}
		\caption{Fundamental Domain $\mc{D}(\Fricke{2})$ for $\Fricke{2}$ (shaded), and the lower arc $\mc{C}(\Fricke{2})$ (in bold).}
		\label{Fig:FundamentalDomainFricke2}
	\end{figure}
	\begin{figure}
		\centering
		\includegraphics[width=0.5\linewidth]{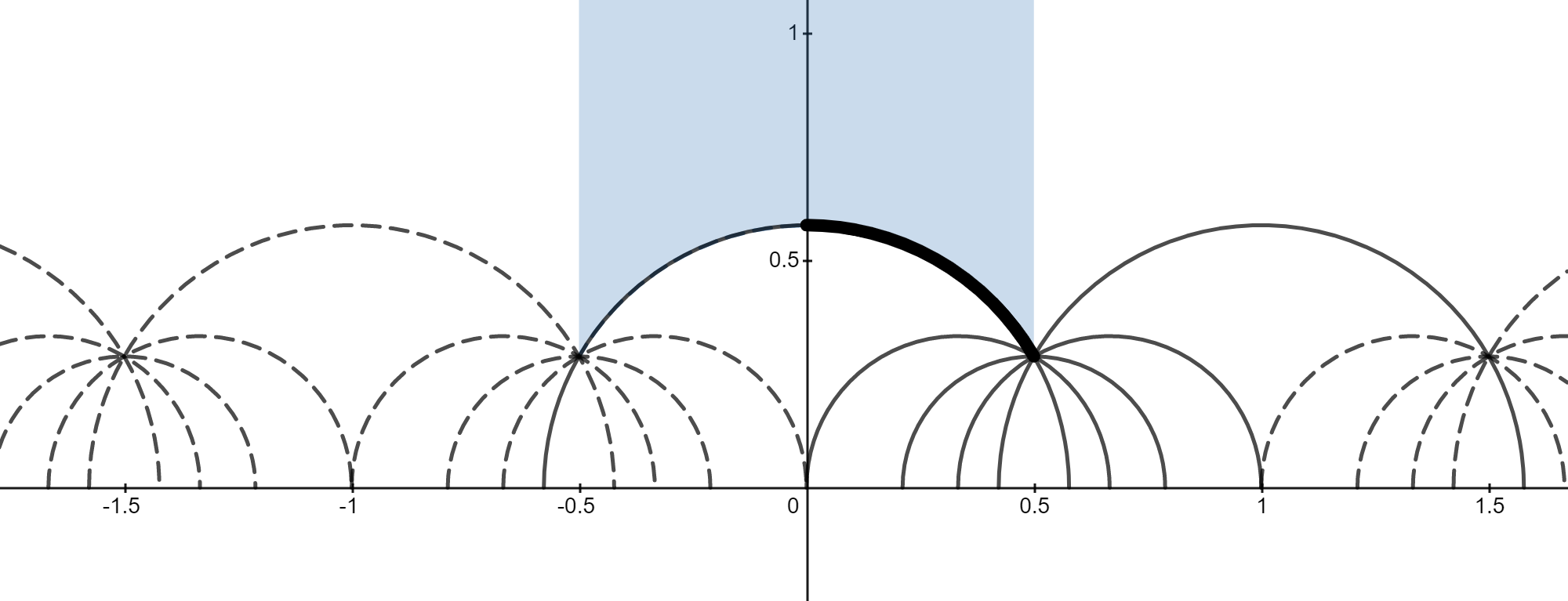}
		\caption{Fundamental Domain $\mc{D}(\Fricke{3})$ for $\Fricke{3}$ (shaded), and the lower arc $\mc{C}(\Fricke{3})$ (in bold).}
		\label{Fig:FundamentalDomainFricke3}
	\end{figure}
	\begin{figure}
		\centering
		\includegraphics[width=0.5\linewidth]{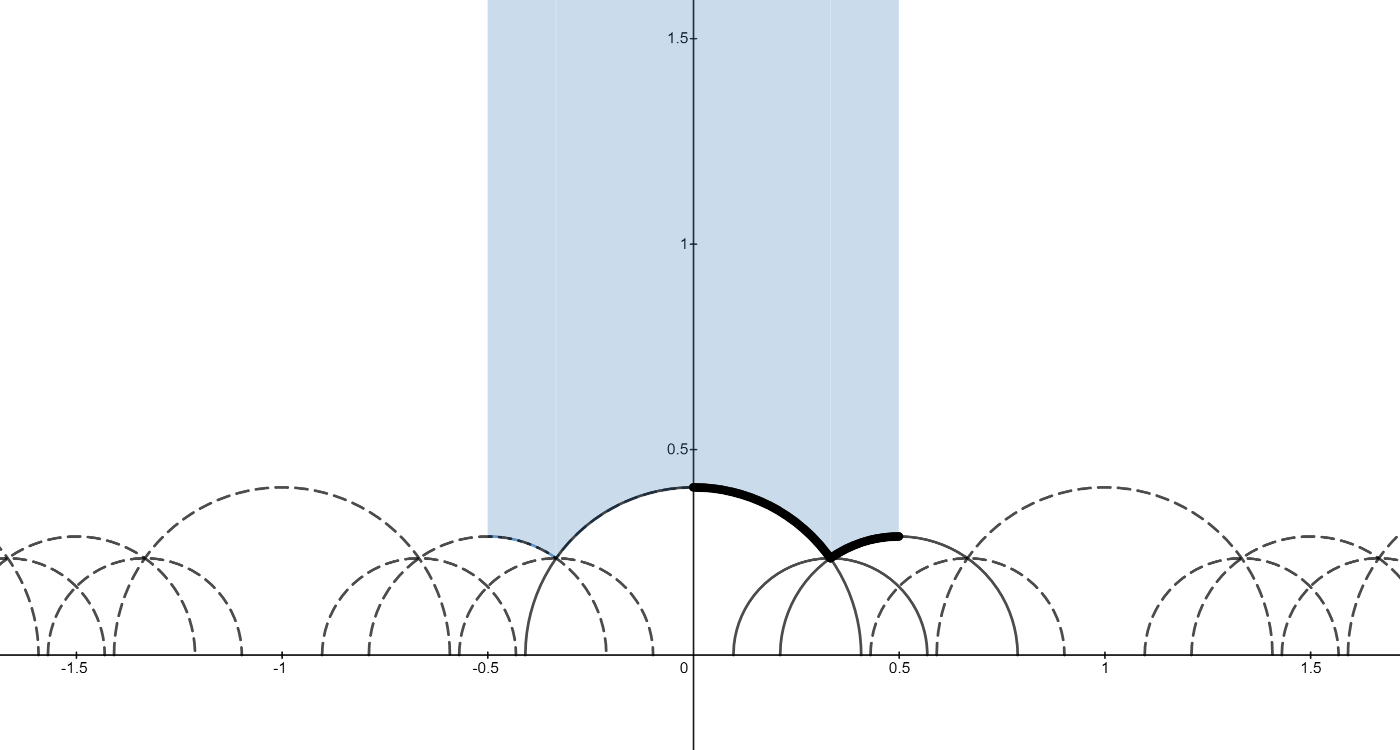}
		\caption{Fundamental Domain $\mc{D}(\Gamma_0(6){+})$ for $\Gamma_0(6){+}$ (shaded), and the lower arcs $\mc{C}(\Gamma_0(6){+})$ (in bold).}
		\label{Fig:FundamentalDomainGamma06+}
	\end{figure}
	\begin{figure}
		\centering
		\includegraphics[width=0.5\linewidth]{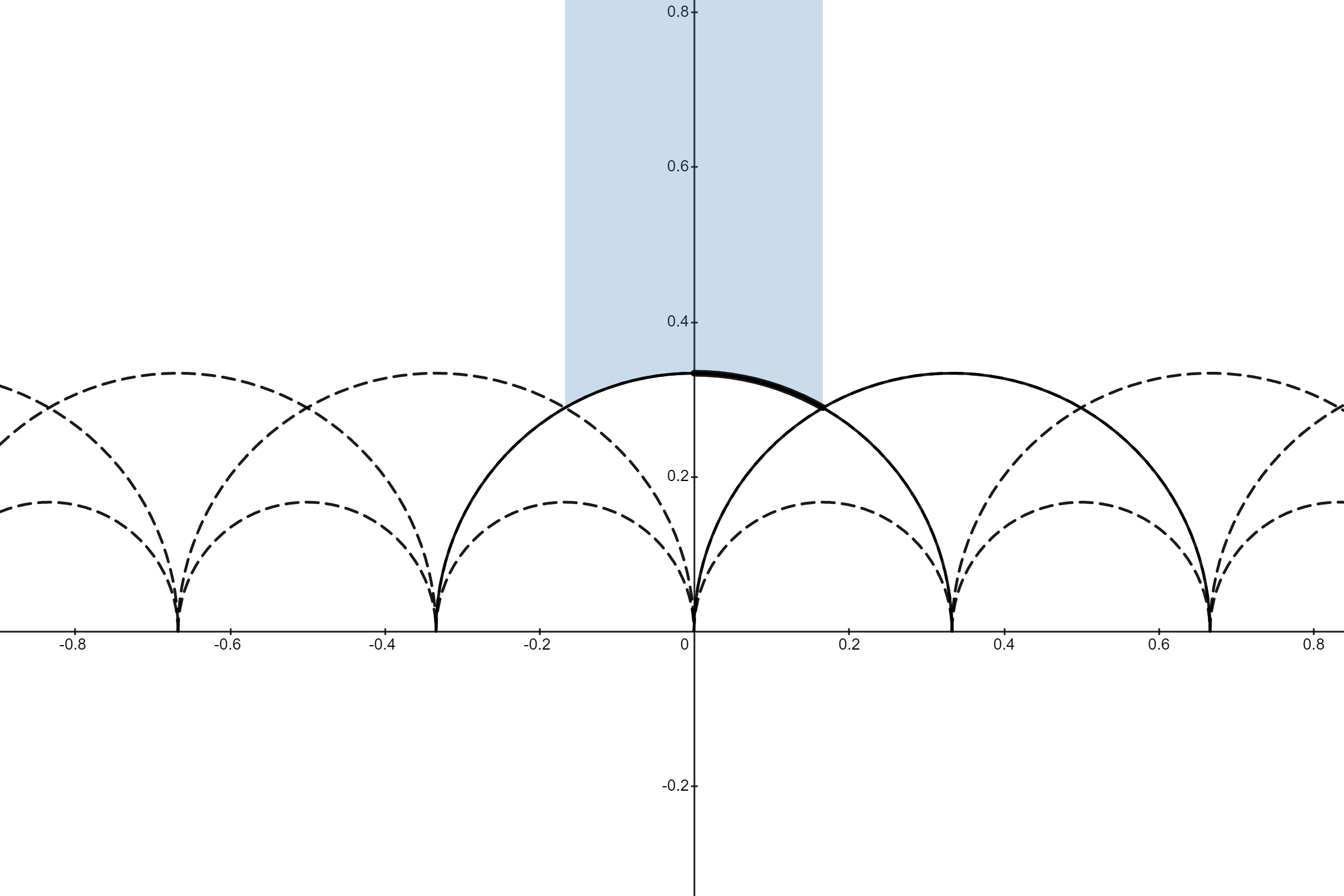}
		\caption{Fundamental Domain $\mc{D}(\Gamma_0(3\divides 3))$ for $\Gamma_0(3\divides 3)$ (blue), and the lower arcs $\mc{C}(\Gamma_0(3\divides 3))$ (in bold). The fundamental domain $\mc{D}(\Gamma_0(3\edivides 3))$ consists of three copies of $\mc{D}(\Gamma_0(3\divides 3))$ `side-by-side.'}
		\label{Fig:FundamentalDomainGamma03bar3}
	\end{figure}
\end{example}

To summarize our results about fundamental domains, let $m,h\in\nn$, and let $\langle e,f,g,\ldots\rangle\subset\Ex(m)$. First, one may construct a fundamental domain for $\Gamma_0(m)+e,f,g,\ldots$ by drawing in all of the arcs $\mc{A}(K)$ for $K\in\Gamma_0(m)+e,f,g,\ldots$. This fundamental domain has real part between $\pm\frac{1}{2}$. One may then produce the fundamental domain $\mc{D}(\Gamma_0(mh\divides h){+e,f,g,\ldots})$ by scaling $\mc{D}(\Gamma_0(m)+e,f,g,\ldots)$ by $\frac{1}{h}$ (Lemma \ref{Corollary:FundamentalDomainConjugates}), which shrinks the total width of the fundamental domain to $\frac{1}{h}$. Finally, one constructs $\mc{D}(\Gamma_0(mh\edivides h){+e,f,g,\ldots})$ by laying copies of $\mc{D}(\Gamma_0(mh\divides h){+e,f,g,\ldots})$ `side-to-side' to construct a fundamental domain which once again has real part between $\pm\frac{1}{2}$.

\subsection{Replicable groups}\label{Subsection:ReplicableGroups}

Consider the set of all groups of the form described above, \[ \mc{R} = \{ \Gamma_0(mh\divides h){+e,f,g,\ldots} \st m,h\in\nn, \langle e,f,g,\ldots\rangle \subset \Ex(m) \}. \] In this section, we examine relationships between these groups. These relationships are mostly known, with varying hypotheses (see \cite{CN} \S 6), though the precise form of Theorem \ref{Theorem:GroupCosets} appears new. Lemma \ref{Lemma:GHKinverseExpression} is a technical result required for our method, and is also new.

We define three types of map on the set $\mc{R}$, which we call conjugation, extension, and replication.

\begin{definition}\label{Def:ReplicationOfGroups}
	Let $\Gamma = \Gamma_0(mh\divides h){+e,f,g,\ldots}\in\mc{R}$, and $a,k,t\in\nn$. We define the \emph{conjugation map} on $\mc{R}$ by \[ \Gamma_t = \begin{psmallmatrix} 1 & 0 \\ 0 & t \end{psmallmatrix}\Gamma\begin{psmallmatrix} t & 0 \\ 0 & 1 \end{psmallmatrix}. \] We also define the \emph{extension map}, \[ \Gamma+k = \begin{cases}
		\Gamma_0(mh\divides h){+e,f,g,k,\ldots} & k\in\Ex(m) \\
		\Gamma_0(mh\divides h){+e,f,g,\ldots} & \text{else}
	\end{cases}, \] and the \emph{replication map}, \[ \Gamma^{(a)} = \Gamma_0\left(\frac{mh}{(mh,a)}\divides \frac{h}{(h,a)}\right){+e',f',g',\ldots} \in \mc{R} \] where $\langle e',f',g' = \langle e,f,g,\ldots\rangle\cap \Ex\left(\frac{m(h,a)}{(mh, a)}\right)$.
\end{definition}

We extend the subscript notation for conjugation to two additional circumstances. First, for any element $K\in\PGL_2^+(\qq)$, we let $K_h$ denote conjugation of $K$ by $\begin{psmallmatrix} h & 0 \\ 0 & 1 \end{psmallmatrix}$, and second, for $\Gamma = \Gamma_0(mh\divides h){+e,f,g,\ldots}$ and any divisor $h'\divides h$, we let $\Gamma_{\frac{1}{h'}} = \Gamma_0(m\frac{h}{h'}\divides \frac{h}{h'}){+e,f,g,\ldots}$. The following describes the extent to which these operations commute with one another.

\begin{lemma}\label{Lemma:OperationCommutation}
	Let $\Gamma\in\mc{R}$ and $a_i,k_i,t_i\in\nn$, for $i=1,2$. Then the following identities hold:
	\begin{enumerate}
		\item $(\Gamma_{t_1})_{t_2} = \Gamma_{t_1t_2}$,
		\item $(\Gamma+k_1)+k_2 = (\Gamma+k_2)+k_1$,
		\item $(\Gamma^{(a_1)})^{(a_2)} = \Gamma^{(a_1a_2)}$,
	\end{enumerate}
	That is, conjugation, extension, and replication each commute with themselves. On the other hand, for any $a,k,t\in\nn$,
	\begin{enumerate}[resume]
		\item $(\Gamma+k)_t = \Gamma_t + k$,
		\item $(\Gamma+k)^{(a)} = \Gamma^{(a)}+k$,
		\item ${\Gamma_h}^{(a)} = {\Gamma^{(\frac{a}{(a,h)})}}_{\frac{h}{(a,h)}}$.
	\end{enumerate}
	In particular, if $(a,h) = 1$, then conjugation and replication commute.
\end{lemma}
\begin{proof}
	Each identity may be checked by writing out the groups in question and comparing the results, using Lemma \ref{Lemma:ExactGroupIntersections} to determine the sets of exact divisors adjoined. Also, recall $\Gamma^{(a)}+k$ strictly extends $\Gamma^{(a)}$ only when $k\in\Ex\left(\frac{m(h,a)}{(mh,a)}\right) = \Ex\left(\frac{m}{(m,\frac{a}{(a,h)})}\right)$.
\end{proof}

We will require one more result relating groups of the form $\Gamma_0(mh\divides h){+e,f,g,\ldots}$. We have not seen this precise statement elsewhere. 

\begin{definition}\label{Def:HeckeSet}
	Let $n\in\nn$, then the \emph{Hecke set of level $n$} is \[ \mc{H}_n := \left\{ \begin{psmallmatrix} a & b \\ 0 & d \end{psmallmatrix} \bigst a,b,d\in\zz, ad = n, 0\leq b< d \right\}. \] For $H = \begin{psmallmatrix} a & b \\ 0 & d	\end{psmallmatrix}\in\mc{H}_n$, we let $\Gamma^{(H)} = \Gamma^{(a)}$ denote the replication-by-$a$ map on $\mc{R}$.
	
	The \emph{reduced Hecke set of level $n$} is \[ \mc{H}_n^* := \{ H = \begin{psmallmatrix} a & b \\ 0 & d \end{psmallmatrix}\in\mc{H}_n \st (a,b,d) = 1 \}. \]
\end{definition}

\begin{definition}\label{Def:Phi_n}
	Let $\Gamma = \Gamma_0(mh\divides h){+e,f,g,\ldots}$, and $(n,h) = 1$. For a canonical coset representative $K = \begin{psmallmatrix} khw & x \\ mh^2y & khz \end{psmallmatrix}$, we define
	\begin{align*}
		\phi_n: \Gamma &\to \mc{H}_n^* \\
		\begin{psmallmatrix} khw & x \\ mh^2y & khz \end{psmallmatrix} &\mapsto \begin{psmallmatrix} a & b \\ 0 & d \end{psmallmatrix},
	\end{align*}
	where $a = (n, \frac{m}{k}y)$, $d = \frac{n}{a}$, and $0\leq b<d$ is such that $b\equiv zh^{-1}\left(\frac{my}{ak}\right)^{-1}\pmod{d}$.
\end{definition}
We show $\phi_n$ is well-defined, and give some basic properties of the map.
\begin{lemma}\label{Lemma:PhinWellDefined}
	Let $\Gamma = \Gamma_(mh\divides h){+e,f,g,\ldots}$, then the function $\phi_n$ is well-defined, and morever descends to a well-defined function on $\LeftQuotientByStab{\Gamma}$. 
	
	Let $h'\in\nn$ with $(n,h')=1$, and $\Gamma_{h'} = \Gamma_0(mhh'\divides hh'){+e,f,g,\ldots}$. Then letting $0\leq b_{h'}< d$ be such that $b_{h'}\equiv b(h')^{-1}\pmod{d}$, the map \[ H = \begin{psmallmatrix} a & b \\ 0 & d \end{psmallmatrix}\in\phi_n(\Gamma)\mapsto H_{h'} = \begin{psmallmatrix} a & b_{h'} \\ 0 & d \end{psmallmatrix}\in\phi_n(\Gamma_{h'}) \] defines a bijection between $\phi_n(\Gamma)$ and $\phi_n(\Gamma_{h'})$. If $m$ is square-free, then $\phi_n(\Gamma) = \phi_n(\Gamma_{h'}) = \mc{H}_n^*$.
\end{lemma}
\begin{proof}
	Since by construction $\frac{my}{ak}$ and $d$ are coprime, and $(d,h)\divides (n,h) = 1$, we find that $\phi_n(K)$ is a well-defined element of $\mc{H}_n$. Choose some $c\in\zz$ such that $c\equiv h^{-1}\left(\frac{my}{ak}\right)^{-1} \pmod{d}$, then since $b$ and $cz$ are congruent modulo $d$, we must have $(b,d) = (cz, d)$. But by construction $c$ is invertible modulo $d$, and thus coprime to $d$, so $(b,d) = (z,d)$. Now, since $K$ is a canonical coset representative, we know $(z, \frac{m}{k}y) = 1$, and since $a\divides \frac{m}{k}y$, this implies $(a,z)=1$. Thus, $(a,b,d) = (a,z,d) = 1$, by associativity of the greatest common divisor.
	
	Note that $\phi_n$ depends only on the lower row of the canonical coset representative for $K$. For any $T^{\frac{r}{h}}\in\Gamma_\infty = \langle T^{\frac{1}{h}}\rangle$, the canonical coset representative $K = \begin{psmallmatrix} khw & x \\ mh^2y & khz \end{psmallmatrix}$ has the same lower row as $K$, so $\phi_n(K) = \phi_n(T^{\frac{r}{h}}K)$. That is, $\phi_n$ is well-defined on $\LeftQuotientByStab{\Gamma}$.
	
	Now, for $K$ as in the definition, $\phi_n(K) = \begin{psmallmatrix} a & b \\ 0 & d \end{psmallmatrix}$ is such that $a$ is independent of $h$ (and therefore $d = \frac{n}{a}$ is as well), and $bh\equiv z\left(\frac{my}{ak}\right)^{-1} \pmod{d}$, where the right hand side is independent of $h$. Thus, when $(n,h')=1$, letting $K_{h'} = \begin{psmallmatrix} k(hh')w & x \\ m(hh')^2y & k(hh')z \end{psmallmatrix} \in \Gamma_{h'}$, we have $\phi_n(K_{h'}) = \begin{psmallmatrix} a & b' \\ 0 & d \end{psmallmatrix}$, where $b'(hh') \equiv bh \pmod{d}$. Therefore, $b' \equiv b(h')^{-1} \pmod{d}$ (since $h'$ is invertible modulo $d$), and since $0\leq b'<d$, we have $b'=b_{h'}$, so the map from $\phi_n(\Gamma)$ to $\phi_n(\Gamma_{h'})$ given in the statement of the Lemma is well-defined. This map is a bijection, again since $h'$ is invertible modulo $d$.
	
	Now, suppose $m$ is square-free, and let $\begin{psmallmatrix} a & b \\ 0 & d \end{psmallmatrix}\in\mc{H}_n^*$, we will find some $K\in\Gamma$ such that $\phi_n(K)=H$, implying $\phi_n(\Gamma) = \mc{H}_n^*$. Since $m$ is square-free, every divisor is exact (Lemma \ref{Lemma:SquareFreeIsAllDivisors}), so define $k\in\Ex(m)$ by $(a,m)=\frac{m}{k}$. Observe that $m\divides ak$, and that since $k$ is an exact divisor, $(a,m)=(a,\frac{m}{k})(a,k)$, so $(a,k)=1$. Define $y=\frac{ak}{m}$, so $a = \frac{m}{k}y$, and $(k,\frac{m}{k}y) = 1$. To find some $z\in\zz$ such that $(kz,\frac{m}{k}y)=1$ and $bh\equiv z\pmod{d}$, let \[ \mc{S} = \left\{ \frac{bh}{(bh,d)}+\frac{d}{(bh,d)}t \bigst t\in\zz \right\}. \] Then $\mc{S}$ contains infinitely primes by Dirichlet's theorem on arithmetic progressions, so we may choose some prime $p\in\mc{S}$ such that $p\ndivides a$. Letting $z=p(bh,d)$, we have $z = bh + dt$, so $z\equiv bh\pmod{d}$. Since $(a,k)=(a,p)=(n,h)=(a,b,d)=1$, and $a=\frac{m}{k}y$, we deduce that \[ (kz,\frac{m}{k}y) = (kz,a) = (a,z) = (a,p(bh,d)) = (a,(bh,d)) = (a,b,d) = 1. \] Let $w,x\in\zz$ be such that $kwz - \frac{m}{k}xy = 1$, then we find $K = \begin{psmallmatrix} khw & x \\ mh^2y & khz \end{psmallmatrix}\in\Gamma$ is such that $(n,\frac{m}{k}y) = (n,a) = a$, and $b\equiv zh^{-1}\left(\frac{my}{ak}\right)^{-1} \equiv zh^{-1}\pmod{d}$, that is, $\phi_n(K)=H$. Since $H\in\mc{H}_n^*$ was arbitrary, $\phi_n(\Gamma) = \mc{H}_n^*$. Finally, note that in this case, for any $h'\in\nn$ with $(n,h')=1$, we know $\phi_n(\Gamma_{h'}) \subset \mc{H}_n^*$ is in bijection with $\phi_n(\Gamma) = \mc{H}_n^*$, so that $\phi_n(\Gamma_{h'}) = \mc{H}_n^*$ as well.
\end{proof}

We will eventually focus exclusively on the case where $m$ is square-free, but for now, we continue in greater generality.

\begin{theorem}\label{Theorem:GroupCosets}
	Let $m,h,n\in\nn$, with $(n,h) = 1$, and let $\Gamma = \Gamma_0(mh\divides h){+e,f,g,\ldots}$. Then \[ \Gamma = \bigsqcup_{H\in\phi_n(\Gamma)} \Gamma \cap \left(\Dmat{n}\Gamma^{(H)}H\right). \] Moreover, for any $H\in\mc{H}_n\setminus\phi_n(\Gamma)$, one has $\Gamma\cap \left(\Dmat{n}\Gamma^{(H)}H\right) = \emptyset$.
\end{theorem}
\begin{proof}
	We first show that the results holds for $h=1$, that is, for the case $\Gamma = \Gamma_0(m){+e,f,g,\ldots}$. Let $K = \begin{psmallmatrix} kw & x \\ my & kz \end{psmallmatrix}\in\Gamma$ be a canonical coset representative. Suppose for some $H = \begin{psmallmatrix} a & b \\ 0 & d \end{psmallmatrix}\in\mc{H}_n$ that $\begin{psmallmatrix} n & 0 \\ 0 & 1 \end{psmallmatrix}KH^{-1} = G \in\Gamma^{(H)}$.  Let $G = \begin{psmallmatrix} gs & t \\ \frac{m}{(a,m)}u & gv \end{psmallmatrix}$ be a canonical coset representative (so $\det G = g$), and compute also that \[ \begin{psmallmatrix} n & 0 \\ 0 & 1 \end{psmallmatrix}KH^{-1} = \begin{psmallmatrix} kdw & ax-bkw \\ \frac{m}{(a,m)}\frac{y(a,m)}{a} & k\frac{az-b\frac{m}{k}y}{n} \end{psmallmatrix}, \] where the matrix representative has determinant $k$, implying that $k\frac{\ell^2}{s^2} = g$, for some $\frac{\ell}{s}\in\qq$ with $(\ell,s)=1$. Scaling appropriately, we must have \[ G = \begin{psmallmatrix} gs & t \\ \frac{m}{(a,m)}u & gv \end{psmallmatrix} = \begin{psmallmatrix} g s\frac{dw}{\ell} & \frac{\ell}{s} ax-bkw \\ \frac{m}{(a,m)}\frac{\ell y(a,m)}{as} & g\frac{s(az-b\frac{m}{k}y)}{\ell n} \end{psmallmatrix}, \] where both matrices are canonical coset representatives.
	
	We now show that we must have $\ell=s=1$, so that $k=g$. Looking at the lower left entry, $\frac{\ell y(a,m)}{as}\in\zz$, so that $as\divides a\ell y$, implying $s\divides y$, since $\ell$ and $s$ are coprime. On the other hand, $k\ell^2 = gs^2$ with $(\ell, s)=1$ implies $s\divides k$. But since $K$ is a canonical coset representative, $kwz - \frac{m}{k}xy = 1$, and so $s\divides (k,y) = 1$, and $s = 1$. Now, since $g = k\ell^2$ is an exact divisor of $\frac{m}{(a,m)}$, we know $k\divides \frac{m}{(a,m)}$, and that $\ell\divides \frac{m}{(a,m)k}$. Looking at the lower right entry above, (recalling $s=1$), \[ \frac{az-b\frac{m}{k}y}{\ell n} = \frac{\frac{a}{(a,m)}z - by\frac{m}{(a,m)k}}{\ell\frac{n}{(a,m)}} \in\zz, \] and since $\ell\divides \frac{m}{(a,m)k}$, this implies $\ell\divides \frac{a}{(a,m)}z$. Since $\ell\divides \frac{m}{(a,m)}$, we find $\ell$ is coprime to $\frac{a}{(a,m)}$, so in fact $\ell\divides z$. But since $\ell\divides\frac{m}{k}$ as well, and $(z,\frac{m}{k})=1$ (again from $kwz - \frac{m}{k}xy = 1$), we find that $\ell = 1$. Thus, if $\begin{psmallmatrix} n & 0 \\ 0 & 1 \end{psmallmatrix}KH^{-1}\in\Gamma^{(H)}$, then $k = g$, and \[ \begin{psmallmatrix} n & 0 \\ 0 & 1 \end{psmallmatrix}KH^{-1} = \begin{psmallmatrix} kdw & ax-bkw \\ \frac{m}{(a,m)}\frac{y(a,m)}{a} & k\frac{az-b\frac{m}{k}y}{n} \end{psmallmatrix} \] is a canonical coset representative of determinant $k$ for some $G\in\Gamma^{(H)}$. 
	
	By inspection the following three criteria are necessary and sufficient for the matrix above to be a canonical coset representative for an element of $\Gamma^{(H)}$: 
	\begin{enumerate}
		\item $k\edivides \frac{m}{(a,m)}$, 
		\item $a\divides y(a,m)$, and 
		\item $n\divides az-b\frac{m}{k}y$.
	\end{enumerate}
	
	By Lemma \ref{Lemma:ExactGroupIntersections}, since $k\in\Ex(m)$, we find $k\in\Ex(\frac{m}{(a,m)})$ if and only if $k\divides\frac{m}{(a,m)}$. Writing $(a,m)= (a,k)(a,\frac{m}{k})$ (using the fact $(k,\frac{m}{k})=1$), we then have $k\divides \frac{k}{(a,k)}\frac{\frac{m}{k}}{(a,\frac{m}{k})}$, so that $k\divides \frac{k}{(a,k)}$, and thus $(a,k)=1$. Since the second criteria above is equivalent to $a\divides my$, and since $(k,\frac{m}{k}y) = 1$, the first two criteria together are equivalent to $a\divides \frac{m}{k}y$.
	
	Now, $a\divides n$, so we must have $a\divides (n,\frac{m}{k}y)$. Let $c = (\frac{n}{a},\frac{my}{ak})$ and write $d = c\tilde{d}$, $\frac{my}{ak} = c\tilde{m}$. Observe that $ac\tilde{d} = n \divides (az - b\frac{m}{k}y)$, so $c \divides (z - b\frac{my}{ak}) = z - bc\tilde{m}$, implying $c\divides z$. But $c\divides \frac{my}{k}$, and $(z,\frac{my}{k}) = 1$, so $c = 1$, and in fact $a = (n,\frac{my}{k})$. The third criterion above is then $d\divides (z - b\frac{my}{ak})$. Moreover, since $ad=n$, we have $(d,\frac{my}{ak}) = 1$, so that this may be written as $b\equiv z\left(\frac{my}{ak}\right)^{-1}\pmod{d}$, demonstrating that $b$ is uniquely determined modulo $d$. Since $0\leq b< d$, this uniquely determines $b\in\zz$. 
	
	An equivalent set of criteria for $\begin{psmallmatrix} n & 0 \\ 0 & 1 \end{psmallmatrix}KH^{-1}\in\Gamma^{(H)}$ to those given above is therefore
	\begin{enumerate}
		\item $a = (n,\frac{m}{k}y)$, 
		\item $d=\frac{n}{a}$, and 
		\item $0\leq b< d$ such that $b\equiv z\left(\frac{my}{ak}\right)^{-1} \pmod{d}$.
	\end{enumerate}
	That is, $\begin{psmallmatrix} n & 0 \\ 0 & 1 \end{psmallmatrix}KH^{-1} \in \Gamma^{(H)}$ if and only if $H = \phi_n(K)$, or, \[ \Gamma = \bigsqcup_{H\in\phi_n(\Gamma)} \Gamma\cap \begin{psmallmatrix} 1 & 0 \\ 0 & n \end{psmallmatrix}\Gamma^{(H)}H, \] when $\Gamma = \Gamma_0(m){+e,f,g,\ldots}$.
	
	Now, let $h\in\nn$ with $(n,h)=1$, and let $\Gamma = \Gamma_0(mh\divides h){+e,f,g,\ldots}$, so $\Gamma_{\frac{1}{h}} = \Gamma_0(m){+e,f,g,\ldots}$. Since $(n,h)=1$, any $a\divides n$ is coprime to $h$, and so by Lemma \ref{Lemma:OperationCommutation}, \[ \Gamma^{(a)} = (\Gamma_{\frac{1}{h}h})^{(a)} = ((\Gamma_{\frac{1}{h}})^{(a)})_h, \] that is, $(\Gamma^{(a)})_{\frac{1}{h}} = (\Gamma_\frac{1}{h})^{(a)}$. Now, we compute that
	\begin{align*}
		\Gamma 
		&= \begin{psmallmatrix} h & 0 \\ 0 & 1 \end{psmallmatrix} \Gamma_{\frac{1}{h}} \begin{psmallmatrix} 1 & 0 \\ 0 & h \end{psmallmatrix} \\
		&= \begin{psmallmatrix} h & 0 \\ 0 & 1 \end{psmallmatrix} \left[ \bigsqcup_{\begin{psmallmatrix} a & b \\ 0 & d \end{psmallmatrix}\in\phi_n(\Gamma_{\frac{1}{h}})} \Gamma_{\frac{1}{h}} \cap \left(\begin{psmallmatrix} 1 & 0 \\ 0 & n \end{psmallmatrix}(\Gamma_{\frac{1}{h}})^{(a)} \begin{psmallmatrix} a & b \\ 0 & d \end{psmallmatrix} \right) \right] \begin{psmallmatrix} 1 & 0 \\ 0 & h \end{psmallmatrix} \\
		&= \bigsqcup_{\begin{psmallmatrix} a & b \\ 0 & d \end{psmallmatrix}\in\phi_n(\Gamma_{\frac{1}{h}})} \Gamma \cap \left(\begin{psmallmatrix} 1 & 0 \\ 0 & n \end{psmallmatrix} \Gamma^{(a)} \begin{psmallmatrix} a & bh \\ 0 & d \end{psmallmatrix} \right) \\
		&= \bigsqcup_{\begin{psmallmatrix} a & b' \\ 0 & d \end{psmallmatrix}\in\phi_n(\Gamma)} \Gamma \cap \left(\begin{psmallmatrix} 1 & 0 \\ 0 & n \end{psmallmatrix} \Gamma^{(a)} \begin{psmallmatrix} a & b' \\ 0 & d \end{psmallmatrix} \right),
	\end{align*}
	where at the last step, we use the fact $\begin{psmallmatrix} 1 & 1 \\ 0 & 1 \end{psmallmatrix}\in \Gamma^{(a)}$, so that for an appropriate choice of $t\in\zz$, we have \[ \begin{psmallmatrix} 1 & t \\ 0 & 1 \end{psmallmatrix}\begin{psmallmatrix} a & bh \\ 0 & d \end{psmallmatrix} = \begin{psmallmatrix} a & bh + dt \\ 0 & d \end{psmallmatrix}, \] where $b' = bh+dt \equiv bh\pmod{d}$ and $0\leq b'<d$. That is, \[ \Gamma = \bigsqcup_{H\in\phi_n(\Gamma)} \Gamma \cap \left(\begin{psmallmatrix} 1 & 0 \\ 0 & n \end{psmallmatrix} \Gamma^{(H)} H \right), \] as desired.
\end{proof}

\begin{remark}
	Suppose that $(n,mh) = 1$, which implies $\Gamma^{(H)} = \Gamma$ for all $H\in\mc{H}_n$, since $(a,mh) = (a,h) = 1$ for any $a\divides n$. Then the above expression becomes \[ \Gamma = \bigsqcup_{H\in\mc{H}_n^*} \Gamma\cap \Dmat{n^{-1}}\Gamma H. \] For $H = \Dmat{n}$, let \[ \Gamma_n = \Gamma\cap \Dmat{n^{-1}}\Gamma\Dmat{n} = \Gamma_0(mnh\divides h){+e,f,g,\ldots}. \] Since the partitions above are invariant under left multiplication by $\Gamma_n$, we conclude that the disjoint union is into cosets in $\Gamma_n\backslash\Gamma$.
	
	In the special case $mh=1$, so $\Gamma = \PSL_2(\zz)$, this is well-known (see, e.g. Lemma 11.11 \cite{Cox}), and is the starting point for the construction of the classical modular equation $\Phi_n(X,Y)$ associated to the $j$-invariant. Analogously, one may construct modular equations whenever the group $\Gamma_0(mh\divides h){+e,f,g,\ldots}$ is genus zero and $(n,mh) = 1$, using the associated Hauptmodul. This has been shown by Cummins and Gannon \cite{CumminsGannon}, who show more generally that functions which satisfy many modular equations must be Hauptmoduls for genus zero groups (or of `trigonometric type' $q^{-1} + \zeta q$ where $\zeta$ is a 24th root of unity).
\end{remark}

Let $n\in\nn$, and $\Gamma = \Gamma_0(mh\divides h){+e,f,g,\ldots}$ with $(n,h)=1$. Given some $K\in\Gamma$, Theorem \ref{Theorem:GroupCosets} tells us precisely when $\Dmat{n}K = GH$ for some $G\in\Gamma^{(H)}$: when $H = \phi_n(K)$. However, we will also need knowledge of a slightly different situation. 

\begin{lemma}\label{Lemma:GHKinverseExpression}
	Let $\Gamma = \Gamma_0(mh\divides h){+e,f,g,\ldots}$, and $(n,h)=1$. Let $H = \begin{psmallmatrix} a & b \\ 0 & d \end{psmallmatrix}\in\mc{H}_n$ and $G = \begin{psmallmatrix} \ell hs & t \\ \frac{m}{(a,m)}h^2u & \ell hv \end{psmallmatrix} \in\Gamma^{(H)}$ a canonical coset representative. Suppose $K = \begin{psmallmatrix} khw & x \\ mh^2y & khz \end{psmallmatrix}\in \Gamma$ is a canonical coset representative such that $\pi(K) = \pi(GH)$.
	
	Then either $G,K\in\Omega_\infty$, and for some $r\in\zz$, \[ GHK^{-1} = T^{\frac{r}{dh}}\Dmat{\frac{a}{d}}, \] or else $G,K\notin\Omega_\infty$, and for some $r\in\zz$, \[ GHK^{-1} = T^{\frac{r}{hN_1N_2}}\Dmat{\frac{n}{N_1^2N_2}}, \] where $N_1 = \frac{au}{(a,m)y}$ and $N_2 = \frac{k}{\ell}$ are integers.
	
	In particular, either $\sigma(GHK^{-1}) = n$, in which case $\Dmat{n}KH^{-1}\in [G]$ in $\LeftQuotientByStab{\Gamma^{(H)}}$ and $H=\phi_n(K)$, or else $\sigma(GHK^{-1})\leq \frac{n}{2}$.
\end{lemma}
\begin{proof}
	First, observe since $(n,h)=1$, that $\Gamma^{(H)} = \Gamma_0(\frac{m}{(a,m)}h\divides h){+e',f',g',\ldots}$, where by Lemma \ref{Lemma:ExactGroupIntersections}, \[ \langle e,f,g,\ldots\rangle \cap\Ex\left(\frac{m}{(a,m)}\right) = \{ \ell' \in\langle e,f,g,\ldots \rangle \st (\ell', (a,m)) = 1\}. \] Since any exact divisor of $m$ divides $m$, we find $(\ell,a,m) = (\ell,a) = 1$ (where the canonical coset representative for $G$ given has determinant $\ell h^2$).
	
	Suppose $G\in\Omega_\infty$, then $GH\in\Omega_\infty$, so $\pi(GH) = \infty$, and thus $\pi(K) = \infty$ by assumption, and $K\in\Omega_\infty$. Then, using the involutory decomposition (Thm. \ref{Theorem:InvolutoryDecomposition})
	\[ GHK^{-1} = T^{\frac{t}{h}} T^{\frac{b}{d}}\Dmat{\frac{a}{d}} T^{-\frac{x}{h}} = T^{\frac{bh+dt-ax}{dh}}\Dmat{\frac{a}{d}}. \] Letting $r = bh+dt-ax$, we have the desired result.
	
	Now suppose $G\notin\Omega_\infty$, so $GH\notin\Omega_\infty$, and we deduce that $K\notin\Omega_\infty$ as well, since $\pi(K) = \pi(GH)$. Expressing $\pi(K)$ and $\pi(GH)$ in terms of the entries of the given matrices, we have \[ \frac{kz}{mhy} = \frac{\ell(dv+bhu\frac{m}{\ell(a,m)})}{\frac{a}{(a,m)}mhu}. \] Clearing denominators and defining $c = dv + bhu\frac{m}{\ell (a,m)}$, we obtain the equation (in $\zz$) \[ kz\left(\frac{a}{(a,m)}u\right) = \ell yc. \] Here $(kz,y)=1$ since $kwz-\frac{m}{k}xy=1$ ($K$ being given as a canonical coset representative), and thus $y\divides \frac{au}{(a,m)}$, and $N_1 \in \zz$. Similarly, we have $(\ell,u)=1$ since $G$ was given as a canonical coset representative, and combined with $(\ell,a)=1$, we find $\ell\divides kz$. Since $(kz, \frac{m}{k}) = 1$, we find $(\ell,\frac{m}{k})=1$, but $\ell\divides m = \frac{m}{k}k$, so we must have $\ell\divides k$, son $N_2\in\zz$ as well. Returning to our $\pi(K)=\pi(GH)$ equation, we find \[ zN_1N_2 = z\frac{k}{\ell}\left(\frac{au}{y(a,m)}\right) = c = dv + bhu\frac{m}{\ell(a,m)}. \] In particular, $zN_1N_2 \equiv dv \pmod{\frac{mu}{\ell (a,m)}}$, and so \begin{equation}\label{Eqn:pirelations} kwzN_1N_2 \equiv kwdv \pmod{\frac{mu}{\ell(a,m)}}. \end{equation}
	
	Using the involutory decomposition, we have
	\begin{align*}
		(GH)(K^{-1}) 
		&= \left(T^{\theta(GH)}\Dmat{\rho^2(GH)}ST^{-\pi(GH)}\right)\left(T^{\pi(K)}S\Dmat{\frac{1}{\rho^2(K)}}T^{-\theta(K)}\right) \\
		&= T^{\theta(G)-\frac{\rho^2(GH)}{\rho^2(K)}\theta(K)}\Dmat{\frac{\rho^2(GH)}{\rho^2(K)}},
	\end{align*}
	since $\theta(GH) = \theta(G)$, and using the identities of Lemma \ref{Lemma:MatrixCommutations}. We compute \[ \rho^2(GH) = \sigma(H^{-1})\rho^2(G) = n\frac{\ell(a,m)^2}{(amhu)^2}. \] Using $r_{mh\divides h}^2$ (Lemma \ref{Lemma:RadiusFromPi}), we may compute
	\begin{align*}
		\rho^2(K)
		&= r_{mh\divides h}^2\circ \pi(K) \\
		&= r_{mh\divides h}^2\circ \pi(GH) \\
		&= \frac{ \left( mh \ell c, mhu\frac{a}{(a,m)}\right)\left( h \ell c, mhu\frac{a}{(a,m)}\right) }{mh^2 \left(mhu\frac{a}{(a,m)}\right)^2} \\
		&= \frac{\ell (a,m)^2}{(amhu)^2} \left( \ell c, \frac{au}{(a,m)}\right)\left( c, \frac{m}{\ell}\frac{au}{(a,m)}\right) \\
		&= \frac{1}{n}\rho^2(GH) M_1M_2,
	\end{align*}
	where $M_1 = \left( \ell c, \frac{au}{(a,m)}\right)$ and $M_2 = \left( c, \frac{m}{\ell} \frac{au}{(a,m)}\right)$. Thus, $\frac{\rho^2(GH)}{\rho^2(K)} = \frac{n}{M_1M_2}$, and so \[ GHK^{-1} = T^{\theta(G) - \frac{n}{M_1M_2}\theta(K)}\Dmat{\frac{n}{M_1M_2}}. \]
	
	Observe that since $(\ell, au) = 1$, in fact $M_1 = \left( c,\frac{au}{(a,m)}\right)$, and this implies $M_1$ divides $M_2$, with $\frac{M_2}{M_1} = \left( \frac{c}{M_1}, \frac{m}{\ell}\frac{au}{M_1(a,m)} \right) = \left(\frac{c}{M_1}, \frac{m}{\ell}\right)$. Recall from $\pi(K) = \pi(GH)$, we deduced (up to multiplication by $y$) that $zN_2\left(\frac{au}{(a,m)}\right) = cy$. We now find $zN_2\left(\frac{au}{M_1(a,m)}\right) = \frac{c}{M_1}y$, and we must have $y\divides \frac{au}{M_1(a,m)}$, since $(zN_2,y) = (\frac{k}{\ell}z,y)\divides (kz,y) = 1$. Moreover, $\frac{c}{M_1}$ is coprime to $\frac{au}{M_1(a,m)}$ by construction, so $\frac{c}{M_1}\divides z N_2$, and thus \[ \frac{zN_2M_1}{c} \frac{N_1}{M_1} = 1. \] We find $M_1 = N_1$, and $\frac{c}{M_1} = N_2z$, so that \[ \frac{M_2}{M_1} = \left(\frac{c}{M_1}, \frac{m}{\ell}\right) = \left( N_2z, N_2\frac{m}{k} \right) = N_2(z,\frac{m}{k}) = N_2. \] That is, \[ \frac{\rho^2(GH)}{\rho^2(K)} = \frac{n}{M_1M_2} = \frac{n}{N_1^2\frac{M_2}{M_1}} = \frac{n}{N_1^2N_2}. \]
	
	Thus far, we have \[ GHK^{-1} = T^{\theta(G) - \frac{n}{N_1^2N_2}\theta(K)}\Dmat{\frac{n}{N_1^2N_2}}, \] so it remains only to show that $hN_1N_2(\theta(G) - \frac{n}{N_1^2N_2}\theta(K)) \in \zz$. We compute
	\begin{align*}
		hN_1N_2\theta(G) - \frac{hn}{N_1}\theta(K)
		&= \frac{N_2N_2 s}{\frac{mu}{\ell(a,m)}} - \frac{ad}{\frac{au}{(a,m)y}}\frac{kw}{my} \\
		&= \frac{N_1N_2s}{\frac{mu}{\ell(a,m)}} - \frac{dkw}{\frac{mu}{(a,m)}} \\
		&= \frac{N_1N_2s - dwN_2}{\frac{mu}{\ell(a,m)}}.
	\end{align*}
	
	Now, observe \[ \frac{\frac{m}{k}yN_1^2N_2}{a} = \frac{mu}{\ell(a,m)}N_1 \in\zz, \] so $a\divides \frac{m}{k}yN_1^2N_2$. Since $kwz\equiv 1 \pmod{\frac{my}{k}}$ is equivalent to $kwzN_1^2N_2 \equiv N_1^2N_2 \pmod{\frac{my}{k}N_1^2N_2}$, a fortiori we have $kwzN_1^2N_2 \equiv N_1^2N_2 \pmod{\frac{mu}{\ell(a,m)}N_1}$, or equivalently, \[ kwzN_1N_2 \equiv N_1N_2 \pmod{\frac{mu}{(\ell(a,m)}}. \] Combining this with \eqref{Eqn:pirelations}, and using the fact $\ell sv\equiv 1 \pmod{\frac{mu}{\ell(a,m)}}$, we find \[ (N_1N_2)s \equiv (kwdv)s  \equiv \frac{k}{\ell}dw (\ell sv) \equiv (N_2)dw \pmod{\frac{mu}{\ell(a,m)}}. \] This proves that $GHK^{-1}$ always has the given form.
	
	Finally, $\sigma(GHK^{-1}) = \frac{a}{d}$ or $\sigma(GHK^{-1}) = \frac{n}{N_1^2N_2}$. Thus either $\sigma(GHK^{-1}) = n$, or else we have $\sigma(GHK^{-1}) \leq \frac{n}{2}$. When $\sigma(GHK^{-1}) = n$ and $G,K\in\Omega_\infty$, we must have $a=n$ and $d=1$, so that $GHK^{-1} = T^{\frac{r}{h}}\Dmat{n}$, and $\Dmat{n}KH^{-1}\in [G]$ in $\LeftQuotientByStab{\Gamma^{(H)}}$. On the other hand, when $\sigma(GHK^{-1}) = n$ and $G,K\notin\Omega_\infty$, we must have $N_1 = N_2 = 1$, so that $GHK^{-1} = T^{\frac{r}{h}}\Dmat{n}$, and again $\Dmat{n}KH^{-1}\in [G]$. Since $\Dmat{n}KH^{-1}\in\Gamma^{(H)}$, we must have $H = \phi_n(K)$, by Theorem \ref{Theorem:GroupCosets}.
\end{proof}

In the above lemma, there is no guarantee that a $K$ satisfying $\pi(K) = \pi(GH)$ exists, generically. However, we can ensure such a $K$ exists by restricting ourselves to considering groups $\Gamma_0(mh\divides h){+}$ with $m$ square-free. Observe that whenever \[ \pi(\Gamma) := \{ \pi(K) \st K\in\Gamma \} = \pp^1(\qq), \] we may guarantee that there exists some $K\in\Gamma$ such that $\pi(K) = \pi(GH)$, no matter what $\pi(GH)\in\pp^(\qq)$ is.

\begin{lemma}[Single Cusp Criterion]\label{Lemma:SingleCuspCriterion}
	Let $\Gamma\in\mc{R}$. Then $\pi(\Gamma) = \pp^1(\qq)$ if and only if $\Gamma = \Gamma_0(mh\divides h){+}$ with $m$ square-free.
\end{lemma}
\begin{proof}
	Let $\Gamma = \Gamma_0(mh\divides h){+e,f,g,\ldots}$, then \[ \pi(\Gamma) = \bigsqcup_{k\in\langle e,f,g,\ldots\rangle} \left\{ \frac{-kz}{mhy} \st (kz, \frac{m}{k}y) = 1 \right\}. \] On the other hand, for any $\frac{p}{hq}\in\pp^q(\qq)$ with $(p,q)=1$, we may uniquely define $c = (m,q)$, and therefore uniquely define $\ell,r\in\zz$ with $(\ell,r)=1$ by $m = c\ell$, $q = cr = \frac{m}{\ell}r$. Then $\frac{p}{hq} = \frac{\ell p}{mhr}$ with $(p,q) = (p, \frac{m}{\ell}r) = 1$. That is, \[ \pp^1(\qq) = \bigsqcup_{\ell\divides m} \left\{ \frac{\ell p}{mhq} \st (p,\frac{m}{\ell}r) = (\ell,r) = 1 \right\}.\] Note that whenever $\ell\edivides m$, that is, when $(\ell,\frac{m}{\ell}) = 1$, we have $(\ell, \frac{m}{\ell}r) = 1$, and therefore $(\ell p,\frac{m}{\ell}r) = 1$. In other words, \[ \pp^1(\qq) = \pi(\Gamma) \sqcup \left[ \bigsqcup_{\ell\divides m, \ell\notin\langle e,f,g,\ldots\rangle} \bigsqcup_{\ell\divides m} \left\{ \frac{\ell p}{mhq} \st (p,\frac{m}{\ell}r) = (\ell,r) = 1 \right\} \right].  \] In particular, $\pi(\Gamma) = \pp^1(\qq)$ if and only if $\langle e,f,g,\ldots\rangle = \{ k \st k \text{ divides } m \}$. Since $\Ex(m) = \{ k \st k \text{ divides } m \}$ if and only if $m$ is square-free, by Lemma \ref{Lemma:SquareFreeIsAllDivisors}, we conclude that $\pi(\Gamma) = \pp^1(\qq)$ if and only if $\Gamma = \Gamma_0(mh\divides h){+}$ with $m$ square-free.
\end{proof}

\section{Functions}\label{Section:Functions}

\subsection{Modular Surfaces}\label{Subsection:ModularSurfaces}

Each of the groups $\Gamma = \Gamma_0(mh\divides h){+e,f,g,\ldots}$ is a discrete group, and we may consider the quotient Riemann surface $Y(\Gamma) = \hh/\Gamma$, which is, properly speaking, an orbifold, having \emph{elliptic points}, which are cone singularities arising from equivalence classes of points $\tau\in\hh$ with non-trivial stabilizer subgroup (\cite{Shimura}). Such a subgroup is always finite cyclic, and $\Gamma$-equivalent points have isomorphic stabilizer subgroups. This surface may be compactified to form a compact Riemann surface $X(\Gamma)$ by adding a finite number of \emph{parabolic points}, called the \emph{cusps} of $\Gamma$, which correspond to equivalence classes of points in $\pp^1(\qq)$ (the stabilizer of any such point is conjugate to $\Gamma_\infty$ by some element of $\Omega$, and therefore infinite cyclic). We remark the cusps are in $\pp^1(\qq)$ because $\Gamma$ is commensurate with $\PSL_2(\zz)$, which has this set of cusps, see Shimura \cite{Shimura} chapter 2 for details\footnote{For Shimura, $\Omega = \PGL_2^+(\qq)$ is the commensurator of $\PSL_2(\zz)$, and therefore of $\Gamma$ as well.}.

Due to Lemma \ref{Lemma:SingleCuspCriterion}, we will be focusing on groups $\Gamma$ for which $\pi(\Gamma) = \pp^1(\qq)$. Note that \[ \pi(\Gamma) = \{ K^{-1}\infty \st K\in\Gamma \} = [\infty], \] where $[\infty]$ denotes the equivalence class of $\infty\in\pp^1(\qq)$, that is, a cusp of $\Gamma$. Lemma \ref{Lemma:SingleCuspCriterion} is called the single cusp criterion, since it gives necessary and sufficient conditions for the modular surface $Y(\Gamma)$ to have one cusp.

The point of view of modular surfaces motivates much of the theory involving modular functions and modular forms, but in practice, we calculate using a fundamental domain. Restricting ourselves to groups with one cusp also ensures that $\mc{D}(\Gamma)$ is bounded away from the real line in $\hh$.

\begin{corollary}\label{Corollary:SingleCuspImaginaryBound}
	Let $\Gamma = \Gamma_0(mh\divides h)+$ with $m$ square-free. Then $\tau\in\mc{D}(\Gamma)$ implies $\Im\tau\geq\frac{\sqrt{3}}{2mh}$.
\end{corollary}
\begin{proof}
	We have that $\Gamma$ has one cusp, by Lemma \ref{Lemma:SingleCuspCriterion}. In particular, we showed that $\pi(\Gamma) = \pp^1(\qq)$, so in particular there exists an arc $\mc{A}(K)$ with center $\frac{t}{mh}$ for every $t\in\zz$. This arc has squared radius \[ r^2_{mh\divides h}\left(\frac{t}{mh}\right) = \frac{(mht,mh)(ht,mh)}{mh^2(mh)^2} = \frac{(t,m)}{(mh)^2} \geq \frac{1}{(mh)^2}. \] So let $\tau\in\mc{D}(\Gamma)$, and let $t\in\zz$ be such that $\frac{t-\frac{1}{2}}{mh} \leq \Re\tau \leq \frac{t+\frac{1}{2}}{2mh}$. Since $\Im\tau\geq\Im K\tau$, we know that $\tau$ is on or above the arc centered at $\frac{t}{mh}$ with radius at least $\frac{1}{mh}$, so \[ \frac{1}{(mh)^2} \leq \abs{\tau-\frac{t}{mh}}^2 = (\Im\tau)^2 + (\Re\tau - \frac{t}{mh})^2 \leq (\Im\tau)^2 + \frac{1}{(2mh)^2}. \] Thus, $\Im\tau \geq \frac{\sqrt{3}}{2mh}$, as desired.
\end{proof}
\begin{figure}
	\centering
	\includegraphics[width=0.5\linewidth]{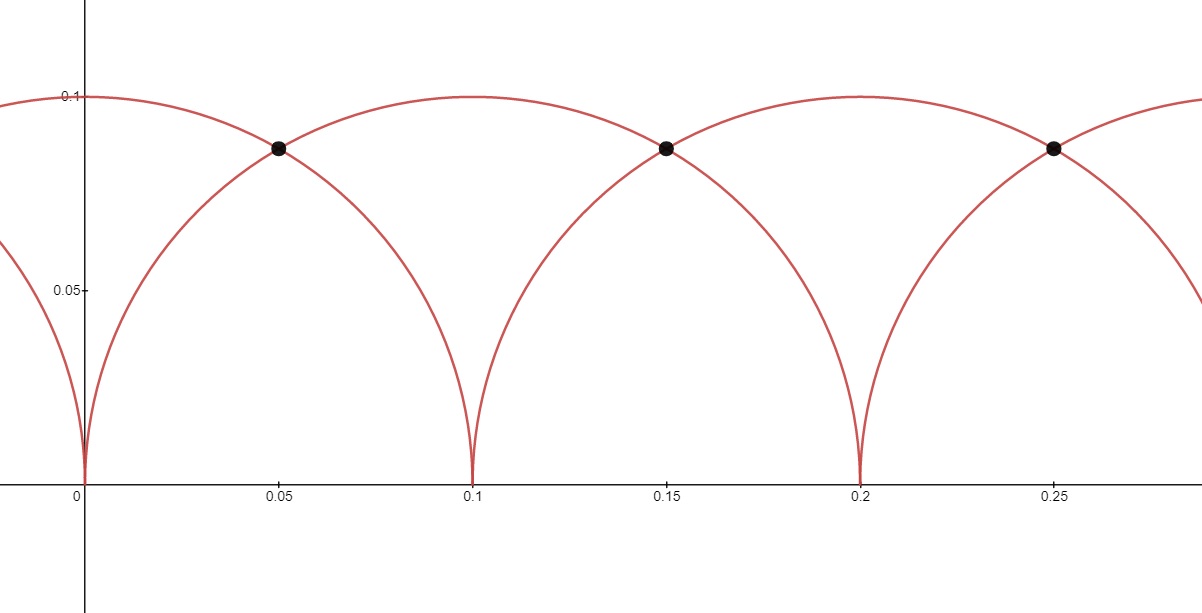}
	\caption{Example of Corollary \ref{Corollary:SingleCuspImaginaryBound}: Here, $m=5, h=2$. There is an arc $\mc{A}(K)$ having radius at least $\frac{1}{10}$ (pictured) centered at each $\frac{t}{10}$ for $t\in\zz$, so if $\tau\in\mc{D}$ then $\tau$ is certainly on or above all the arcs above. The minimal imaginary part of $\tau$ is therefore at least $\frac{\sqrt{3}}{2}\frac{1}{10}$ (when $\tau$ is one of the marked points).}
	\label{Figure:MinimalImaginaryPart}
\end{figure}

We remark that if $\Gamma$ has more than one cusp, then any fundamental domain must approach some point in $\pp^1(\qq)\setminus\{\infty\}$, and is therefore not bounded away from $\rr$, that is, there exists a fundamental domain bounded away from the real line if and only if $\Gamma$ has one cusp.

\begin{definition}[Critical set]\label{Def:CriticalSet}
	Let $\Gamma = \Gamma_0(mh\divides h){+e,f,g,\ldots}$ or $\Gamma_0(mh\edivides h){+e,f,g,\ldots}$. Then the \emph{critical set} for $\Gamma$ is \[ \mc{K} := \{ [K]\in\LeftQuotientByStab{\Gamma} \st \mc{A}(K)\cap\mc{D}(\Gamma) \neq \emptyset \}. \]
\end{definition}
By Corollary \ref{Corollary:ArcsAreLeftQuotByStab}, $\mc{A}(K)$ is independent of coset representative for $[K]$, so the critical set is well-defined. (The fundamental domains are well-defined by Thm. \ref{Theorem:FundamentalDomains} and Corollary \ref{Corollary:FundamentalDomainsExactGroups}.) 

\begin{lemma}\label{Lemma:CriticalSetIsFinite}
	Let $\Gamma = \Gamma_0(mh\divides h)+$ with $m$ square-free. Then $\mc{K}$ is a finite, non-empty set.
\end{lemma}
\begin{proof}
	First, note that for any $K\in\Gamma_\infty$, $\mc{A}(K) = \hh$ certainly intersects the fundamental domain for $\Gamma$, so $[K]\in\mc{K}$, and $\mc{K}$ is non-empty. We may always take $\eye$ as a coset representative for $\Gamma\infty$. 
	
	On the other hand, by Corollary \ref{Corollary:SingleCuspImaginaryBound}, $\mc{D}(\Gamma)$ is bounded away from the real line by at least $\frac{\sqrt{3}}{2mh}$. That is, if $[K]\in\mc{K}$ does not stabilize infinity, then $\rho^2(K) \geq \frac{3}{(2mh)^2}$. But for any canonical coset representative $K = \begin{psmallmatrix} khw & x \\ mh^2y & khz \end{psmallmatrix}\in\Gamma$ with $y > 0$ (since $K\notin\Gamma_\infty$), we may compute $\rho^2(K) = \frac{kh^2}{(mh^2y)^2} = \frac{k}{(mhy)^2}$, and there are only a finite number of (coprime, by $kwz - \frac{m}{k}xy = 1$) pairs $(k,y)$ such that $\frac{k}{(mhy)^2}\geq \left(\frac{\sqrt{3}}{2mh}\right)^2 = \frac{3}{(2mh)^2}$, that is, such that $0 < y < 2\sqrt{\frac{k}{3}}$. For each such pair $(k,y)$, there are only finitely many choices of $z$ (with $z$ satisfying $(kz,\frac{m}{k}y) = 1$) such that $-\frac{1}{2h}-\frac{\sqrt{k}}{mhy} \leq \pi(K) \leq \frac{1}{2h}+\frac{\sqrt{k}}{mhy}$. If $\pi(K)$ is not in the given interval, then the arc centered at $\pi(K)$ with radius $\frac{\sqrt{k}}{mhy}$ certainly does not intersect $\mc{D}(\Gamma)$, which has real part between $\pm\frac{1}{2h}$.
	
	The above gives us a finite set of triples $(k,y,z)$ with $\left(kz,\frac{m}{k}y\right) = 1$. We may identify each $(k,y,z)$ with an element of $[K]$, as in the proof of Thm. \ref{Theorem:FundamentalDomains}, with $\pi(K)=-\frac{kz}{mhy}$ and $\rho^2(K) = \frac{k}{(mhy)^2}$. This gives us an element $[K]$, corresponding to an arc $\mc{A}(K)$ which may intersect $\mc{D}(\Gamma)$. Moreover, by construction, for $K\in\Gamma$, if $\mc{A}(K)$ intersects $\mc{D}(\Gamma)$, then necessarily $[K]$ corresponds to one of the triples $(k,y,z)$ we identified. Thus, there are only finitely many $[K]\in\mc{K}$.
\end{proof}

The proof above also gives an algorithm for computing $\mc{K}$, by listing all possible $(k,y,z)$ as described, and checking whether they intersect $\mc{D}(\Gamma)$. Indeed, if $\mc{D}(\Gamma)$ has not been computed, then this set of all $(k,y,z)$ may be used to compute the fundamental domain, by drawing all these arcs and taking the region above them, and in the appropriate real interval.

\begin{table}[!ht]
	\setlength{\tabcolsep}{18pt}
	\renewcommand{\arraystretch}{1.5}
	\centering
	\begin{tabular}{| l c |}
		\hline
		Group & $\mc{K}$ \\ 
		\hline
		$\PSL_2(\zz)$ & $\left\{ \eye, \begin{psmallmatrix} 0 & -1 \\ 1 & 0 \end{psmallmatrix}, \begin{psmallmatrix} 0 & -1 \\ 1 & -1 \end{psmallmatrix} \right\}$ \\
		$\Fricke{2}$ & $\left\{ \eye, \begin{psmallmatrix} 0 & -1 \\ 2 & 0 \end{psmallmatrix}, \begin{psmallmatrix} 0 & -1 \\ 2 & -2 \end{psmallmatrix}, \begin{psmallmatrix} 1 & -1 \\ 2 & -1 \end{psmallmatrix} \right\}$ \\
		$\Gamma_0(6){+}$ & $\left\{ \eye, \begin{psmallmatrix} 0 & -1 \\ 6 & 0 \end{psmallmatrix}, \begin{psmallmatrix} 3 & -2 \\ 6 & -3 \end{psmallmatrix}, \begin{psmallmatrix} 2 & -1 \\ 6 & -2 \end{psmallmatrix} \right\}$ \\
		$\Gamma_0(3\divides 3)$ & $\left\{ \eye,  \begin{psmallmatrix} 0 & -1 \\ 9 & 0 \end{psmallmatrix}, \begin{psmallmatrix} 0 & -1 \\ 9 & -3 \end{psmallmatrix} \right\}$ \\
		\hline
	\end{tabular}
	\caption{A complete set of coset representatives for the critical set $\mc{K}$ for a few groups $\Gamma_0(mh\divides h){+}$.}
	\label{Table:CriticalSets}
\end{table}
\begin{example}\label{Example:CriticalSets}
	As an example of computing the critical set, let $\Gamma = \Fricke{2} = \Gamma_0(2){+}$, so $m=2$, $h=1$. It may be useful to reference the fundamental domain given in Figure \ref{Fig:FundamentalDomainFricke2}. To collect the set of all triples $(k,y,z)$ in the above proof, we first collect all pairs $(k,y)$ with $k,y$ coprime, with $k\in\Ex(2) = \{1,2\}$, and with $0 < y < 2\sqrt{\frac{k}{3}}$, yielding only the pairs $(1,1)$ and $(2,1)$. Since $y=1$, for $k=1,2$, we consider $z\in \left[-\frac{1+\sqrt{k}}{2},\frac{1+\sqrt{k}}{2}\right]$, so $z\in\{-1,0,1\}$ with $(kz,\frac{2}{k}) = 1$, and we must check the arcs $\mc{A}(K)$ associated to triples \[ (k,y,z) \in \{ (1,1,1), (1,1,-1), (2,1,1), (2,1,0), (2,1,-1) \}, \] via the mapping taking $(k,y,z)$ to the arc with center $\pi(K) = -\frac{kz}{mhy}$ and radius $\rho(K)=\frac{\sqrt{k}}{mhy}$. We find only three intersect the fundamental domain, $(2,1,0)$, $(2,1,-1)$, and $(1,1,-1)$, which are associated to arcs $\mc{A}(K)$, where $K$ may be given by $\begin{psmallmatrix} 0 & -1 \\ 2 & 0 \end{psmallmatrix}, \begin{psmallmatrix} 0 & -1 \\ 2 & -2 \end{psmallmatrix}, \begin{psmallmatrix} 1 & -1 \\ 2 & -1 \end{psmallmatrix}$, respectively. (Here, the lower left entry of each matrix is $mh^2y = 2$ and the lower right entry is $kz$.) We remark that the remaining two arcs intersect the closure of the fundamental domain for $\Fricke{2}$, at the point $\tau = -\frac{1}{2}+\frac{i}{2}$, and a different choice of lower boundary would result in a slightly different critical set.
	
	Analogous calculations give the critical sets for $\Gamma_0(6){+}$ and $\Gamma_0(3\divides 3)$. Note that as one should expect, the cricial set for $\Gamma_0(3\divides 3)$ is conjugate to the critical set for $\PSL_2(\zz)$ by $\Dmat{3} = \begin{psmallmatrix} 3 & 0 \\ 0 & 1 \end{psmallmatrix}$, since the groups themselves are conjugate by the same linear fractional transformation.
\end{example}

\begin{remark}
	The critical set need not be identified at this point. It will be identified as a byproduct of a later calculation which uses a larger collection of triples $(k,y,z)$, see Example \ref{Example:NcComputation} in \S\ref{Subsection:CompletingApproximation} below.
\end{remark}

\subsection{Modular Functions}\label{Subsection:ModularFunctions}

Recall that a character is a homomorphism $\chi:\Gamma\to\cc^*$, taking values in the roots of unity.

\begin{definition}
	Let $\Gamma = \Gamma_0(mh\divides h){+e,f,g,\ldots}$, and let $\chi:\Gamma\to\cc^\times$ be a homomorphism taking values in the $h$th roots of unity. Then a \emph{modular function} for $\Gamma$ with character $\chi$ is a function $f:\hh\to\cc$ such that
	\begin{enumerate}
		\item $f(\tau) = \chi(K)f(K\tau)$ for all $K\in\Gamma$,
		\item $f(\tau)$ is meromorphic on $\hh$, and 
		\item $f(\tau)$ satisfies certain bounded growth conditions as $\tau\to\qq\cup\{\infty\}$.
	\end{enumerate}
	If moreover $f(\tau)$ is holomorphic on $\hh$, we say $f$ is a \emph{weakly-holomorphic} modular function for $\Gamma$ with character $\chi$. If the character $\chi$ is trivial, we say simply that $f$ is a (weakly-holomorphic) modular function for $\Gamma$.
\end{definition}
We will not give the precise growth conditions here. Rather, we note that the first two parts of the definition of a modular function ensure that $f$ descends to a well-defined meromorphic function on the quotient Riemann surface $\lmod{\ker\chi}{\hh}$, and the growth condition ensures this function is also meromorphic after compactification by adjoining points for each cusp \cite{SERRE},\cite{Shimura}. Note that a modular function with character $\chi$ for $\Gamma$ is also a modular function (with trivial character) for the subgroup $\ker\chi$.

The set of all modular functions (with trivial character) for $\Gamma$ is a field under the usual pointwise operations, and the set of all weakly-holomorphic modular functions for $\Gamma$, denoted $\mc{M}_0^!(\Gamma)$, is an algebra under the usual operations of pointwise addition, subtraction, multiplication, and division (which is not allowed among weakly-holomorphic modular functions, as this would in general introduce poles in the upper half-plane). We will be particularly interested in the case where the algebra of weakly-holomorphic modular functions for $\Gamma$ is generated by a single element, so is isomorphic to $\cc[X]$, a polynomial algebra in a single variable.

The modular functions we consider will be modular functions for groups $\Gamma_0(mh\divides h){+e,f,g,\ldots}$ with character $\lambda$, as in \S\ref{Subsection:MoreGroups} above. That is, we consider modular functions with trivial character for $\Gamma_0(mh\edivides h){+e,f,g,\ldots}$.

\subsection{Genus Zero Groups, Hauptmoduls}\label{Subsection:GenusZeroGroups}

As noted above, a discrete group $\Gamma\subset\Omega$ commensurable with $\PSL_2(\zz)$ defines a Riemann surface $Y(\Gamma) = \lmod{\Gamma}{\hh}$, which may be compactified by the addition of a finite number of cusps, corresponding to the orbits of $\pp^1(\qq)$ under $\Gamma$. The genus of this compact Riemann surface $X(\Gamma) = \widehat{\lmod{\Gamma}{\hh}}$ is, informally, the number of `holes' in this surface, so that a genus zero surface is topologically a sphere, a genus one surface a torus, and so on. There are only finitely many genus zero groups of the form $\Gamma_0(mh\divides h){+e,f,g,\ldots}$, and a complete list may be found in \cite{Ferenbaugh}.

By pulling back to $\hh$, the meromorphic functions on $X(\Gamma)$ may be identified with the field of modular functions for $\Gamma$ (with trivial character). Essentially by Liouville's theorem, the only holomorphic functions on $X(\Gamma)$ are constant, while functions which have their poles confined to the cusps of $X(\Gamma)$ correspond to weakly-holomorphic modular functions for $\Gamma$.

Importantly for our purposes, the field of meromorphic functions on a genus zero surface $X(\Gamma)$ is generated by a single function, which witnesses a biholomorphic equivalence between $X(\Gamma)$ and the Riemann sphere $\pp^1(\cc) = \cc\cup\{\infty\}$. Any such function $f$ is called a \emph{Hauptmodul}. It is common to choose as Hauptmodul a function $f$ mapping the cusp $[i\infty]\in X(\Gamma)$ to $\infty\in\pp^1(\cc)$, so that $f$ is also a generator of the algebra of weakly-holomorphic functions $\mc{M}_0^!(\Gamma)$. Such a choice of Hauptmodul is always possible: if $g$ is a Hauptmodul for $\Gamma$ such that $g([i\infty]) = c\in\cc$, then $f = \frac{-1}{g - c} = \begin{psmallmatrix} 0 & -1 \\ 1 & -c \end{psmallmatrix}g$ is a Hauptmodul for $\Gamma$ (since $\PSL_2(\cc)$ acts bijectively on $\pp^1(\cc)$) such that $f([i\infty]) = \infty$. This choice of Hauptmodul is still not unique, since any linear function of $f$ is a Hauptmodul mapping $[i\infty]$ to $\infty$. 

\begin{lemma}\label{Lemma:WeaklyHolomorphicModularFunctionsPolynomialAlgebra}
	Let $\Gamma\subset\Omega$ be a genus zero group commensurable with $\PSL_2(\zz)$. Then $\mc{M}_0^!(\Gamma) \isom \cc[X]$ if and only if $\Gamma$ has one cusp.
\end{lemma}
\begin{proof}
	Let $f$ be a Hauptmodul for $\Gamma$ taking $[i\infty]$ to $\infty$, noting the field of modular functions is $\cc(f)$. Suppose $\Gamma$ has one cusp, and note since $f$ maps $Y(\Gamma)$ to $\cc$ surjectively, a non-constant polynomial in $f$ necessarily has zeros in $Y(\Gamma)$. Let $g\in\mc{M}_0^!(\Gamma)$ and write $g = \frac{F(f)}{G(f)}\in\cc(f)$. Since $g$ has no poles in $Y(\Gamma)$, $G(f)$ must be constant, so $g = F(f)$, and $\mc{M}_0^!(\Gamma) = \cc[f]$. Now suppose $\Gamma$ has more than one cusp, say $[r]\neq [i\infty]$, for some $r\in\qq$, and suppose for a contradiction that $\mc{M}_0^!(\Gamma) = \cc[h]$ for some weakly-holomorphic modular function $h$. Since $f\in\mc{M}_0^!(\Gamma)$, $f$ is a polynomial in $h$, while $h$ is a rational function of $f$, so in fact $f$ is an invertible (thus linear) polynomial in $h$, so that $\mc{M}_0^!(\Gamma) = \cc[f]$. But now let $g = \frac{-1}{f-f([r])}$. As a rational function of $f$, $g$ is a modular function for $\Gamma$. Moreover, since $f(\tau)-f([r]) = 0$ if and only if $[\tau]=[r]$ (by bijectivity of $f:X(\Gamma)\to\pp^1(\cc)$), $g$ has a pole only at $[r]$, a cusp of $\Gamma$, so $g\in\mc{M}_0^!(\Gamma)$. But $g$ is non-constant, and any non-constant polynomial in $f$ has a pole at $[i\infty]$, so $g\notin\cc[f]$, a contradiction, implying $\mc{M}_0^!(\Gamma) \supsetneq \cc[f]$ whenever $\Gamma$ has more than one cusp.
\end{proof}

\begin{example}
	Let $\Gamma = \PSL_2(\zz) = \Gamma_0(1)$, then a common choice of such a Hauptmodul is the $j$-invariant \cite{Cox}, \[ j(\tau) = q^{-1} + 744 + 196884q + 21493760q^2 + 864299970q^3 + \ldots, \] which is a weakly-holomorphic modular function for $\PSL_2(\zz)$. Here, $q = e^{2\pi i\tau}$, so that the series above (called the $q$-expansion of $j$) defines a function on the unit disk centered at $q=0$. Since $q\to 0$ as $\tau\to i\infty$, the $q$-series is interpreted as the Laurent series at the cusp corresponding to infinity on $X(\PSL_2(\zz))$. Note that $\PSL_2(\zz)$ has one cusp, by the single cusp criterion \ref{Lemma:SingleCuspCriterion}. Explicitly, for any $-\frac{d}{c}\in\qq$ with $(c,d)=1$, we may choose $a,b\in\zz$ such that $ad-bc=1$, and $M = \begin{psmallmatrix} a & b \\ c & d \end{psmallmatrix}\in\PSL_2(\zz)$ maps $\pi(M) = -\frac{d}{c}$ to $i\infty$, so that $[i\infty] = \pp^1(\qq)$. By Lemma \ref{Lemma:WeaklyHolomorphicModularFunctionsPolynomialAlgebra} above, the algebra of weakly-holomorphic modular functions is generated by $j$.
\end{example}

Note that $q = e^{2\pi i\tau} = e^{2\pi i(\tau+k)}$ for any $k\in\zz$, but since $\PSL_2(\zz)$ contains $T^k$ for any $k\in\zz$, the passage to the $q$-series above is well-defined. More generally if the stabilizer of $\infty$ in $\Gamma$, is $\langle T^{\frac{1}{h}}\rangle$ then modular functions for $\Gamma$ may be written as functions of $q^{h}$.

\begin{example}
	The function \[ j(3\tau) = q^3 + 744 + 196884q^3 + \cdots \] is a modular function for $\Gamma_0(3\divides 3)$. It is a series in $q^3$. The function \[ T_{3C} = \sqrt[3]{j(3\tau)} = q^{-1} + 248q^2 + 4124q^5 + \cdots \] is a modular function for $\Gamma_0(3\divides 3)$ with character $\lambda$, and so a modular function (with trivial character) for $\Gamma_0(3\edivides 3) = \ker\lambda$. Since $\Gamma_0(3\divides 3)$ is conjugate to $\PSL_2(\zz)$ by $\Dmat{3}\in\Omega$, $\Gamma_0(3\divides 3)$ has one cusp just like $\PSL_2(\zz)$, or one may use the single cusp criterion again.
\end{example}

Observe the $j$-invariant has additional property that it is a Hauptmodul carrying the cusp infinity on $X(\PSL_2(\zz))$ to the point at infinity on the Riemann sphere, but it is not unique in having this property. We will instead prefer to use the function \[ T_{1A}(\tau) = j(\tau) - 744 = q^{-1} + 196884q + \cdots, \] which is the `normalized Hauptmodul' for $\PSL_2(\zz)$. (We will explain the naming convention `$T_{1A}$' shortly, in Subsection \S\ref{Subsection:MonstrousMoonshine}). 

\begin{definition}\label{Def:NormalizedHauptmodul}
	Let $\Gamma = \Gamma_0(mh\divides h){+e,f,g,\ldots}$ be a genus zero group, then the \emph{normalized Hauptmodul} for $\Gamma$ is the unique weakly-holomorphic modular function $f(\tau)$ (with trivial character) for $\Gamma$ which is holomorphic at all non-infinite cusps, and having $q$-expansion \[ f(\tau) = q^{-h} + O(q^h), \] where $O(q)$ is big-O notation. Similarly, the normalized Hauptmodul for a group of the form $\Gamma_0(mh\edivides h){+e,f,g,\ldots}$ is the unique weakly holomorphic modular function for $\Gamma_0(mh\edivides h)+e,f,g,\dots$ which is holomorphic at all non-infinite cusps and has $q$-expansion $q^{-1} + O(q)$.
\end{definition}

Note above that $T_{3C}(q) = q^{-1} + O(q^{2})$, not simply $q^{-1} + O(q)$, which follows from the fact $T_{3C}^3(q) = j(q^3) = q^{-3} + 744 + O(q^3)$. We will use this to locate zeros of certain modular functions in \S\ref{Subsection:Zeros:ConjugateGroups}.

Combining the single cusp criterion for groups $\Gamma = \Gamma_0(mh\divides h){+e,f,g,\ldots}$ (Lemma \ref{Lemma:SingleCuspCriterion}) with Lemma \ref{Lemma:WeaklyHolomorphicModularFunctionsPolynomialAlgebra} above, we observe that $\mc{M}_0^!(\Gamma) \isom \cc[X]$ if and only if $m$ is square-free and $\langle e,f,g,\ldots \rangle = \Ex(m)$. Let $f$ be the normalized Hauptmodul for $\Gamma$, and consider $\cc[f] = \mc{M}_0^!(\Gamma)$ as a vector space. We may choose a basis by choosing a polynomial $F_n(f)$ of degree $n$ for each $n\geq 0$, and $F_n(f) = f^n$ is one such choice. However, we will be more interested in a different choice of basis, the \emph{Faber polynomials} for $f$, defined below.

\subsection{Replicable functions}\label{Subsection:Replication}

We have defined replication for groups in Subection \S\ref{Subsection:ReplicableGroups}. We now define replication for functions. We provide one of several equivalent definitions \cite{CN}.

\begin{definition}[Faber Polynomials]\label{Def:FaberPolynomial}
	Let $f(q) = q^{-1}+\sum_{n\geq 0} c_nq^n$ be a formal $q$-series with coefficients in $\cc$. Then the $n$th Faber polynomial for $f(q)$ is the unique polynomial $F_{n,f}(X)\in\cc[X]$ such that \[ F_{n,f}(q) := F_{n,f}(f(q)) = q^{-n}+O(q), \] where $O(q)$ is big-O notation. 
\end{definition}

Observe $F_{n,f}(X)$ is unique, since if $G(f(q)) = q^{-n} + O(q)$ then $F_{n,f}-G$ is a polynomial such that $(F_{n,f}-G)(f(q)) = O(q)$, and must therefore be identically the zero polynomial. We will be particularly interested in the case that $f(q)$ is a modular function for $\Gamma$. In particular, if $f$ is the normalized Hauptmodul for a genus zero group $\Gamma_0(mh\divides h){+}$ with $m$ square-free, then the set $\{ F_{n,f}(f) = q^{-n}+O(q) \st n\in \nn\cup\{0\} \}$ is a basis for $\cc[f] = \mc{M}_0^!(\Gamma)$.

\begin{definition}[Replicable functions]\label{Def:ReplicableFunction}
	Let $f(q) = q^{-1} + \sum_{n\geq 0} c_nq^n$ be a function on the unit disk. Then $f$ is a \emph{replicable function} if there exists a sequence of functions $\{ f^{(1)}(q), f^{(2)}(q), f^{(3)}(q), \ldots \}$ such that for all $n\in\nn$, \[ F_{n,f}(q) = \sum_{H\in\mc{H}_n} f^{(H)}(H\tau), \] where for $H = \begin{psmallmatrix} a & b \\ 0 & d \end{psmallmatrix}$, $f^{(H)} := f^{(a)}$. If moreover $f^{(a)}(q)$ is a replicable function for all $a\in\nn$, we say $f$ is a \emph{completely replicable function}. If there exists some $k$ such that $f^{(a)} = f^{(a+k)}$ for all $a\in\nn$, then we say $f$ is of (finite) level $k$.
\end{definition}

Note the level $k$ of a replicable function of finite level is not unique, though there is a unique minimal level. This is a rather restrictive definition, and one may reasonably ask whether there are any replicable functions out there.

\begin{example}
	Let $f(q) = q^{-1}$, and let $f^{(a)}(q) = f(q)$ for all $a\in\nn$. Then since $f(q)^n = q^{-n}$, we have $F_{n,f}(X) = X^n$, while \[ \sum_{H\in\mc{H}_n} f(H\tau) = \sum_{a\divides n}\sum_{0\leq b<d} e^{2\pi i\frac{b}{d}}q^{-\frac{a}{d}} = q^{-n} = F_{n,f}(f(q)), \] since the sum of all $d$th roots of unity is zero unless $d=1$. That is, $f$ is a completely replicable function of level 1.
\end{example}

\begin{example}\label{Example:jIsReplicable}
	Let $f(q) = j(\tau)$, and let $f^{(a)}(q) = j(\tau)$ for all $a\in\nn$. Then for any $n\in\nn$ \[ \sum_{H\in\mc{H}_n} f^{(H)}(H\tau) = \sum_{H\in\mc{H}_n} j(H\tau) \] is, up to a factor of $n$, the classical Hecke operators for $\PSL_2(\zz)$ \cite{SERRE}. The classical Hecke operators act on modular functions, that is, $\sum_{H\in\mc{H}_n} j(H\tau)$ is also a modular function for $\PSL_2(\zz)$, and is holomorphic away from the cusp at infinity, and therefore a polynomial in $j(\tau)$. The formula for the action of the Hecke operators on $q$-series \cite{SERRE} shows that $\sum_{H\in\mc{H}_n} j(H\tau) = q^{-n}+O(q)$, so by uniqueness of the Faber polynomials, $\sum_{H\in\mc{H}_n} j(H\tau) = F_{n,f}(q)$. That is, $j(\tau)$ is a completely replicable function of finite level.
\end{example}

Perhaps you are not satisfied with these `trivial' examples of replication. For a larger collection of examples, we turn to Monstrous Moonshine.

\subsection{Monstrous Moonshine}\label{Subsection:MonstrousMoonshine}

In the mid-1970's, several connections were noticed between modular functions such as the $j$-invariant and the largest sporadic simple group $\bb{M}$, known as the Monster group. A set of conjectures, known as Monstrous Moonshine, were formulated by Conway and Norton \cite{CN}, and proved by Borcherds \cite{Borcherds} in the early 1990's. Monstrous Moonshine is now seen as part of a broader connection between certain finite groups, especially the sporadic simple groups, and automorphic forms.

Monstrous Moonshine provides a supply of some 171 replicable functions, but a conjecturally complete list of completely replicable functions with integer coefficients includes over 300 functions \cite{ACMS}

\begin{theorem}\label{Theorem:MoonshineReplication}
	Associated to each conjugacy class $[g]$ of elements of $\bb{M}$ is a function $T_g(q) = q^{-1}+O(q)$, which is the normalized Hauptmodul for a genus zero group $\Gamma = \Gamma_0(mh\edivides h){+e,f,g,\ldots}$. Moreover,  $T_g^{(a)}(q) := T_{g^a}(q)$ is the normalized Hauptmodul for $\Gamma^{(a)}$ (see \ref{Def:ReplicationOfGroups}), and we have \[ F_{n,g}(q) := F_{n,T_g}(q) = \sum_{H\in\mc{H}_n} T_g^{(H)}(H\tau), \] where for $H = \begin{psmallmatrix} a & b \\ 0 & d \end{psmallmatrix}$, we write $T_{g}^{(H)} = T_g^{(a)}$. In particular, $T_g(q)$ is a completely replicable function of finite level $mh$.
\end{theorem}

This is true due to Borcherds' proof of Monstrous Moonshine \cite{Borcherds}. Conjugacy classes of $\bb{M}$ are generally referred to by their name in the Atlas of Finite Groups \cite{ATLAS}, which are of the form `\#N' where `\#' is the order of the elements of the group, and `N' is a letter differentiating congugacy classes of elements of the same order.

\begin{example}
	The conjugacy class of the identity in $\bb{M}$ is denoted 1A, and we have already seen $T_{1A} = j(\tau) - 744 = q^{-1} + 196884q + \cdots$. There is also a conjugacy class of elements of order three denoted 3C, and corresponding to $T_{3C} = q^{-1} + 248q^2 + 4124q^{5}+\cdots$. The square of an element of conjugacy class 3C is again in class 3C, while the cube is the identity, so in class 1A. Thus, \[ T_{3C}^{(a)} = \begin{cases} T_{1A} & 3\divides a \\ T_{3C} & 3\ndivides a. \end{cases} \] As a concrete example, take $n=3$, then the Hecke set of level 3 is \[ \mc{H}_3 = \left\{ \begin{psmallmatrix} 3 & 0 \\ 0 & 1 \end{psmallmatrix}, \begin{psmallmatrix} 1 & 0 \\ 0 & 3 \end{psmallmatrix}, \begin{psmallmatrix} 1 & 1 \\ 0 & 3 \end{psmallmatrix}, \begin{psmallmatrix} 1 & 2 \\ 0 & 3 \end{psmallmatrix}  \right\}, \] and
	\begin{align*}
		F_{n,3C}(q) 
		&= \sum_{H\in\mc{H}_3} T_{3C}^{(H)}(H\tau) \\
		&= T_{1A}(3\tau) + T_{3C}\left(\frac{\tau}{3}\right) + T_{3C}\left(\frac{\tau+1}{3}\right) + T_{3C}\left(\frac{\tau+2}{3}\right).
	\end{align*}
\end{example}

Later work by others has developed the theory further, for example Carnahan's work on Generalized Moonshine, and from this we may learn some facts about the signs of coefficients of these $q$-series. The element $\begin{psmallmatrix} 0 & -1 \\ mh^2 & 0 \end{psmallmatrix}\in\Gamma_0(mh\divides h)+m$ is generally known as the \emph{Fricke involution}, and Carnahan defines an element $g\in\bb{M}$ to be a Fricke element of the Monster if the invariance group $\Gamma_0(mh\edivides h){+e,f,g,\ldots}$ for $T_{g}$ contains $\Gamma_0(mh\edivides h)+m$.
\begin{lemma}
	$T_{g}$ has non-negative $q$-series coefficients if and only if $g$ is a Fricke element of $\bb{M}$.
\end{lemma}
\begin{proof}
	Carnahan shows that for Fricke elements, the coefficients of the expansion of $T_g$ at zero are all dimensions of certain spaces. By invariance under the Fricke involution, the same holds for the coefficients of the expansion at infinity. By inspection of the remaining cases \cite{CN}, we find $T_g$ always has negative $q$-series coefficients when $g$ is not a Fricke element.
\end{proof}
\begin{corollary}\label{Corollary:SingleCuspAllPositiveFourierCoefficients}
	$T_{g}^{(a)}$ has non-negative $q$-series coefficients for all $a\in\nn$ if and only if the invariance group of $T_g$ is $\Gamma = \Gamma_0(mh\edivides h)+$ with $m$ square-free. That is, all replicates of $T_g$ have non-negative $q$-series coefficients if and only if $\Gamma$ has one cusp.
\end{corollary}
\begin{proof}
	Let the invariance group of $T_{g}$ be $\Gamma = \Gamma_0(mh\divides h){+e,f,g,\ldots}$. Suppose $k\divides m$ and $k\notin\langle e,f,g,\ldots\rangle$. Then \[ \Gamma^{(\frac{mh}{k})} = \Gamma_0\left( \frac{mh}{ (mh,\frac{mh}{k}) } \bigst \frac{h}{(h,\frac{mh}{k})}\right)+e',f',g',\ldots = \Gamma_0(k)+e',f',g',\ldots, \] where $\langle e',f',g',\ldots \rangle\subset\langle e,f,g,\ldots\rangle$, and thus $\Gamma_0(k)+k\not\subset\Gamma_0(k)+e,f,g,\ldots$. That is, $g^{\frac{mh}{k}}$ is not a Fricke element, or equivalently, $T_{g}^{(\frac{mh}{k})}$ has negative $q$-series coefficients.
	
	On the other hand, suppose every $k\divides m$ is an element of $\langle e,f,g,\ldots\rangle$, so that $\langle e,f,g,\ldots\rangle = Ex(m)$ and $m$ is square-free, by Lemma \ref{Lemma:SquareFreeIsAllDivisors}. By Lemma \ref{Lemma:SingleCuspCriterion}, this is equivalent to $\Gamma$ having one cusp. It remains only to show that $T_g^{(a)}$ has non-negative $q$-series coefficients, or equivalently, that $\Gamma^{(a)}$ contains its Fricke involution. By the definition of $\Gamma^{(a)}$ (Def \ref{Def:ReplicationOfGroups}), this is equivalent to $\frac{m(h,a)}{(mh,a)}\in \Ex(m)\cap \Ex\left(\frac{m(h,a)}{(mh,a)}\right)$. Since $\Ex(m)$ contains all divisors of $m$, $\frac{m(h,a)}{(mh,a)}\in \Ex(m)$. And for any $n\in\nn$, $n\in\Ex(n)$, so $\frac{m(h,a)}{(mh,a)}\in \Ex\left(\frac{m(h,a)}{(mh,a)}\right)$. Thus, $\frac{m(h,a)}{(mh,a)}$ is in the intersection so $\Gamma^{(a)}$ contains its Fricke involution, and $T_g^{(a)}$ has non-negative $q$-series coefficients.
\end{proof}

\section{Faber Polynomials for Replicable Functions}\label{Section:ZerosOfFaberPolynomials}

Let $f(q) = q^{-1} + \sum_{k\geq 0} a_kq^k$ be a replicable function, and a Hauptmodul for a genus zero group $\Gamma = \Gamma_0(mh\edivides h){+e,f,g,\ldots}$. Our goal is to locate the zeros of the Faber polynomial $F_{n,f}(X)$ for $f(q)$, for all $n\in\nn$. We will use two different strategies, depending on whether or not $(n,h) = 1$. We begin with the case where $(n,h) > 1$, which uses a certain `harmonic' relationship among replicable functions observed by Conway and Norton \cite{CN}. After that, we handle the case of $(n,h) = 1$ using an extension of the method of Asai, Kaneko, and Ninomiya \cite{AKN}.

\subsection{Harmonics and Faber Polynomials}\label{Subsection:Zeros:ConjugateGroups}

Let $f(q)$ be a replicable function for a group of the form $\Gamma_0(mh\edivides h){+e,f,g,\ldots}$, with replicates $f^{(a)}$, and let $n\in\nn$ with $(n,h) = d > 1$. In this case, the form of the replicate $f^{(d)}(q)$ is particularly nice, being the $d$th \emph{harmonic} of $f$ \cite{CN}. That is, for some constant $c$, the function $f^{(d)}(q)$ satisfies \[ f^{(d)}(q^d) + c = f(q)^d. \] From this identity we may deduce a relationship between the Faber polynomials for $f$ and $f^{(d)}$.

\begin{lemma}\label{Lemma:ConjugateGroupFaberZeros}
	Let $f(q) = q^{-1} + \sum_{k\geq 0} a_kq^k$, let $d\divides n$, and suppose $g(q) = q^{-1} + \sum_{k\geq 0} b_kq^k$ is such that $g(q^d) + c = f(q)^d$ for some constant $c$. Then \[ F_{n,f}(X) = F_{\frac{n}{d},g}(X^d - c). \] In particular, if $Y_1, Y_2,\ldots Y_\frac{n}{d}$ are the $\frac{n}{d}$ zeros of $F_{\frac{n}{d},g}(Y)$ counting multiplicity, then the $n$ zeros of $F_{n,f}$ are precisely the $d$ complex roots of $Y_1+c, Y_2+c, Y_\frac{n}{d}+c$, counting multiplicity.
\end{lemma}
\begin{proof}
	Since $F_{\frac{n}{d},g}(X^d-c)$ is a polynomial such that \[ F_{\frac{n}{d},g}(f(q)^d - c) =  F_{\frac{n}{d},g}(g(q^d)) = \left(q^d\right)^{-\frac{n}{d}} + O(q^d) = q^{-n} + O(q^d), \] by uniqueness of Faber polynomials, we conclude $F_{n,f}(X) = F_{\frac{n}{d},g}(X^d - c)$.  
\end{proof}

Bannai, Kojima, and Miezaki (\cite{BKM}, Conjecture 3.1) observed this relationship, Lemma \ref{Lemma:ConjugateGroupFaberZeros} settles a conjecture of theirs that states whenever the roots of $F_{\frac{n}{d},g}$ are real, the roots of $F_{n,f}$ are $h$th roots of real numbers. Indeed, we can say even more, but first we apply Lemma \ref{Lemma:ConjugateGroupFaberZeros} to an example.

\begin{example}\label{Example:3bar3ZerosFrom1}
	The normalized Hauptmodul $T_{3C}(q) = q^{-1} + 248q^2 + 4124q^5 + 34752q^8 + \ldots$ for $\Gamma_0(3\edivides 3)$ satisfies the relationship $T_{3C}(q)^3 = T_{1A}(q^3) + 744 = q^{-3} + 744 + 196884q^3 + \ldots$, where $T_{1A}(q)$ is the normalized Hauptmodul for $\PSL_2(\zz)$ (\cite{CN}). Asai, Kaneko, and Ninomiya located the $n$ zeros of $F_{n,1A}(Y)$ in the interval $(-744, 984)$, or equivalently, with $j(q) = T_{1A} + 744$, the zeros of $F_{n,j}(X)$ are in $(0,1728)$. The zeros of $F_{3n,3C}(X)$ are all third roots of the zeros of $F_{n,j}(X)$, so one third are in the interval $(0,12)$, and the rest are found by rotating these zeros through an angle of $\pm\frac{2\pi}{3}$ in the complex plane. (See \cite{BKM} Section 3.) We locate the zeros of the Faber polynomials of degree not divisible by 3 below in Section \S\ref{Section:Cases}.
\end{example}

Whenever $(n,h) = d > 1$, Lemma \ref{Lemma:ConjugateGroupFaberZeros} reduces the problem of locating the zeros of $F_{n,f}$ to finding the zeros of $F_{ \frac{n}{d}, f^{(d)} } (X)$. But even when $(n,h) = 1$, the zeros of $F_{n,f}(X)$ come in sets of $h$th roots of unity. Recall the group $\Gamma_0(mh\edivides h){+e,f,g,\ldots}$ is the kernel of a homomorphism $\lambda:\Gamma_0(mh\divides h){+e,f,g,\ldots}\to\cc^\times$ which takes values in the $h$th roots of unity, and in particular $\lambda(T^{\frac{1}{h}}) = e^{-\frac{2\pi i}{h}}$ (\cite{CN}).

\begin{lemma}\label{Lemma:ConjugateGroupMoreFaberZeros}
	Let $f(\tau) = q^{-1} + \sum_{n\geq 1} a_kq^k$ be the normalized Hauptmodul for $\Gamma_0(mh\edivides h){+e,f,g,\ldots}$, and the $h$th harmonic of $g(\tau) = q^{-1}+ \sum_{n\geq 0} b_kq^k$. Let $n\in\nn$ and write $n = hs+t$ for $s,t\in\zz$ with $0\leq t<h$. Then $F_{n,f}(X) = X^t g(X^h)$ for some polynomial $g(X)$ of degree $s$.
	
	In particular, $X=0$ is a zero of order $t$ of $F_{n,f}(X)$, and if $F_{n,f}(X_0) = 0$ then $F_{n,f}(e^{2\pi i\frac{c}{h}}X_0) = 0$ for all $c\in\zz$. Equivalently, if $F_{n,f}(\tau_0) = 0$ then $F_{n,f}(T^\frac{c}{h}\tau_0) = 0$ for all $c\in\zz$.
\end{lemma}
\begin{proof}
	Let $\zeta_h = e^{\frac{2\pi i}{h}}$. Note that $F_{n,f}(X_0) = 0$ implying $F_{n,f}(\zeta_h^cX_0) = 0$ will follow immediately from $F_{n,f}(X) = X^tG(X^h)$. Since $f$ transforms with character $\lambda$, with $\lambda(T^{\frac{1}{h}}) = \zeta_h^{-1}$, we have $f(\tau) = \lambda(T^{\frac{1}{h}})f(T^{\frac{1}{h}}\tau)$, so $f(T^{\frac{1}{h}}\tau) = \zeta_h f(\tau)$. Since $F_{n,f}(X)$ and $F_{n,f}(\tau)$ are related by $X=f(\tau)$, we see that $\tau\mapsto T^{\frac{1}{h}}\tau$ is equivalent to $X\mapsto \zeta_h X$, and so the two statements in the lemma are indeed equivalent. 
	
	As the $h$th harmonic of $g(\tau)$, we have $f(q)^h = g(q^h) + c$, for some constant $c$. Since \[ g(q^h) + c = q^{-h} + (b_0+c) + b_1q^h + b_2q^{2h} + \ldots, \] the algebra of formal power series then forces $f(q)$ to be of the form \[ f(q) = q^{-1} + a_{h-1}q^{h-1} + a_{2h-1}q^{2h-1} + \ldots. \] Because $f(q)$ only contains terms $a_jq^j$ with $j\equiv -1\pmod{h}$, the product $f(q)^n$ for any $n$ contains only terms $\tilde{a}_\ell q^\ell$ with $\ell\equiv -n\pmod{h}$. Now, consider constructing $F_{n,f}(\tau) = q^{-n} + O(q)$ by computing $f^n(q) = q^{-n} + na_{h-1}q^{h-n} + \ldots$, noting only powers of $q$ congruent to $-n$ modulo $h$ appear. We subtract $na_{h-1}f^{n-h}(q) = na_{h-1}q^{h-n} + \ldots$, noting again, all powers of $q$ are congruent to $-n$ modulo $h$. Repeating this process at most $n$ times, each time taking some power $f^{n-ch}$ of $f$, we construct $F_{n,f}(X) = X^n - na_{h-1}X^{n-h} + \ldots$ using only powers of $X$ congruent to $n$ modulo $h$, so $F_{n,f}(X) = X^t G(X^h)$, as desired.
\end{proof}

Lemma \ref{Lemma:ConjugateGroupMoreFaberZeros} allows us to reduce the problem of locating zeros of $f(\tau)$ in the fundamental domain $\mc{D}(\Gamma_0(mh\edivides h)+e,f,g\ldots)$ to that of locating zeros just in $\mc{D}(\Gamma_0(mh\divides h)+e,f,g\ldots)$, since by construction, the fundamental domain for the group $\Gamma_0(mh\edivides h){+e,f,g,\ldots}$ consists of $h$ translated copies of the fundamental domain for $\Gamma_0(mh\divides h){+e,f,g,\ldots}$. Moreover, $F_{n,f}(\tau)$ has zeros of total order $t$ at the points $\tau_0$ such that $f(\tau_0) = 0$. By locating $s$ additional zeros of $F_{n,f}(\tau)$ in $\mc{D}(\Gamma_0(mh\divides h){+e,f,g,\ldots})$ (that is, at points $\tau$ such that $f(\tau) \neq 0$), we therefore locate a total of $n = hs + t$ zeros on $\mc{D}(\Gamma_0(mh\edivides h){+e,f,g,\ldots})$. These $s$ zeros are located by the generalization of the method of Asai, Kaneko, and Ninomiya (\cite{AKN}), which we prove in the next subsection. The example using $\Gamma_0(3\divides 3)$ will then be completed in \S\ref{Subsection:3C}.

Thus, if we know the location of the zeros of all Faber polynomials for the ``proper'' replicates of $f(q)$ (those replicates $f^{(a)} \neq f$), then we also know the location of the zeros of the Faber polynomials $F_{n,f}$ for all $n$ with $(n,h) > 1$. Moreover, when $(n,h) = 1$, the above paragraph shows that we need only locate the zeros for $F_{n,f}(q)$ on $\mc{D}(\Gamma_0(mh\divides h){+e,f,g,\ldots})$. That is, we may safely take the point of view that $f$ is a modular function with character $\lambda$ for $\Gamma_0(mh\divides h){+e,f,g,\ldots}$.

\subsection{Beginning the approximation}

Fix a genus zero group $\Gamma_0(mh\edivides h){+e,f,g,\ldots}$ with normalized Hauptmodul $f(\tau)$ a completely replicable function of finite level. While the results of the previous section allow us to identify some zeros of the Faber polynomials $F_{n,f}(X)$, we will need to locate many of the zeros using an approximation and the intermediate value theorem. Our first goal will be to approximate $f^{(H)}(H\tau)$ for any $H\in\mc{H}_n$ (see Def. \ref{Def:HeckeSet}). We will use the following estimating function.

\begin{definition}\label{Def:Estimates}
	Let $\Gamma = \Gamma_0(mh\divides h){+e,f,g,\ldots}$ be a group appearing in Monstrous Moonshine. Then for any $n\in\nn$, and $H\in\mc{H}_n$, we define $E_H : \hh\times\Gamma^{(H)}\to \cc$ by \[ E_H(\tau, G) = \lambda^{(H)}(G)^{-1}e^{-2\pi iGH\tau}. \] We call $E_H(\tau,G)$ the \emph{exponential estimate with respect to $H$ and $G$}, or simply an \emph{estimate}.
\end{definition}
\begin{proof}
	Since $\Gamma$ appears in Monstrous Moonshine, $\Gamma^{(H)}$ also appears in Monstrous Moonshine. Then by Lemma \ref{Lemma:LambdaProperties}, the character $\lambda^{(H)}$ exists, so $E_H(\tau,G)$ is well-defined.
\end{proof}

In fact, $E_H$ descends to a function on $\hh\times (\LeftQuotientByStab{\Gamma^{(H)}})$.

\begin{lemma}\label{Lemma:EstimatesConstantUnderTranslation}
	With the assumptions of Definition \ref{Def:Estimates}, for any $G\in\Gamma^{(H)}$ and $V\in\Gamma^{(H)}_\infty$, one has \[ E_H(\tau,VG) = E_H(\tau,G). \] That is, $E_H(\tau,[G])$ is well-defined, where $[G]\in\LeftQuotientByStab{\Gamma^{(H)}}$.
\end{lemma}
\begin{proof} 
	Letting $H = \begin{psmallmatrix} a & b \\ 0 & d \end{psmallmatrix}$, we find that $\Gamma^{(H)}_\infty = \Gamma^{(a)}_\infty = \langle T^{\frac{(h,a)}{h}}\rangle$. Thus $V = T^{\frac{r(h,a)}{h}}$ for some $r\in\zz$, and by the description of $\lambda^{(H)}$ (Subsection \S\ref{Subsection:MoreGroups}), we have $V\tau = \tau+\frac{r(h,a)}{h}$ for any $\tau\in\hh$, and $\lambda^{(H)}(V) = e^{-\frac{2\pi ir(h,a)}{h}}$. Thus
	\begin{align*}
		E_H(\tau, VG) 
		&= \lambda^{(H)}(VG)^{-1}e^{-2\pi iVG\tau} \\
		&= e^{\frac{2\pi ir(h,a)}{h}}\lambda^{(H)}(G)^{-1}e^{-2\pi i(G\tau + \frac{r(h,a)}{h})} \\
		&= \lambda^{(H)}(G)^{-1}e^{-2\pi iG\tau} \\
		&= E_H(\tau, G).
	\end{align*}
\end{proof}

We may often express $E_H(\tau,[G])$ using an element of $\Gamma$ rather than $\Gamma^{(H)}$. Thus, as $H\in\mc{H}_n$ varies, we need not keep track of multiple groups $\Gamma^{(H)}$.

\begin{corollary}\label{Corollary:EstimateBounds}
	Let $\Gamma = \Gamma_0(mh\divides h){+}$ with $m$ square-free, let $H\in\mc{H}_n$ with $(n,h)=1$, $G\in\Gamma^{(H)}$, and $K\in\Gamma$ such that $\pi(GH)=\pi(K)$. Then for any $\tau\in\mc{D}(\Gamma)$, either \[ \abs{E_H(\tau,G)} \leq e^{\pi n\Im\tau}, \] or else $\phi_n(K)=H$ and \[ E_H(\tau,G) = E_H(\tau, \Dmat{n}KH^{-1}). \] In this latter case, $\abs{E_H(\tau,G)} \leq e^{2\pi n\Im\tau}$.
\end{corollary}
\begin{proof}
	First, by Lemma \ref{Lemma:SingleCuspCriterion}, we have $\pi(\Gamma) = \pp^1(\qq)$, so there indeed exists a $K\in\Gamma$ such that $\pi(GH)=\pi(K)$. For such a $K$, we then have \[ E_H(\tau,G) := \lambda^{(H)}(G)^{-1}e^{-2\pi iGH\tau} = \lambda^{(H)}(G)^{-1}e^{-2\pi i(GHK^{-1})K\tau}, \] so $\abs{E_H(\tau,G)} = e^{2\pi\sigma(GHK^{-1})\Im K\tau}$. By Lemma \ref{Lemma:GHKinverseExpression}, either $\sigma(GHK^{-1}) \leq \frac{n}{2}$, or else $\sigma(GHK^{-1})=n$ and $[G] = \Dmat{n}KH^{-1}$. This immediately implies $\abs{E_H(\tau,G)}\leq e^{\pi n\Im\tau}$, or else $E_H(\tau,G) = E_H(\tau, \Dmat{n}KH^{-1})$ and $\abs{E_H(\tau,G)} \leq e^{2\pi n\Im\tau}$. In the latter case, since $\Dmat{n}KH^{-1}\in\Gamma^{(H)}$, one must have $H = \phi_n(K)$, by the proof of Theorem \ref{Theorem:GroupCosets}.
\end{proof}

Our next lemma is a generalization of a key step in Asai, Kaneko, and Ninomiya's \cite{AKN} proof for the case of $\PSL_2(\zz)$ and $j(\tau)$, where they bound $\abs{j(\tau) - 744 - e^{-2\pi i\tau}} < 1335$.

\begin{lemma}\label{Lemma:FunctionBoundsOnFundDomain}
	Let $\Gamma = \Gamma_0(mh\divides h){+}$ with $m$ square-free appear in Monstrous Moonshine. Let $n\in\nn$, $H = \begin{psmallmatrix} a & b \\ 0 & d \end{psmallmatrix} \in\mc{H}_n$, and $G_H\in\Gamma^{(H)}$ such that $G_HH\tau\in\mc{D}(\Gamma^{(H)})$. Let $f^{(H)}(\tau) = q^{-1} + \sum_{k=1}^\infty c_k^{(H)}q^k$ be the normalized Hauptmodul (with character $\lambda^{(H)}$) for $\Gamma^{(H)}$, and let \[ y_0^{(a)} = y_0^{(H)} = \inf\{\Im\tau\st \tau\in\mc{D}(\Gamma^{(H)}) \}. \] Then we have \[ \abs{ f^{(H)}(H\tau) - E_H(\tau,G_H) } \leq \max_{r\in\nn}\left\{ f^{(r)}(iy_0^{(r)}) - e^{2\pi y_0^{(r)}} \right\}. \]
\end{lemma}
\begin{proof}
	Note that $\Gamma^{(H)}$ has one cusp for all $H\in\mc{H}_n$, by Lemma \ref{Lemma:SingleCuspCriterion}, since explicitly, we have $\Gamma^{(H)} = \Gamma_0\left(\frac{m(h,a)}{(mh,a)}\frac{h}{(h,a)}\divides \frac{h}{(h,a)}\right){+}$ (Def. \ref{Def:ReplicationOfGroups}), and $\frac{m(h,a)}{(mh,a)}\divides m$ is square-free. Thus, $y_0^{(H)}$ is strictly greater than zero.
	
	Now, by Corollary \ref{Corollary:SingleCuspAllPositiveFourierCoefficients}, $f^{(H)}$ has non-negative $q$-series coefficients. Thus,	
	\begin{align*}
		\abs{ f^{(H)}(H\tau) - E_H(\tau,G_H) }
		& = \abs{ \lambda^{(H)}(G_H)^{-1}f^{(H)}(G_HH\tau) - \lambda^{(H)}(G_H)^{-1}e^{-2\pi iG_HH\tau} } \\
		&\leq \sum_{k=1}^\infty c_k^{(H)}\abs{e^{2\pi ikG_HH\tau}} \\
		&= \sum_{k=1}^\infty c_k^{(H)}e^{-2\pi k\Im G_HH\tau} \\
		&= \sum_{k=1}^\infty c_k^{(H)}e^{-2\pi ky_0^{(H)}} \tag{$G_HH\tau\in\mc{D}(\Gamma)$}\\
		&= f^{(H)}(iy_0^{(H)}) - e^{-2\pi i(iy_0^{(H)})}.
	\end{align*}
	
	Since every $f^{(H)} = f^{(a)}$ for some $a\in\nn$, \[ \abs{ f^{(H)}(H\tau) - E_H(\tau,G_H) } \leq \sup_{r\in\nn}\{ f^{(r)}(iy_0^{(r)}) - e^{2\pi y_0^{(r)}} \}. \] Since $f$ is of finite level, this supremum is over a finite set (taking $r$ in the set of all divisors of $mh$ suffices), and is therefore equal to the maximum.
\end{proof}

Note that $G_H$ is not uniquely defined in the above lemma, in the case where $H\tau$ has nontrivial stabilizer in $\Gamma^{(H)}$, which necessarily places $G_HH\tau$ on the lower boundary of $\mc{D}(\Gamma^{(H)})$. On the other hand, if $G_HH\tau$ is not on the lower boundary $\mc{C}(\Gamma^{(H)})$, then $G_H$ is unique.

\begin{corollary}\label{Corollary:FirstBound}
	With the assumptions of Lemma \ref{Lemma:FunctionBoundsOnFundDomain}, let $B = \max_{r\in\nn}\{ f^{(r)}(iy_0^{(r)}) - e^{2\pi y_0^{(r)}} \}$. For any $H\in\mc{H}_n^*$, let $K_H\in\Gamma$ be such that $\pi(K_H)=\pi(G_HH)$, and let $\tau\in\mc{D}(\Gamma)$. Then \[ \bigabs{ F_{n,f}(\tau) - \sum_{H\in\mc{H}_n^*} E_H(\tau,G_H) } \leq Bn^2 + \abs{(\mc{H}_n\setminus\mc{H}_n^*)} e^{\pi n\Im \tau}, \] where $F_{n,f}(X)$ is the $n$th Faber polynomial for $f(\tau) = f^{(1)}(\tau)$.
\end{corollary}
\begin{proof}
	Note that $K_H$ exists, since $\Gamma$ has one cusp (\ref{Lemma:SingleCuspCriterion}), and that $\phi_n(\Gamma) = \mc{H}_n^*$ (Lemma \ref{Lemma:PhinWellDefined}). Then \[ E_H(\tau,G_H) = \lambda^{(H)}(G_H)^{-1}e^{-2\pi iG_HH\tau} = \lambda^{(H)}(G_H)^{-1}e^{-2\pi iG_HHK_H^{-1}K_H\tau}, \] where by Lemma \ref{Lemma:GHKinverseExpression}, either $\sigma(GHK_H^{-1})=n$, $H=\phi_n(K_H)$, and $\Dmat{n}K_HH^{-1} = [G_H]$ (so in particular, $\Dmat{n}K_HH^{-1}\in\Gamma^{(H)}$, making $E_H(\tau,\Dmat{n}K_HH^{-1})$ well-defined), implying \[ E_H(\tau,G_H) = E_H(\tau,\Dmat{n}K_HH^{-1}) = \lambda^{(H)}(\Dmat{n}K_HH^{-1})e^{-2\pi inK_H\tau}, \] or else $\sigma(GHK_H^{-1})\leq \frac{n}{2}$ and thus $\abs{E_H(\tau,G_H)} = e^{2\pi \sigma(GHK_H^{-1})\Im K\tau} \leq e^{\pi n\Im \tau}$, since $\tau\in\mc{D}(\Gamma)$ implies $\Im K_H\tau \leq \Im\tau$. Moreover, if $H\notin\mc{H}_n^*$ then $\Dmat{n}KH^{-1}\notin\Gamma^{(H)}$ for any $K\in\Gamma$, so that necessarily $\sigma(GHK_H^{-1})\leq\frac{n}{2}$, and $\abs{E_H(\tau,G_H)}\leq e^{\pi n\Im\tau}$.
	
	Recall (Def. \ref{Def:ReplicableFunction}) for a replicable function $f$ that \[ F_{n,f}(\tau) = q^{-n} + O(q) = \sum_{H\in\mc{H}_n} f^{(H)}(H\tau). \] When $f$ also satisfies the hypotheses of Lemma \ref{Lemma:FunctionBoundsOnFundDomain}, the quantity $B$ is well-defined, and 
	\begin{align*}
		\bigabs{ F_{n,f}(\tau) - \sum_{H\in\mc{H}_n} E_H(\tau,G_H) }
		&= \bigabs{ \sum_{H\in\mc{H}_n} f^{(H)}(H\tau) - \sum_{H\in\phi_n(\Gamma)} E_H(\tau,G_H) } \\
		&\leq \sum_{H\in\mc{H}_n} \abs{ f^{(H)}(H\tau) - E_H(\tau, G_H) } \\
		&\leq B\sigma_1(n) \\
		&\leq Bn^2,
	\end{align*}
	since $\sigma_1(n) = \sum_{d\divides n} d \leq n^2$. Combined with the deductions in the above paragraph about $\abs{E_H(\tau,G_H)}$ for $H\notin\mc{H}_n^*$, we obtain the desired inequality.		
\end{proof}

While we would like to say that $E_H(\tau,G_H) = E_H(\tau,\Dmat{n}K_HH^{-1})$ for all $H\in\mc{H}_n^*$, we do not know that $\Dmat{n}K_HH^{-1}\in\Gamma^{(H)}$. By Lemma \ref{Lemma:GHKinverseExpression}, we know that \[ \phi_n^{-1}(H) = \{ K\in\Gamma \st \Dmat{n}KH^{-1}\in\Gamma^{(H)} \}, \] so it would suffice to determine when $K_H\in\phi_n^{-1}(H)$. Unfortunately, we have not yet found a general method of determining if $\phi_n(K_H)=H$ using only $\pi(K_H)=\pi(G_HH)$ and $G_HH\tau\in\mc{D}(\Gamma^{(H)})$.

\subsection{Completing the approximation}\label{Subsection:CompletingApproximation}

Corollary \ref{Corollary:EstimateBounds} shows that when $\tau\in\mc{D}(\Gamma)$, we have the inequality $\abs{E_H(\tau,G_H)} \leq e^{2\pi n\Im\tau}$ for all $H\in\mc{H}_n$ (with the appropriate assumptions). Moreover, from the proof of Cor. \ref{Corollary:EstimateBounds}, this inequality can be an equality only if $H = \phi_n(K) \in\mc{H}_n^*$ and $\Im K\tau = \Im\tau$, so that $\tau\in\mc{A}(K)$. That is, for some $r\in\zz$, $\Dmat{n^{-1}}T^{\frac{r}{h}}G_HH$ must be a coset representative for an element of the critical set \[ \mc{K} = \{ [K]\in\LeftQuotientByStab{\Gamma} \st \mc{A}(K)\cap\mc{D}(\Gamma) \}, \] which is finite (see \S\ref{Subsection:ReplicableGroups}). For sufficiently large $n$, we will show that these terms $E_H(\tau,G_H)$ corresponding to a $[K]\in\mc{K}$ will suffice for approximating $F_{n,f}(\tau) \approx \sum_{H\in\mc{H}_n} E_H(\tau,G_H)$.

We first define constants $c$ and $N$, which are easily computed, then show how to use these to construct a bound for $F_{n,f}(\tau)$.

\begin{definition}\label{Def:NcBound}
	Let $\Gamma = \Gamma_0(mh\divides h){+}$ with $m$ square-free, and let $K_i = \begin{psmallmatrix} k_ihw_i & x_i \\ mhy_i & k_ihz_i \end{psmallmatrix}$, for $1\leq i\leq \abs{\mc{K}}$, be a complete set of reduced coset representatives for $\mc{K}$, with $K_1 \in \Gamma_\infty$. Let $\mc{C}(\Gamma) = \mc{C}_1\cup\mc{C}_2\cup\cdots\cup\mc{C}_t,$ where each $\mc{C}_i$ is a part of a side of the fundamental domain, as a hyperbolic polygon, contained entirely within some $\mc{A}_{K_i}$, where $[K_i]\in\mc{K}$, and let \[ \mc{T} = \{ \tau \st \tau\text{ is the endpoint of some }\mc{C}_i, 1\leq i\leq t \}. \] 
	
	From this data, compute the intermediate data
	\begin{table}[!ht]\label{Table:PreNcValues}
		\setlength{\tabcolsep}{18pt}
		\renewcommand{\arraystretch}{1.5}
		\centering
		\begin{tabular}{| c c |}
			\hline
			Intermediate data & formula \\ 
			\hline
			$y_0$ & $\min\{\Im\tau\st \tau\in\mc{T} \}$  \\
			$\mc{U}$ & $\left\{ (k,y,z) \bigst k\edivides m, 0 < y < \frac{\sqrt{2k}}{mhy_0}, \abs{z} < \frac{2\sqrt{2} + \sqrt{m+4}}{y_0h\sqrt{2km}}, (kz,\frac{m}{k}y) = 1 \right\}$ \\
			$c(k,y,z)$ & $\max\left\{ \frac{k}{\abs{mhy\tau + kz}^2} \st \tau\in\mc{T} \right\}$ \\
			$c_0$ & $\max\{c(k,y,z) \st (k,y,z)\in\mc{U}, c(k,y,z) \neq 1\}$ \\
			\hline
			$N_1$ & $\max\left\{ \frac{\abs{\tau - \pi(K_i)}^2}{\Im\tau\rho^2(K_i)} \bigst 2\leq i\leq \abs{\mc{K}}, \tau\in\mc{T} \right\}$ \\
			\hline
			$N_2$ & $\max\left\{ \frac{mh}{(k_i,k_j)}\abs{k_iz_iy_j - k_jz_jy_i} \bigst 2\leq i\leq \abs{\mc{K}}, 2\leq j\leq \abs{\mc{K}} \right\}$ \\
			\hline
		\end{tabular}
	\end{table}
	
	Then
	\begin{align*}
		c &:= \max\left\{\frac{1}{2}, c_0 \right\}, \\ 
		N &:= \max\{ N_1, N_2 \}.
	\end{align*}
\end{definition}

Each of the constants above is defined as a maximum over a finite set, so readily computable. To avoid repetition, note that \[ N_2 = \max\left\{ \frac{mh}{(k_i,k_j)}\abs{k_iz_iy_j - k_jz_jy_i} \bigst 2\leq i < j\leq \abs{\mc{K}} \right\}. \] To show the two definitions given for $y_0$ in Def. \ref{Def:NcBound} agree, first note \[ \min\{\Im\tau\st\tau\in\overline{\mc{D}(\Gamma)} \} = \min\{\Im\tau\st\tau\in\boundary\mc{D}(\Gamma) \}, \] and that any point on the boundary is the image under $\Gamma$ of some point either on the lower boundary $\mc{C}$ or on the `side boundary' of $\mc{D}(\Gamma)$, along the line $\Re\tau = \frac{1}{2h}$, and it therefore follows \[ \min\{\Im\tau\st\tau\in\boundary\mc{D}(\Gamma) \} = \min\{\Im\tau\st \tau\in\mc{C} \}. \] Now, each arc $\mc{C}_i$ of $\mc{C}(\Gamma)$ has two endpoints, and since these arcs are half-circles centered on the real axis, any point with minimal imaginary part on the arc $\mc{C}_i$ is an endpoint, so that \[ \min\{\Im\tau\st \tau\in\mc{C} \} = \min\{\Im\tau\st \tau\in\mc{T} \}, \] and thus the two definitions of $y_0$ agree.

Before proving some properties of the constants $c$ and $N$, we give some examples of their computation, which we will use in \S\ref{Section:Cases}.

\begin{table}[!ht]\label{Table:NcValues}
	\setlength{\tabcolsep}{18pt}
	\renewcommand{\arraystretch}{1.5}
	\centering
	\begin{tabular}{| l c c |}
		\hline
		Group & $c$ & $N$  \\ 
		\hline
		$\Fricke{2}$ & $\frac{1}{2}$ & $3\sqrt{2}$  \\
		$\Gamma_0(6){+}$ & $\frac{1}{2}$ & 13 \\
		$\Gamma_0(3\divides 3)$ & $\frac{1}{2}$ & 9 \\
		\hline
	\end{tabular}
	\caption{Values of $c$ and $N$ for a few groups $\Gamma_0(mh\divides h){+}$.}
\end{table}
\begin{example}\label{Example:NcComputation}
	To compute the value of $c$ for $\Fricke{2}$ (where $m=2$, $h=1$), we reference the fundamental domain given in Figure \ref{Fig:FundamentalDomainFricke2}, as well as the critical set in Table \ref{Table:CriticalSets}. From this data, we find the lower boundary is $\mc{C}(\Fricke{2}) = \{ \tau\in\hh \st 0\leq \Re\tau \leq \frac{1}{2}, \abs{\tau}^2 = \frac{1}{2}, \}$ so that $\mc{T} = \{ \frac{i}{\sqrt{2}}, \frac{1}{2}+\frac{i}{2} \}$, and $y_0 = \frac{1}{2}$. Then one computes that \[ \mc{U} = \left\{ (1,1,z) \st z = -5,-3,-1,1,3,5 \right\}\cup\left\{ (2,1,z) \st z = -3,-2,-1,0,1,2,3 \right\}, \] with 
	\begin{align*}
		c(1,1,z) &= \max\left\{\abs{z+\sqrt{2}i}^{-2}, \abs{z+1+i}^{-2} \right\}, \\
		c(2,1,z) &=\max\left\{\abs{z\sqrt{2}+i}^{-2}, 2\abs{2z+1+i}^{-2} \right\}.
	\end{align*} 
	For the triples in $\mc{U}$, we find $c(k,y,z)=1$ when $(k,y,z)=(1,1,-1),(2,1,0),(2,1,-1)$, while the maximum over all other triples is $c_0 = \frac{1}{3} = c(1,1,1) = c(2,1,1)$, so $c = \frac{1}{2}$. 
	
	The triples $(1,1,-1)$, $(2,1,0)$, and $(2,1,-1)$ ruled out above correspond to the non-identity elements of the critical set, via $\pi(K_i) = -\frac{kz}{mhy} = -\frac{kz}{2y}$. We confirm these are the same triples from Example \ref{Example:CriticalSets} (and indeed, one need not compute the critical set beforehand, it may be identified at this point). These three triples and the set $\mc{T}$ therefore comprise the necessary information for computing $N_1$ and $N_2$, since we may compute $\rho^2(K_i) =  r_{mh\divides h}^2\circ\pi(K_i) = \frac{k}{(mhy)^2}$ (see Lemma \ref{Lemma:RadiusFromPi}). We find $N_1 = 3\sqrt{2}$, with the maximum achieved at $\tau = \frac{i}{\sqrt{2}}$ and $(k,y,z) = (1,1,-1)$ or $(2,1,-1)$, while $N_2 = 2$, so that $N = 3\sqrt{2}$. In practice, we will round this up to $N=5$.
	
	For $\Gamma_0(6){+}$, the lower boundary contains two arcs (see Figure \ref{Fig:FundamentalDomainGamma06+}), the set of all endpoints is $\mc{T} = \{ \frac{i}{\sqrt{6}}, \frac{i}{2\sqrt{3}}+\frac{1}{2}, \frac{i}{3\sqrt{2}} + \frac{1}{3} \}$, and $y_0 = \frac{1}{3\sqrt{2}}$. Then a little algebra gives that \[ c(k,y,z) = \max\left\{\frac{1}{\frac{6}{k}y^2+kz^2}, \frac{1}{\frac{3}{k}y^2+k\left(\frac{3}{k}y+z\right)^2}, \frac{1}{\frac{2}{k}y^2 + k\left(\frac{2}{k}y + z\right)^2} \right\}. \] We have that \[ \mc{U} = \left\{ (k,y,z) \st 0 < y < k, \abs{z} < \frac{2\sqrt{3}+\sqrt{15}}{\sqrt{k}}, (kz, \frac{6}{k}y) = 1 \right\}, \] and find that $c(k,y,z) = 1$ for $(k,y,z) = (6,1,0), (3,1,-1), (2,1,-1)$, corresponding to the three non-identity elements of the critical set (see Example \ref{Example:CriticalSets}), while otherwise, $c(k,y,z) \leq \frac{1}{2}$, since the denominator of the above expressions is a positive integer. After simplifying, \[ N_1 = \max\left\{ \sqrt{6}\left(\frac{6}{k}y^2 + kz^2\right), 2\sqrt{3}\left( \frac{3}{k}y^2+k\left(\frac{3}{k}y+z\right)^2 \right), 3\sqrt{2}\left( \frac{2}{k}y^2 + k\left(\frac{2}{k}y + z\right)^2 \right) \right\}, \] where $(k,y,z) \in\{ (6,1,0), (3,1,-1), (2,1,-1) \}$, so that $N_1 = 5\sqrt{6}$, while $N_2 = 6$, so we may take $N = 13$.
	
	The computations for $\Gamma_0(3\divides 3)$ proceed analogously. 
\end{example}

We use the quantity $c$ in Lemma \ref{Lemma:ImKtauBound} below, but first we verify that by excluding all $c(k,y,z) = 1$, we obtain $c < 1$. Taking $c$ to be at least $\frac{1}{2}$, as we have defined it, does no harm, and simplifies certain arguments.

\begin{lemma}\label{Lemma:cLessThan1}
	Let $\Gamma = \Gamma_0(mh\divides h)+$ with $m$ square-free. Then $\frac{1}{2} \leq c < 1$.
\end{lemma}
\begin{proof}
	Note that $\mc{U}$ is a finite set, as is $\mc{T}$, so all maxima above exist. Observe that for some $K\in\Gamma$ and $\tau\in\mc{T}$, we have \[ c(k,y,z) = \frac{k}{\abs{mhy\tau + kz}^2} = \frac{\frac{k}{(mhy)^2}}{\abs{\tau - \left(-\frac{kz}{mhy}\right)}^2} = \frac{\rho^2(K)}{\abs{\tau - \pi(K)}^2}. \] Since $\tau\in\mc{T}$ is in the closure of $\mc{D}(\Gamma)$, \[ \Im K\tau = \frac{\rho^2(K)}{\abs{\tau - \pi(K)}^2}\Im\tau \leq \Im\tau, \] so $c(k,y,z) \leq 1$ for all $(k,y,z)\in\mc{U}$. Thus $c < 1$.
\end{proof}

Indeed, each triple $(k,y,z)$ corresponds to some $\pi(K) = -\frac{kz}{mhy}$, defining some $[K]$. In the proof above, we see that $c(k,y,z) = 1$ if and only if $\frac{\rho^2(K)}{\abs{\tau-\pi(K)}^2} = 1$ for some choice of $\tau\in\mc{D}(\Gamma)$, implying that in fact the corresponding element $[K]\in\mc{K}$, since then $\Im K\tau = \Im\tau$. That is, the definition of $c$ in Def. \ref{Def:NcBound} is constructed to exclude the critical set, and the triples with $c(k,y,z)=1$ are precisely the non-identity elements of the critical set. The computation of $c$ therefore provides another opportunity to determine the critical set.

\begin{lemma}\label{Lemma:ImKtauBound}
	Let $K\in\Gamma = \Gamma_0(mh\divides h)+$, with $m$ square-free, and $c$ as Definition \ref{Def:NcBound}. For all $[K]\notin\mc{K}$ and all $\tau\in\mc{C}(\Gamma)$, $\Im K\tau \leq c\Im\tau$.
\end{lemma}
\begin{proof}
	For any $[K]\notin\mc{K}$, we know $\mc{A}_K\cap\mc{D}(\Gamma) = \emptyset$. Since $\mc{D}(\Gamma)$ is on or outside all critical arcs, $\mc{C}(\Gamma)$ is strictly outside $\mc{A}_K$. That is, $\Im K\tau < \Im\tau$, for any $\tau\in\mc{C}(\Gamma)$. Since $\Im K\tau = \frac{\rho^2(K)}{\abs{\tau - \pi(K)}^2}\Im\tau$, we must show that the maximum value of $\frac{\rho^2(K)}{\abs{\tau - \pi(K)}^2}$ for all $[K]\notin\mc{K}$ and $\tau\in\mc{C}(\Gamma)$ is given by $c$ above, or is bounded by $\frac{1}{2}$. 
	
	Note that for any $[K]$, the maximum value of $\Im K\tau$ along any arc $\mc{C}_i\subset \mc{C}(\Gamma)$ is achieved at an endpoint of $\mc{C}_i$, since $\abs{\tau - \pi(K)}$ is either constant or strictly monotonic $\tau$ moves along $\mc{C}_i$, depening on if $\mc{C}_i$ is on a circle centered at $\pi(K)$ or not. (If the distance between $\tau$ and $\pi(K)$ increases and then decreases along some arc of a circle $\mc{C}_i$, then this arc is tangent to a circle centered at $\pi(K)$ at the point of inflection, but two dinstinct circles with centers in $\qq$ may only be tangent in $\rr$, not in $\hh$.)
	
	Let $y_0 = \min\{ \Im\tau\st\tau\in\mc{C}(\Gamma) \}$ be the minimal imaginary part of any point on the lower boundary the fundamental domain $\mc{D}(\Gamma)$. Let $[K]\in\lmod{\Gamma_\infty}{\Gamma}$ with $[K]\notin\mc{K}$, so in particular $[K] \neq \Gamma_\infty$, and let $\tau\in\mc{C}(\Gamma)$. Certainly $\abs{\tau - \pi(K)}^2 \geq y_0^2$, since the distance between $\tau$ and the real axis is at least $y_0$, so $\rho^2(K) > \frac{1}{2}\abs{\tau - \pi(K)}^2$ only if $\rho^2(K) > \frac{y_0^2}{2}$. On the other hand, $\rho^2(K) \leq \frac{m}{(mh)^2} = \frac{1}{mh^2}$, since the largest radius is associated to the Fricke involution $w_m$. Thus, $\rho^2(K) > \frac{1}{2}\abs{\tau - \pi(K)}^2$ only if $\abs{\tau - \pi(K)}^2 < \frac{2}{mh^2}$. We deduce $\abs{\pi(K)} < \sqrt{\frac{2}{mh^2}} + \abs{\tau}$. But observe $\abs{\tau}^2 = (\Re\tau)^2 + (\Im\tau)^2 \leq \frac{1}{(2h)^2} + \frac{1}{mh^2} = \frac{m+4}{4mh^2}$, since $\tau\in\mc{C}$ implies $\Im\tau \leq \frac{1}{mh^2}$, the radius of the Fricke involution, and $-\frac{1}{2h} < \Re\tau\leq \frac{1}{2h}$. Thus, \[ \abs{\pi(K)} < \sqrt{\frac{2}{mh^2}} + \sqrt{\frac{m+4}{4mh^2}} = \frac{2\sqrt{2} + \sqrt{m+4}}{2h\sqrt{m}} =: B. \] 
	
	For any $K\in [K]$, a reduced coset representative $K = \begin{psmallmatrix} khw & x \\ mh^2y & khz \end{psmallmatrix}$ satisfies $\abs{\pi(K)} = \frac{k\abs{z}}{mhy}$, $\rho^2(K) = \frac{k}{(mhy)^2}$, with $(kz,\frac{m}{k}y) = 1$, and we note moreover $y > 0$ since $[K]\neq \Gamma_\infty$. Thus $\rho^2(K) = \frac{k}{(mhy)^2} > \frac{y_0^2}{2}$ only if $0 < y^2 < \frac{2k}{(mhy_0)^2}$, giving $0 < y < \frac{\sqrt{2k}}{mhy_0}$. Now $\abs{\pi(K)} = \frac{k\abs{z}}{mhy} < B$ only if \[ \abs{z} < \frac{Bmhy}{k} < \frac{Bmh}{k}\frac{\sqrt{2k}}{mhy_0} = \frac{B}{y_0}\sqrt{\frac{2}{k}} =  \frac{2\sqrt{2} + \sqrt{m+4}}{y_0h\sqrt{2km}}. \] Thus, if $[K]\notin\mc{K}$ is such that $\Im K\tau > \frac{1}{2}\Im\tau$, then $[K]$ is determined by some triple $(k,y,z)\in\mc{U}$ from Definition \ref{Def:NcBound}. Then $\max\frac{\rho^2(K)}{\abs{\tau-\pi(K)}^2}$ is achieved at some endpoint $\tau\in\mc{T}$, and thus given by $c(k,y,z)$, so that $c$ as defined does indeed guarantee $\Im K\tau \leq c\Im\tau$ for all $[K]\notin\mc{K}$ and all $\tau\in\mc{C}(\Gamma)$.
\end{proof}
\begin{corollary}\label{Corollary:LastBound}
	Let $\Gamma = \Gamma_0(mh\divides h){+}$ with $m$ square-free, and $n\geq N$ with $(n,h)=1$. For any $H\in\mc{H}_n$ and any $G\in\Gamma^{(H)}$, either \[ E_H(\tau,G) = \lambda_K e^{-2\pi inK\tau} \] for some $[K]\in\mc{K}\cap\phi_n^{-1}(H)$, where $\lambda_K = \lambda^{(\phi_n(K))}(\Dmat{n}K\phi_n(K)^{-1})^{-1}$, or else \[ \abs{E_H(\tau,G)} \leq e^{2\pi nc\Im\tau}. \]
\end{corollary}
\begin{proof}
	By Corollary \ref{Corollary:EstimateBounds}, we know that either $E_H(\tau,G) = E_H(\tau,\Dmat{n}KH^{-1})$ with $\phi_n(K)=H$ and $\pi(K) = \pi(GH)$, or else $\abs{E_H(\tau,G)} \leq e^{\pi n\Im\tau} \leq e^{2\pi nc\Im\tau}$, in which case we are done. In the former case, $\pi(K) = \pi(GH)$ does not uniquely determine some $K\in\Gamma$, but it does uniquely determine $[K]\in\LeftQuotientByStab{\Gamma}$, by Corollary \ref{Corollary:PiRhoOnCosets}. By Lemma \ref{Lemma:ImKtauBound} above, if $[K]\notin\mc{K}$, then $\Im K\tau \leq c\Im\tau$, and therefore \[ \abs{E_H(\tau,G)} = \abs{E_H(\tau,\Dmat{n}KH^{-1})} = e^{2\pi n\Im K\tau} \leq e^{2\pi nc\Im \tau}. \] Otherwise, $[K]\in\mc{K}$, and \[ E_H(\tau,G) = E_H(\tau,\Dmat{n}KH^{-1}) = \lambda^{(H)}(\Dmat{n}KH^{-1})^{-1}e^{-2\pi inK\tau}, \] as desired.
\end{proof}

In the above corollary, we have either bounded $E_H(\tau,G)$ by $e^{2\pi nc\Im\tau}$, or else expressed $E_H(\tau,G)$ entirely in terms of $[K]\in\mc{K}$. In this latter case, we must have that $H = \phi_n(K)$ for some $K\in\mc{K}$, so there are at most $\abs{\mc{K}}$ such terms $E_H(\tau,G)$ among the sum $\sum_{H\in\mc{H}_n} E_H(\tau,G)$. By taking $n$ sufficiently large, we may guarantee there are precisely $\abs{\mc{K}}$ terms.

\begin{lemma}\label{Lemma:PhinInjective}
	Let $\Gamma = \Gamma_0(mh\divides h)+$ with $m$ square-free, and let $\mc{K}$ be the critical set for $\Gamma$. For any $n > N$, $\varphi_n\restrict_{\mc{K}} :\mc{K}\to \mc{H}_n$ is injective.
\end{lemma}
\begin{proof}
	Let $K_i = \begin{psmallmatrix} k_ihw_i & x_i \\ mhy_i & k_ihz_i \end{psmallmatrix}$, for $1\leq i\leq \abs{\mc{K}}$, be a complete set of reduced coset representatives for $\mc{K}$, and recall \[ N \geq \max_{i,j\leq \abs{\mc{K}}} \left\{ \frac{mh}{(k_i,k_j)}\abs{k_iz_iy_j - k_jz_jy_i} \right\}. \] Let $i,j \leq \abs{\mc{K}}$ be such that $H = \varphi_n(K_i) = \varphi_n(K_j)$. By Lemma \ref{Lemma:GHKinverseExpression}, $\Dmat{n}K_iH^{-1}\in\Gamma^{(H)}$ and $\Dmat{n}K_jH^{-1}\in\Gamma^{(H)}$, so $\Dmat{n}K_iH^{-1}\left(\Dmat{n}K_jH^{-1}\right) = \Dmat{n}K_iK_j^{-1}\Dmat{n}^{-1} \in \Gamma^{(H)}$. A reduced coset representative for $\Dmat{n}K_iK_j^{-1}\Dmat{\frac{1}{n}}$ must have determinant $\ell\frac{h^2}{(a,h)^2}$ with $\ell$ square-free, and the square-free part of the determinant of $\Dmat{n}K_iK_j^{-1}\Dmat{\frac{1}{n}}$ is $\frac{k_ik_j}{(k_i,k_j)^2}$. We compute that the lower left entry of a reduced coset representative for $\Dmat{n}K_iK_j^{-1}\Dmat{\frac{1}{n}}$ is therefore \[ \frac{mh^2(k_iz_iy_j - k_jz_jy_i)}{n(k_i,k_j)(a,h)} = \frac{mh^2}{(mh,a)(h,a)}\left[\frac{(mh,a)(k_iz_iy_j - k_jz_jy_i)}{n(k_i,k_j)}\right], \] so that $n$ divides $\frac{(mh,a)(k_iz_iy_j - k_jz_jy_i)}{n(k_i,k_j)}$. But since $n > N \geq \frac{mh\abs{k_iz_iy_j - k_jz_jy_i}}{n(k_i,k_j)}$, we conclude that $\frac{(mh,a)(k_iz_iy_j - k_jz_jy_i)}{n(k_i,k_j)} = 0$. Since the lower left entry of $K_iK_j^{-1}$ is therefore zero, $K_iK_j^{-1}\in\Gamma_\infty$, so $[K_i] = [K_j]$, and $\varphi_n\restrict_{\mc{K}}$ is injective.
\end{proof}

In choosing $G\in\Gamma^{(H)}$ for $E_H(\tau,G)$, we asked that $GH\tau\in\mc{D}(\Gamma^{(H)})$. We are also insisting that $[K]\in\mc{K}$ satisfy $\Dmat{n}KH^{-1}\in [G]$ in $\LeftQuotientByStab{\Gamma^{(H)}}$, that is, we must be able to translate $(\Dmat{n}KH^{-1})H\tau = \Dmat{n}K\tau$ horizontally by some $\frac{r}{h}$ and land in $\mc{D}(\Gamma^{(H)})$. To ensure this is the case, we may guarantee that $\Dmat{n}K\tau$ is strictly above every (non-infinite) arc $\mc{A}(G)$ in $\Gamma^{(H)}$, for any possible $H$. Since the radius of $\mc{A}(M)$ for $M=\begin{psmallmatrix} a & b \\ c & d \end{psmallmatrix}\in\PGL_2^+(\qq)$ is $\rho(M) = \frac{\sqrt{\det(M)}}{c^2} \leq 1$, it suffices to ensure $\Im nK\tau > 1$.

\begin{lemma}\label{Lemma:nImKtauGreaterThan1}
	Let $\Gamma = \Gamma_0(mh\divides h)+$, let $\mc{K}$ be the critical set for $\Gamma$, and let $N$ be as above. Then for all $[K]\in\mc{K}$, all $n > N$ and all $\tau\in\mc{C}(\Gamma)$, $n\Im K\tau > 1$.
\end{lemma}
\begin{proof}
	It suffices to show that $\Im K\tau \geq \frac{1}{N}$ for all $[K]\in\mc{K}$ and all $\tau\in\mc{C}(\Gamma)$. First, suppose $[K] = [\eye]$, so $\Im K\tau = \Im\tau$. Since $\tau\in\mc{D}(\Gamma)$, for each $K_i$ with $2\leq i\leq\abs{\mc{K}}$, \[ \Im K_i\tau = \frac{\rho^2(K_i)}{\abs{\tau - \pi(K_i)}^2}\Im\tau \leq \Im\tau, \] and so $\frac{\abs{\tau-\pi(K_i)}^2}{\rho^2(K_i)}\geq 1$. Thus, $N_1 \geq \frac{1}{\Im\tau}$, so $\Im K\tau \geq \frac{1}{N}$.
	
	Now suppose $[K] \neq [\eye]$. We have \[ \Im K\tau \geq \frac{\Im\tau \rho^2(K)}{\abs{\tau - \pi(K)}^2}, \] with $\tau\in\mc{C}(\Gamma)$. We claim the minimum value of the right hand side must occur when $\tau\in\mc{T}$, equivalently, that the maximum value of $\abs{\tau-\pi(K)}^2$ occurs for $\tau\in\mc{T}$. From this, by definition of $N_1$, we again will have $\Im K\tau \geq \frac{1}{N}$. Observe that $\abs{\tau - \pi(K)}$ changes between increasing and decreasing only if $\tau$ is (instantaneously) traveling a path tangent to a circle centered at $\pi(K)\in\qq$. Since $\tau$ lies on some $\mc{C}_i$, a portion of a circle centered in $\qq$, the arc $\mc{C}_i$ can only be tangent to a circle centered at $\pi(K)$ if $\mc{C}_i$ is part of a circle centered at $\pi(K)$. In this case, $\abs{\tau-\pi(K)}$ is constant on all of $\mc{C}_i$, otherwise, $\abs{\tau-\pi(K)}$ is monotonic on each $\mc{C}_i$, so the extreme values occur at the endpoints, in $\mc{T}$. Thus, $\Im K\tau\geq \frac{1}{N}$.
\end{proof}

We may finally prove our main tool for bounding these modular functions.

\begin{theorem}\label{Theorem:MainBound}
	Let $f$ be a replicable function from Monstrous Moonshine having eigengroup $\Gamma = \Gamma_0(mh\divides h)+$, with $m$ square-free. Let \[ M = \max_{r\in\nn}\left\{ f^{(r)}(iy_0^{(r)}) - e^{2\pi y_0^{(r)}} \right\}, \] as in Lemma \ref{Lemma:FunctionBoundsOnFundDomain}, and let $N$ and $c$ be as in Definition \ref{Def:NcBound}. Then for all $\tau\in\mc{C}(\Gamma)$ and all $n \geq N$ with $(n, h) = 1$, \[ \abs{F_{n,f}(\tau) - \sum_{[K]\in\mc{K}} \lambda_Ke^{-2\pi inK\tau}} < (M + e^{2\pi cn\Im\tau})n^2, \] where $\lambda_K := \lambda^{(\varphi_n(K))}\left(\Dmat{n}K\varphi_n(K)^{-1}\right)^{-1}$, and $\lambda_Ke^{-2\pi inK\tau}$ is independent of choice of $K$.
\end{theorem}
\begin{proof}
	Recall as an intermediate result in the proof of Corollary \ref{Corollary:FirstBound}, we have \[ \bigabs{ F_{n,f}(\tau) - \sum_{H\in\mc{H}_n} E_H(\tau,G_H) } \leq Mn^2, \] where each $G_H$ is such that $G_HH\tau\in\mc{D}(\Gamma^{(H)})$. For each $H\in\mc{H}_n$, let $K_H\in\Gamma$ be such that $\pi(K_H)=\pi(G_HH)$. By Corollary \ref{Corollary:LastBound}, if $H\notin\phi_n(\mc{K})$, then $\abs{E_H(\tau,G_H)} \leq e^{2\pi nc\Im\tau}$, so that \[ \bigabs{\sum_{H\notin\phi_n(\mc{K})}} \leq \sum_{H\notin\phi_n(\mc{K})} e^{2\pi nc\Im\tau} \leq n^2e^{2\pi nc\Im\tau}. \] Since $n>N$, $\phi_n\restrict_{\mc{K}}$ injective, so $\phi_n(\mc{K})\subset\mc{H}_n$ is in bijection with $\mc{K}$. For every $[K]\in\mc{K}$, one has $\Im \Dmat{n}K\tau > 1$ by Lemma \ref{Lemma:nImKtauGreaterThan1}, and by the reasoning above that lemma, each $[K]\in\mc{K}$ corresponds, up to translation in $\Gamma^{(H)}_\infty$, with a choice of $G_H$ such that $G_HH\tau\in\mc{D}(\Gamma^{(H)})$. That is, \[ \sum_{H\in\phi_n(\mc{K})} E_H(\tau,G_H) = \sum_{[K]\in\mc{K}} \lambda_Ke^{-2\pi inK\tau}. \] Thus, \[ \bigabs{ F_{n,f}(\tau) - \sum_{[K]\in\mc{K}} \lambda_Ke^{-2\pi inK\tau} } \leq Mn^2 + \bigabs{ \sum_{H\notin\phi_n{\mc{K}}} E_H(\tau,G_H) } \leq (M+e^{2\pi nc\Im\tau})n^2. \]
\end{proof}

For any group $\Gamma$ appearing in the Monstrous Moonshine correspondence, we may compute $\mc{K}$, $N$, and $c$ entirely in terms of $\Gamma$. Unfortunately, at this point the root of unity $\lambda_K$ still depends on $\Gamma^{(H)}$, and we have no better way of determining $\lambda_K$ than choosing some representative $K$ for $[K]$, letting $H=\phi_n(K)$ and computing $\lambda^{(H)}(\Dmat{n}KH^{-1})^{-1}$.

\section{Proofs of Theorem 1}\label{Section:Cases}

We apply the method developed in Section \ref{Section:ZerosOfFaberPolynomials} to three specific cases, to demonstrate the method. We first present the case of $\Fricke{2}$, as this demonstrates the basic procedure. We then consider $\Gamma_0(6){+}$, which is interesting for having multiple arcs on the lower boundary, and then provide an example of how conjugation affects the procedure with $\Gamma_0(3\edivides 3)$. The other cases are similar to these three but we omit the details here (see \cite{Toomey})

We have additionally verified that this procedure locates zeros for the Fricke groups of level $3$, $5$, and $7$, as well as for $\Gamma_0(10){+}$. Moreover, where the Hauptmodul $T_{11A}$ for $\Fricke{11}$ is real-valued, this method locates a positive proportion of the zeros of the associated Faber polynomials (we expect the zeros located are all the real zeros).

In all cases, we are only able to approximate $F_{n,f}(\tau) = q^{-n} + O(q)$ for $n > N$, and must verify the zeros for $n\leq N$ separately. However, for the cases we consider in Theorem \ref{Theorem:MainTheorem}, the zeros for $n<200$ were located by Shigezumi (\cite{SHI}), and one may compute $N<200$ for each group, so that these manual calcuations have already been done.

\subsection{The case \texorpdfstring{$\Gamma = \Fricke{2}$}{of the Fricke group of level 2}}

We consider the case of the group \[ \Gamma = \Fricke{2} = \left\{\begin{psmallmatrix} kw & x \\ my & kz \end{psmallmatrix} \bigst kwz - \frac{m}{k}xy = 1 \right\}. \] This is a genus zero group having normalized Hauptmodul (\cite{CN}) \[ T_{2A}(q) = q^{-1} + 4372q + 96256q^2 + 1240002q^3 + \ldots, \] where $\eta(q) = q^{\frac{1}{24}}\prod_{n=1}^{\infty} (1-q^n)$ is the Dedekind eta function. We have that \[ \mc{C} = \mc{C}(\Fricke{2}) = \left\{\tau\in\hh \st \abs{\tau}^2 = \frac{1}{2}, 0\leq \Re\tau \leq \frac{1}{2} \right\}, \] the lower boundary of our fundamental domin \[ \mc{D} = \mc{D}(\Fricke{2}) = \mc{C}\cup\left\{\tau\in\hh \st \abs{\tau} > \frac{1}{2}, 0< \Re\tau\leq \frac{1}{2} \right\}. \] The function $T_{2A}$ is real-valued on $\mc{C}$, taking values in the interval $[-104, 152]$ (\cite{SHI}), so we may use the intermediate value theorem to count zeros along this arc.

Let $W_2 = \begin{psmallmatrix} 0 & -1 \\ 2 & 0 \end{psmallmatrix}\in\Gamma$, and let $K_2 = \begin{psmallmatrix} -1 & 0 \\ 2 & -1 \end{psmallmatrix}\in\Gamma$. We calculate the critical set using Lemma \ref{Lemma:CriticalSetIsFinite}, and obtain \[ \mc{K} = \left\{ [T], [W_2], [W_2T^{-1}], [K_2] \right\}, \] corresponding to the `infinite arc' $\hh$, and the three arcs centered at $0$, $1$, and $\frac{1}{2}$, which intersect $\mc{C}$. As representatives for the elements of $\mc{K}$, we take \[ \left\{ \begin{psmallmatrix} 1 & 0 \\ 0 & 1 \end{psmallmatrix}, \begin{psmallmatrix} 0 & -1 \\ 2 & 0 \end{psmallmatrix}, \begin{psmallmatrix} 0 & -1 \\ 2 & -2 \end{psmallmatrix}, \begin{psmallmatrix} -1 & 0 \\ 2 & -1 \end{psmallmatrix} \right\}, \] respectively, so that \[ \sum_{[K]\in\mc{K}} \lambda_Ke^{-2\pi inK\tau} 
= e^{-2\pi in\tau} + e^{-2\pi in\frac{-1}{2\tau}} + e^{-2\pi in\frac{-1}{2\tau-2}} +e^{-2\pi in\frac{-\tau}{2\tau-1}}. \] Since $\Gamma = \Fricke{2}$, all replicates of $\Gamma$ have cusp width 1 at infinity, and so necessarily $\lambda_K = 1$ for each $K\in\mc{K}$. When restricting to $\tau = x+iy\in\mc{C}$, where $\abs{\tau}^2=\frac{1}{2}$, we compute that \[ \abs{2\tau-2}^2 = (2x-2)^2+(2y)^2 = \abs{2\tau}^2 - 8x + 4 = 6-8x, \] and similarly $\abs{2\tau - 1}^2 = 3-4x$, so that 
\begin{align*}
	\sum_{[K]\in\mc{K}} \lambda_Ke^{-2\pi inK\tau} 
	&= e^{-2\pi in\tau} + e^{2\pi in\overline{\tau}} + e^{2\pi in\frac{2\overline{\tau}-2}{6-8x}} +e^{2\pi in\frac{\tau(2\overline{\tau}-1)}{3-4x}} \\
	&= e^{2\pi ny}\left( e^{-2\pi inx} + e^{2\pi inx}\right) + e^{2\pi in\frac{\overline{\tau}-1}{3-4x}} +e^{2\pi in\frac{1-\tau}{3-4x}} \\
	&= 2e^{2\pi ny}\cos\left(2\pi nx\right) + 2e^{2\pi n\frac{y}{3-4x}}\cos\left(\frac{2\pi(x-1)}{3-4x}\right).
\end{align*}

In Example \ref{Example:NcComputation}, we calculate $c=\frac{1}{2}$, and $N = 3\sqrt{2}$, so we take $N\geq 5$. In the appendix, we compute that $T_{1A}\left(\frac{i\sqrt{3}}{2}\right) = 4\cdot 15^3(30-17\sqrt{3})^3 - 744$ and $T_{2A}\left(\frac{i}{2}\right) = 544$, so that in Theorem \ref{Theorem:MainTheorem}, we have \[ M = \max\left\{ 4\cdot 15^3(30-17\sqrt{3})^3 - 744 - e^{\pi\sqrt{3}}, 544 - e^{\pi} \right\} \approx 1334.813\ldots < 1335, \] which yields the bound \[ \bigabs{ F_{n}(\tau) - \sum_{K\in\mc{K}} e^{2\pi nK\tau} } \leq (1335 + e^{\pi n\Im\tau})n^2. \] Moving the two terms corresponding to $[K] = [W_2T^{-1}], [K_2]$ to the right hand side, restricting to $\tau = x+iy\in\mc{C}$, and multiplying through by $e^{-2\pi ny}$, we establish the bound 
\begin{align*}
	\bigabs{ F_{n}(\tau)e^{-2\pi ny} - 2\cos(2\pi nx) } 
	&\leq \left[ 1335n^2 + n^2e^{\pi ny} + \abs{2e^{2\pi n\frac{y}{3-4x}}\cos\left(2\pi\frac{x-1}{3-4x}\right)} \right]e^{-2\pi ny} \\
	&= \left[ 1335n^2 + n^2e^{\pi ny}\right]e^{-2\pi ny} + 2e^{2\pi ny\frac{4x-2}{3-4x}}, \stepcounter{equation}\tag{\theequation}\label{2Abound}
\end{align*}
since cosine is bounded by $1$ for real inputs. We analyze the bound in \ref{2Abound} in two parts.

Observe that \[ \frac{\partial}{\partial y} \left[ 1335n^2 + n^2e^{\pi ny}\right]e^{-2\pi ny} = -\pi n^3\left[e^{-\pi ny} + e^{-2\pi ny}\right] < 0 \] for all $y>0$, so is maximal on $\mc{C}$ when $y = \frac{1}{2}$. Making this substitution, we then find \[ \frac{\partial}{\partial n} \left[ 1335n^2 + n^2e^{\frac{\pi}{2} n}\right]e^{-\pi n} = \left[2-\pi n\right] Mne^{-\pi n} + \left[2 - \frac{pi}{2} n\right] ne^{-\frac{\pi}{2} n} \] is negative for $n > \frac{4}{\pi}$. Thus, for any $n\geq N=5$ and any $\tau\in\mc{C}$, the first term in \ref{2Abound} is bounded, with \[ \left[ 1335n^2 + n^2e^{\pi ny}\right]e^{-2\pi ny} \leq \left[ 1335(5)^2 + (5)^2e^{\frac{5\pi}{2}}\right]e^{-5\pi} \approx 0.01474 < 0.015. \]

\begin{remark}
	The above has shown that $F_n(\tau)$ is very well approximated by $\sum_{[K]\in\mc{K}} \lambda_K e^{-2\pi inK\tau}$ along the lower arc $\mc{C}$. In fact, along `most' of $\mc{C}$, only two terms of the four in this sum are needed to approximate $F_n(\tau)$. Heuristically, a term $\lambda_K e^{-2\pi inK\tau}$ can only be of significant magnitude when $\tau$ is `near' $\mc{A}(K)$. The lower boundary $\mc{C}$ is a segment of $\mc{A}(W_2)$, where $W_2 = \begin{psmallmatrix} 0 & -1 \\ 2 & 0 \end{psmallmatrix}$), and also a subset of $\mc{A}(\eye) = \hh$, so the two terms associated to $K = \eye, W_2$ always contribute significantly to the behavior of $F_n(\tau)$. The arcs associated to the remaining two elements of the critical set intersect $\mc{C}$ at the point $\tau = \frac{1+i}{2}$, and so make a significant contribution to the behavior of $F_n(\tau)$ only near this point.
\end{remark}

We would now like to bound the second term in \ref{2Abound} on the lower arc $\mc{C}$. Unfortunately, at the point $x+iy = \frac{1}{2}+\frac{i}{2} \in \mc{C}$, one has $2e^{2\pi ny\frac{4x-2}{3-4x}} = 2$, and the resulting bound on $F_{n}(\tau)e^{-2\pi ny}$ would not be effective for locating zeros using the intermediate value theorem (the function is not bounded within $2$ of $2\cos(2\pi nx)$, so we cannot guarantee any sign changes). To produce a better bound, we must avoid this problem point on the lower boundary, but as $n$ increases, we expect to find a zero of $F_n(\tau)$ arbitrarily close to $\frac{1+i}{2}$. To handle this difficulty, for any $n\in\nn$ we define \[ \mc{C}_n = \left\{\tau\in\mc{C} \st 0\leq \Re\tau \leq \frac{6n-1}{12n} \right\}, \] a region which avoids $\tau = \frac{1+i}{2}$, but where one may check that $2\cos(2\pi n\Re\tau)$ still changes sign $n+1$ times.

For $\tau = x+iy\in\mc{C}$, $y=\sqrt{\frac{1}{2}-x^2}$, so that \[ 2e^{2\pi ny\frac{4x-2}{3-4x}} = 2e^{2\pi n\frac{4x-2}{3-4x}\sqrt{\frac{1}{2}-x^2}}, \] and we note that $\frac{4x-2}{3-4x}\sqrt{\frac{1}{2}-x^2}$ is a negative and increasing function of $x$ for $x\in [0,\frac{1}{2})$, so that $2e^{2\pi n\frac{4x-2}{3-4x}\sqrt{\frac{1}{2}-x^2}}$ is increasing and obtains its maximum value on $\mc{C}_n$ when $x = \frac{6n-1}{12n}$. Substituting this value of $x$ in and simplifying, we find that for $\tau\in\mc{C}_n$, we have \[ 2e^{2\pi ny\frac{4x-2}{3-4x}} \leq 2e^{-\frac{\pi}{3n+1}\sqrt{(n+\frac{1}{6})^2-\frac{1}{18}}}, \] and this bound is a decreasing function of $n$, so that for $n\geq N=5$, \[ 2e^{2\pi ny\frac{4x-2}{3-4x}} \leq 2e^{-\frac{\pi}{16}\sqrt{(\frac{31}{6})^2-\frac{1}{18}}} \approx 0.72595 < 0.726. \]

Combining the two bounds above, for all $n\geq 5$ and all $\tau\in\mc{C}_n$, we find \[\bigabs{ F_{n}(\tau)e^{-2\pi ny} - 2\cos(2\pi nx) } < 0.015 + 0.726 = 0.741 < \sqrt{3}. \] For $0\leq k<n$, define $x_k = \frac{k}{2n}$, and define $x_n = \frac{6n-1}{12n}$, noting that for $0\leq k\leq n$, the point $\tau_k := x_k + i\sqrt{\frac{1}{2}-x_k^2} \in \mc{C}_n$. For $0\leq k <n$, we compute $2\cos(2\pi nx_k) = 2\cos(\pi k) = 2(-1)^k$, and for $2\cos(2\pi nx_n) = 2\cos(\pi n - \frac{\pi}{6}) = \sqrt{3}(-1)^n$. Since $F_{n}(\tau)e^{-2\pi ny}$ is real-valued and continuous on $\mc{C}_n$, and stays within $0.726 < \sqrt{3}$ of $2\cos(2\pi nx)$, we deduce that $F_{n}(\tau)e^{-2\pi ny}$  must change sign $n+1$ times on $\mc{C}_n$, and therefore has $n$ simple zeros on (the interior of) $\mc{C}_n\subset\mc{C}$. Equivalently, by the definition of the Faber polynomials, we have found that $F_n(X)$ has $n$ simple zeros on the interval $(-104,152)$, since $T_{2A}(\tau)$ takes values in $[-104,152]$ on $\mc{C}$, and is nonzero at the endpoints.

\subsection{The case \texorpdfstring{$\Gamma = \Gamma_0(6){+}$}{of the normalizer of the Hecke congruence subgroup of level 6}}

We now consider a case of $\Gamma = \Gamma_0(6){+}$, where the lower boundary consists of more than one arc. Let
\begin{align*} 
	\mc{C}_1 
	&= \left\{\tau\in\hh\st \abs{\tau}^2 = \frac{1}{6}, 0\leq \Re\tau \leq \frac{1}{3} \right\}, \\
	\mc{C}_2
	&= \left\{\tau\in\hh\st \abs{\tau-\frac{1}{2}}^2 = \frac{1}{12}, \frac{1}{3}\leq \Re\tau \leq \frac{1}{2} \right\},
\end{align*}
then $\mc{C} = \mc{C}_1\cup\mc{C}_2$ is the lower boundary of $\mc{D}(\Gamma)$ (see Figure \ref{Fig:FundamentalDomainGamma06+}). The normalized Hauptmodul is \[ T_{6A}(\tau) = q^{-1} + 79q + 352q^2 + 1431q^3 + 4160q^4 + 13015q^5 + \ldots, \] which is a completely replicable function appearing in Monstrous Moonshine with replicates \[ T_{6A}^{(a)} = \begin{cases} T_{6A} & a\equiv 1\pmod{6} \\ T_{3A} & a\equiv 2\pmod{6} \\ T_{2A} & a\equiv 3\pmod{6} \\ T_{1A} & a\equiv 0\pmod{6}, \end{cases} \] where $T_{1A} = j - 744$ is the normalized Hauptmodul for $\PSL_2(\zz)$, $T_{2A}$ is the normalized Hauptmodul for $\Fricke{2}$, and $T_{3A}$ is the normalized Hauptmodul for $\Fricke{3}$ \cite{CN} (see \S\ref{Section:SpecialValues} for more on the function $T_{3A}$).

The critical set for $\Gamma_0(6){+}$ is given by coset representatives \[ \mc{K} = \{ \eye, \begin{psmallmatrix} 0 & -1 \\ 6 & 0 \end{psmallmatrix}, \begin{psmallmatrix} 3 & -2 \\ 6 & -3 \end{psmallmatrix}, \begin{psmallmatrix} 2 & -1 \\ 6 & -2 \end{psmallmatrix} \}, \] while $N = 13$ and $c=\frac{1}{2}$ (\ref{Example:CriticalSets}, \ref{Example:NcComputation}). In Section \ref{Section:SpecialValues}, we compute the values of $T_{6A}^{(a)}(y_0^{(a)})$ for all replicates $T_{6A}^{(a)}$, so that 
\begin{align*}
	B 
	&= \max_{a\in\nn} \{ T_{6A}^{(a)}(y_0^{(a)}) - e^{2\pi y_0^{(a)}} \} \\
	&= \max\{ 1565.577 - e^{\pi\sqrt{3}}, 544 - e^{\pi}, 1416 - e^{\frac{\pi}{\sqrt{3}}}, 86 - e^{\frac{\sqrt{2}\pi}{3}} \} \\
	&= 1416 - e^{\frac{\pi}{\sqrt{3}}} \\
	&\approx 1409.866\ldots,
\end{align*}
and we take $B = 1410$. By Theorem \ref{Theorem:MainBound}, we have \[ \bigabs{ F_{n,6A}(\tau) - \sum_{K\in\mc{K}} e^{-2\pi inK\tau} }e^{-2\pi ny} < 1410n^2e^{-2\pi ny} + n^2e^{-\pi ny} \] for all $\tau=x+iy\in\mc{C}$ and all $N \geq 13$. The right hand side above is a decreasing function of $y$, so bounded above at $y_0 = \frac{1}{3\sqrt{2}}$. One checks that $1410n^2e^{-2\pi ny_0} + n^2e^{-\pi ny_0}$ is decreasing for $n\geq 2$, so that for $N\geq 13$, \[ \bigabs{ F_{n,6A}(\tau) - \sum_{K\in\mc{K}} e^{-2\pi inK\tau} }e^{-2\pi ny} < 1410(13^2)e^{-\frac{13\sqrt{2}\pi}{3}} + 13^2e^{-\frac{13\sqrt{2}\pi}{6}} \approx 0.01219\ldots < 0.013. \]

Now, as in the case of $\Fricke{2}$, we write $\sum_{K\in\mc{K}} e^{-2\pi inK\tau}$ in terms of the real and imaginary part of $\tau$. We use the identity $\abs{\tau}^2 = \frac{1}{6}$ for $\tau = x+iy\in\mc{C}_1$, and $\abs{\tau-\frac{1}{2}}^2 = \frac{1}{12}$ for $\tau = x+iy\in\mc{C}_2$ to determine (after considerable algebra) that \[ \sum_{K\in\mc{K}} e^{-2\pi inK\tau} = 2e^{2\pi ny}\cos(2\pi nx) + \begin{cases} 2e^{2\pi n\frac{y}{5-12x}}\cos(2\pi n\frac{3-7x}{5-12x}) & \tau\in\mc{C}_1 \\ 2e^{2\pi n\frac{y}{6x-1}}\cos(2\pi n\frac{x}{6x-1}) & \tau\in\mc{C}_2. \end{cases} \] Thus, we find for $\tau = x+iy\in\mc{C}$ and $n\geq 13$ that \[ \bigabs{F_{n,6A}(\tau)e^{-2\pi iny} - 2\cos(2\pi nx)} < 0.013 + \begin{cases} 2e^{-2\pi ny\frac{4-12x}{5-12x}} & \tau\in\mc{C}_1 \\ 2e^{-2\pi ny\frac{6x-2}{6x-1}} & \tau\in\mc{C}_2. \end{cases} \]

\begin{remark}
	The term $2e^{2\pi ny}\cos(2\pi nx)$ arises from different pairs of terms depending on the arc. On $\mc{C}_1$, we find $2e^{2\pi ny}\cos(2\pi nx) = e^{2\pi in\tau} + e^{2\pi in\frac{-1}{6\tau}}$, where these exponentials come from $\begin{psmallmatrix} 1 & 0 \\ 0 & 1 \end{psmallmatrix}$ and $\begin{psmallmatrix} 0 & -1 \\ 6 & 0 \end{psmallmatrix}$ in the critical set. On $\mc{C}_2$, the term $2e^{2\pi ny}\cos(2\pi nx)$ arises from the exponentials coming from the identity matrix and $\begin{psmallmatrix} 3 & -2 \\ 6 & -3 \end{psmallmatrix}$. In both cases, the non-identity matrix is the element of the critical set associated to the arc $\mc{C}_i$ containing $\tau$, and we find this holds generally in other cases, at least where the arc has center in $\frac{1}{2}\zz$.
\end{remark}

To locate zeros of $F_{n,6A}(\tau)$, we must exclude a small region near the elliptic point $\tau = \frac{1}{3} + \frac{i}{3\sqrt{2}}$. We thus define \[ \mc{C}_{1}^* = \left\{ \tau\in\mc{C}_1 \st 0\leq \Re\tau \leq \frac{1}{3} - \frac{1}{6n} \right\}, \] \[ \mc{C}_{2}^* = \left\{ \tau\in\mc{C}_2 \st \frac{1}{3} + \frac{1}{6n} \leq \Re\tau \leq \frac{1}{2} \right\}, \] depending on $n$, and let $\mc{C}^* = \mc{C}_1^*\cup\mc{C}_2^*$. One checks that $e^{-2\pi ny\frac{4-12x}{5-12x}}$ is an increasing function of $x$ and a decreasing function of $y$ on $\mc{C}_1$, so bounded on $\mc{C}_1^*$ by taking the right endpoint of $\mc{C}_1^*$, where $x = \frac{1}{3} - \frac{1}{6n}$, and $y>0$ is defined by $x^2+y^2=\frac{1}{6}$. By an entirely similar argument, we may bound $e^{-2\pi ny\frac{6x-2}{6x-1}}$ using the left endpoint of $\mc{C}_2^*$, where $x= \frac{1}{3}+\frac{1}{6n}$, and $y>0$ is defined by $(x-\frac{1}{2})^2+y^2 = \frac{1}{12}$. Substituting these values for $x$ and $y$, we obtain expressions in terms of $n$, \[ e^{-2\pi ny\frac{4-12x}{5-12x}} = e^{-\frac{2\pi\sqrt{2n^{2}+4n-1}}{3\left(n+2\right)}}, \]  \[ e^{-2\pi ny\frac{6x-2}{6x-1}} = e^{-\frac{\pi\sqrt{2x^{2}+2x-1}}{3\left(x+1\right)}}, \] both of which are decreasing functions of $n$ for $n\geq 1$ (indeed, for $n > -\frac{1}{2} + \frac{\sqrt{3}}{2}$ in the former case, and for $n > -1 +\sqrt{\frac{3}{2}}$ in the latter). Substituting $n=13$, we obtain bounds
\[ e^{-2\pi ny\frac{4-12x}{5-12x}} \leq 0.063681\ldots < 0.064 \qquad \text{on } \mc{C}_1^*, \] and \[ e^{-2\pi ny\frac{6x-2}{6x-1}}\leq 0.24047\ldots < 0.25 \qquad \text{on } \mc{C}_2^*, \] implying for $\tau\in\mc{C}*$ and $n\geq 13$ that \[ \bigabs{F_{n,6A}(\tau)e^{-2\pi iny} - 2\cos(2\pi nx)} < 0.013 + 0.25 = 0.263. \] 

We now use this approximation to locate all zeros of $F_{n,6A}$. We define points $\tau_k = x_k+iy_k\in\mc{C}$ by $x_k = \frac{k}{2n}$. When $k\neq\frac{2n}{3}$, we note that $\tau_k\in\mc{C}^*$ and $2\cos(2\pi nx_k) = 2\cos(\pi k) = \pm 2$. 

For $n$ not divisible by 3, then, our approximation forces $F_{n,6A}$ to have the same sign as cosine at all $n+1$ points $\tau_k$ on $\mc{C}$. Since $F_{n,6A}$ is real-valued on $\mc{C}$, by the intermediate value theorem these $n+1$ sign changes give us $n$ simple zeros for $F_{n,6A}$ on $\mc{C}$. when $3\ndivides n$, and thus forcing $F_{n,6A}$ to have $n$ zeros on the lower boundary of $\mc{D}(\Gamma)$.

Now, when $3\divides n$, let $k^* = \frac{2n}{3}$, and note we know the sign of $F_{n,6A}$ at all points $\tau_k\in\mc{C}$ except $\tau_k^* = \frac{1}{3}+\frac{i}{3\sqrt{2}}$. In particular, at the points $\tau_k$ neighboring $\tau_{k^*}$ (which are on $\mc{C}^*$), we have \[ 2\cos(2\pi nx_{k^*\pm 1}) = 2\cos\left(2\pi n\left(\frac{1}{3}\pm \frac{1}{2n}\right)\right) = 2\cos(2\pi\frac{n}{3}\pm \pi) = -2, \] so $F_{n,6A}$ is negative. However, at the (inner) endpoints of $\mc{C}_1^*$ and $\mc{C}_2^*$, we compute \[ 2\cos\left(2\pi n\left(\frac{1}{3}\pm\frac{1}{6n}\right)\right) = 2\cos\left(\frac{2\pi n}{3} \pm \frac{\pi}{3}\right) = 1, \] so $T_{n,6A}$ is positive at these endpoints of $\mc{C}_1^*$ and $\mc{C}_2^*$. Thus, we find that $F_{n,6A}$ changes sign $n+1$ times on $\mc{C}^*\subset\mc{C}$, and by the intermediate value theorem, we once again locate $n$ zeros on the lower arcs of $\mc{D}{\Gamma}$, as desired.

\subsection{The case of \texorpdfstring{$\Gamma = \Gamma_0(3\edivides 3)$}{of the normalizer of the Hecke congruence subgroup of level 9}}\label{Subsection:3C}

Finally, we provide an example of a group $\Gamma_0(mh\edivides h)+$ with $h>1$, the group $\Gamma = \Gamma_0(3\edivides 3)$. This group is a normal subgroup of \[ \Gamma_0(3\divides 3) = \begin{psmallmatrix} 1 & 0 \\ 0 & 3 \end{psmallmatrix}\PSL_2(\zz)\begin{psmallmatrix} 3 & 0 \\ 0 & 1 \end{psmallmatrix} = \left\{ \begin{psmallmatrix} 3w & x \\ 9y & 3z \end{psmallmatrix} \bigst wz-xy = 1 \right\}, \] which the kernel of a homomorphism $\lambda:\Gamma_0(3\divides 3)\to\cc$ obeying $\lambda\left(T^{\frac{1}{3}}\right) = e^{-\frac{2\pi i}{3}}$ (\cite{CMS}). This group is genus zero with normalized Hauptmodul \cite{CN} \[ T_{3C}(q) = \sqrt[3]{T_{1A}(q^3)+744} = q^{-1} + 248q^2 + 4124q^5 + 34752q^8 + \ldots. \] The lower boundary of the fundamental domain $\mc{D}(\Gamma)$ consists of three arcs, each a portion of circles with radius $\frac{1}{3}$, centered at $x=0$, $x=-\frac{1}{3}$, and $x=\frac{1}{3}$ on the real axis. (See Figure \ref{Fig:FundamentalDomainGamma03bar3}, which gives a fundamental domain for $\Gamma_0(3\divides 3)$; the fundamental domain for $\Gamma$ consists of three copies of this fundamental domain laid side by side.) Using the relation between $T_{3C}$ and $T_{1A}$, we find that $T_{3C}$ takes values in the interval $[0,12]$ on \[ \mc{C} = \left\{ \tau\in\hh \st \abs{\tau} = \frac{1}{3}, 0\leq \Re\tau \leq \frac{1}{3} \right\}, \] and by translation by $T^{\frac{1}{3}}$, the values of $T_{3C}$ on the rest of the lower boundary are given by the line segments in $\cc$ from $0$ to $12e^{\pm\frac{2\pi i}{3}}$ (that is, the interval $[0,12]$ multiplied by $\lambda(T^{\pm\frac{1}{3}})$, so rotated $\mp \frac{2\pi}{3}$ in the complex plane).

Since $T_{3C}(q)^3 = T_{1A}(q^3) + 744$, Lemma \ref{Lemma:ConjugateGroupFaberZeros} implies that whenever $3\divides n$, we have \[ F_{n,3C}(X) = F_{\frac{n}{3},1A}(X^3-744), \] so that if $F_{\frac{n}{3},1A}(Y) = 0$, then the third roots of $Y+744$ are the zeros of $F_{n,3C}$. Recall Asai, Kaneko, and Ninomiya \cite{AKN} located the $\frac{n}{3}$ zeros of $F_{\frac{n}{3},1A}(X)$ on the interval $[-744,984]$, implying the zeros of $F_{n,3C}(X)$ are all third roots of values in $[0,12]$. Since $T_{3C}$ takes these values on the lower arcs of $\mc{D}(\Gamma)$, all of the zeros of $F_{n,3C}(\tau)$ are on the lower boundary whenever $3\divides n$.

When $3\ndivides n$, write $n = 3s+t$ with $t\in\{1,2\}$, then Lemma \ref{Lemma:ConjugateGroupMoreFaberZeros} gives us that \[ F_{n,3C}(X) = X^t g(X^3), \] where $g(X)$ is a polynomial of degree $s$, so that $F_{n,3C}(\tau)$ has a zero of order $t$ at $T_{3C}(\tau) = 0$, which occurs on the lower arcs of $\mc{D}(\Gamma)$, at the point $\tau = \frac{1}{6}+\frac{\sqrt{3}}{6}$, as well as its horizontal translates by $\pm\frac{1}{3}$. These points are all elliptic points of order 3, so the $3t$ zeros at these three points account for $t$ zeros of the function $F_{n,3C}(\tau)$. Moreover, by showing $F_{n,3C}(\tau)$ has $s$ zeros on the \textit{interior} of $\mc{C}$ (which contains no elliptic points) we will locate $3s$ additional zeros on the lower boundary of $\mc{D}(\Gamma)$, and therefore all $n=3s+t$ zeros of $F_{n,3C}(\tau)$ in the fundamental domain.

To locate these $s$ zeros, we must approximate $F_{n,3C}(\tau)$ using the same methods as the previous cases. Recall we computed the critical set in Example \ref{Example:CriticalSets} to be \[ \mc{K} = \left\{ \begin{psmallmatrix} 1 & 0 \\ 0 & 1 \end{psmallmatrix}, \begin{psmallmatrix} 0 & -1 \\ 9 & 0 \end{psmallmatrix}, \begin{psmallmatrix} 0 & -1 \\ 9 & -3 \end{psmallmatrix} \right\}, \] and we found in Example \ref{Example:NcComputation} that $N = 9$ and $c=\frac{1}{2}$. In Section \ref{Section:SpecialValues}, we compute that \[ T_{3C}\left(i\frac{\sqrt{3}}{6}\right) = 15\sqrt[3]{4}(30-17\sqrt{3}) \\ T_{1A}\left(i\frac{\sqrt{3}}{2}\right) = 4\cdot 15^3(30-17\sqrt{3})^3 - 744, \] so that \[ M = max\left\{ T_{3C}(i\frac{\sqrt{3}}{6}) - e^{\frac{\pi}{\sqrt{3}}}, T_{1A}(i\frac{\sqrt{3}}{2}) - e^{\pi\sqrt{3}} \right\} \approx 1334.813\ldots < 1335. \] We therefore obtain by Theorem \ref{Theorem:MainBound} the approximation \[ \abs{F_{n,3C}(\tau) - \sum_{K\in\mc{K}}\lambda_K e^{-2\pi inK\tau} } < (1335 + e^{\pi n\Im\tau})n^2. \] 

We now determine the value of $\lambda_K$ for $K = \begin{psmallmatrix} 0 & -1 \\ 9 & 0 \end{psmallmatrix} = \Dmat{\frac{1}{9}}S$. From the description of $\lambda$ in \cite{CMS}, we have $\lambda(T^{\frac{1}{3}}) = e^{-\frac{2\pi i}{3}}$ and $\lambda\left(\begin{psmallmatrix} 3 & 0 \\ 9 & 3 \end{psmallmatrix}\right) = e^{\frac{2\pi i}{3}}$. Noting $\begin{psmallmatrix} 3 & 0 \\ 9 & 3 \end{psmallmatrix} = T^{\frac{1}{3}}\Dmat{\frac{1}{9}}ST^{\frac{1}{3}} = T^{\frac{1}{3}}\begin{psmallmatrix} 0 & -1 \\ 9 & 0 \end{psmallmatrix}T^{\frac{1}{3}}$, and using the fact $\lambda$ is a homomorphism, we find $\lambda_K = 1$. Since for $\tau\in\mc{C}$ we have  $\abs{\tau}^2 = \frac{1}{9}$, we compute $K\tau = \frac{-1}{9\tau} = -\overline{\tau}$. The corresponding terms in our approximation from Theorem \ref{Theorem:MainBound} become \[ e^{-2\pi in\tau} + e^{2\pi in\overline{\tau}} = 2e^{2\pi n\Im\tau}\cos(2\pi n\Re\tau). \]

Further, for $K = \begin{psmallmatrix} 0 & -1 \\ 9 & -3 \end{psmallmatrix}$, we find that \[ \abs{e^{-2\pi inK\tau}} = e^{2\pi n\Im K\tau} \leq e^{2\pi n\Im\tau}, \] since $\tau\in\mc{D}(\Gamma)$ implies $\Im K\tau \leq \Im\tau$ for all $K\in\Gamma$. 

Combining the above, we deduce that for $\tau = x+iy\in\mc{C}$ and $n\geq 9$, we have \[ \abs{F_{n,3C}e^{-2\pi ny} - 2\cos(2\pi nx)} < 1335n^2e^{-2\pi ny} + n^2e^{-\pi ny} + 1. \] The right hand side is a decreasing function of $y$ obtaining a minimum on $\mc{C}$ at $y=\frac{\sqrt{3}}{6}$, and substituting this value in, one checks the resulting expression is a decreasing function of $n$ for $n\geq 9$ (the expression is in fact decreasing for $n\geq 8$). For all $\tau\in\mc{C}$ and all $n\geq 9$, then, we obtain the bound \[ \abs{F_{n,3C}e^{-2\pi ny} - 2\cos(2\pi nx)} < 1335\cdot 9^2e^{-3\pi\sqrt{3}} + 9^2e^{-\frac{3}{2}\pi \sqrt{3}} + 1 \approx 1.0319\ldots < 1.04. \] Since $F_{n,3C}$ stays within a distance of 2 of $2\cos(2\pi nx)$, we must have the same $s+1$ sign changes $2\cos(2\pi nx)$ for $x\in [0,\frac{1}{3}]$, giving $s$ zeros of $F_{n,3C}$ on the interior of $\mc{C}$. Thus $F_{n,3C}$ has $n = 3s+t$ zeros on the lower boundary of $\mc{D}(\Gamma)$, as desired.	

\section{Appendix I: Special values of modular functions}\label{Section:SpecialValues}
In the method of the proof of Theorem \ref{Theorem:MainTheorem}, establishing the bound \[ B = \max\{ f^{(H)}(iy_0^{(H)}) - e^{-2\pi y_0^{(H)}} \st H\in\mc{H}_n \} \] in Corollary \ref{Corollary:FirstBound}, one must compute values of modular functions at (quadratic) purely imaginary points $iy_0$. It suffices to approximate these values using a computer calculation (the $q$-series converges exponentially to a real value), but we may also determine these values exactly. For the cases covered here, we need the values $T_{1A}\left(\frac{i\sqrt{3}}{2}\right)$, $T_{2A}\left(\frac{i}{2}\right)$, $T_{3A}\left(\frac{i}{2\sqrt{3}}\right)$, $T_{6A}\left(\frac{i}{3\sqrt{2}}\right)$, and $T_{3C}\left(\frac{i\sqrt{3}}{6}\right)$.

First, we compute $T_{1A}\left(\frac{i\sqrt{3}}{2}\right)$. We use the known special value of the $j$-invariant at $i\frac{\sqrt{3}}{2}$ (see, e.g. \cite{BBRSTT}) to determine 
\begin{equation}\label{Eqn:T1Avalue}
	T_{1A}\left(i\frac{\sqrt{3}}{2}\right) = j\left(i\frac{\sqrt{3}}{2}\right) - 744 = 4\cdot 15^3(30-17\sqrt{3})^3 - 744 \approx 1565.5777\ldots. 
\end{equation}

Next, we compute $T_{2A}\left(\frac{i}{2}\right)$ using the expression \[ T_{2A}(q) = \left(\frac{\eta(q)}{\eta(q^2)}\right)^{24} + 2^{12}\left(\frac{\eta(q^2)}{\eta(q)}\right)^{24} + 24 \] from Conway and Norton \cite{CN}. Noting that $\Delta(q) = \eta(q)^{24}$ is an eigenfunction of all (classical) Hecke operators $T_n$ (\cite{SERRE}), we use the Hecke operator of level 2 to find 
\[ -24\Delta(i) = \left(\Delta\big\vert T_2\right)(i) := 2^{11}\Delta(2i) + \frac{1}{2}\Delta\left(\frac{i}{2}\right) + \frac{1}{2}\Delta\left(\frac{i+1}{2}\right). \]
As a weight 12 modular form for $\PSL_2(\zz)$, $\Delta$ satisfies $\Delta(2i) = 2^{-12}\Delta(i/2)$ and $\Delta\left(\frac{i+1}{2}\right) = -2^6\Delta(i)$, so we deduce that $\Delta\left(\frac{i}{2}\right) = 8\Delta(i)$. Thus, we find
\begin{equation}\label{Eqn:T2Avalue}
	T_{2A}\left(\frac{i}{2}\right) = \frac{\Delta\left(\frac{i}{2}\right)}{\Delta(i)} + 2^{12}\frac{\Delta(i)}{\Delta\left(\frac{i}{2}\right)} + 24 = 8 + 512 + 24 = 544. 
\end{equation}

Now, we compute $T_{3A}\left(\frac{i}{2\sqrt{3}}\right)$. From Shigezumi \cite{SHI} $T_{3A}(\tau) = q^{-1} + 783q + \ldots$ takes values on $[-42,66]$ on the lower boundary, and in particular, $T_{3A}\left(\frac{i}{\sqrt{3}}\right) = 66$ and $T_{3A}\left(\frac{i}{2\sqrt{3}} + \frac{1}{2}\right) = -42$. The second Faber polynomial for $T_{3A}$ is $F_{2,3A}(X) = X^2 - 2\cdot 783$. Let $\tau_0 = \frac{i}{\sqrt{3}}$, so $T_{3A}(\tau_0) = 66$, and $F_{2,3A}\left(T_{3A}(\tau_0)\right) = 2790$. 

Now, observe that \[ T_{3A}\left( \frac{i}{2\sqrt{3}}\right) = T_{3A}\left(\frac{\tau_0}{2}\right) = T_{3A}(2\tau_0), \] since  $2\tau_0 = \frac{2i}{\sqrt{3}} = \frac{-1}{3\frac{\tau_0}{2}}$. Moreover, $T_{3A}(\frac{\tau_0+1}{2}) = T_{3A}\left(\frac{i}{2\sqrt{3}} + \frac{1}{2}\right) = -42$. The twisted Hecke operator of level 2 from Monstrous Moonshine then gives
\begin{align*} 
	2790 &= F_{2,3A}(\tau_0) \\
	&= \sum_{H\in\mc{H}_2} T_{3A}(H\tau_0) \\
	&= T_{3A}\left(2\tau_0\right) + T_{3A}\left(\frac{\tau_0}{2}\right) + T_{3A}\left(\frac{\tau_0+1}{2}\right) \\
	&= 2T_{3A}\left(\frac{i}{2\sqrt{3}}\right) - 42,
\end{align*}
giving 
\begin{equation}\label{Eqn:T3Avalue} 
	T_{3A}\left(\frac{i}{2\sqrt{3}}\right) = \frac{2832}{2} = 1416.
\end{equation}

Next, we need the value of $T_{6A}\left(\frac{i}{3\sqrt{2}}\right)$. Conway and Norton \cite{CN} provide the following identities:
\[ T_{6A} + 12 = t_{6A} = t_{6B} + \frac{1}{t_{6B}}, \]
\[ t_{6B} = \left[ \frac{\eta(2\tau)\eta(3\tau)}{\eta(\tau)\eta(6\tau)} \right]^{12}, \]
\[ t_{2B} = \left[ \frac{\eta(\tau)}{\eta(2\tau)} \right]^{24}, \]
\[ T_{2A} - 24 = t_{2A} = t_{2B} + \frac{4096}{t_{2B}}. \]

A little algebra gives us that \[ t_{6A}^2 = t_{6B}^2 + \frac{1}{t_{6B}^2} + 2, \] where  \[ t_{6B}(\tau)^2 = \frac{ t_{2B}(3\tau) }{ t_{2B}(\tau) }, \qquad \text{and} \qquad t_{2B} = \frac{t_{2A}}{2} \pm \sqrt{\frac{t_{2A}^2}{4} - 4096}.\] Since $3\tau_0 = \frac{i}{\sqrt{2}}$, we already know $T_{2A}(3\tau_0) = 152$ (\cite{SHI}), implying $t_{2B}(3\tau_0) = 64$, and thus \[ t_{6A}^2(\tau_0) = \frac{64}{t_{2B}(\tau_0)} + \frac{t_{2B}(\tau_0)}{64} + 2. \]

To determine $t_{2B}(\tau_0)$, we find $T_{2A}(\tau_0)$ using the theory of twisted Hecke operators and Faber polynomials. We compute that the Faber polynomial $F_{3,2A}(X) = X^3 - 13118X - 288768$, so that \[ F_{3,2A}(T_{2A}(\frac{i}{\sqrt{2}})) = F_{3,2A}(152) = 1229408. \] By Monstrous Moonshine,
\begin{align*}
	F_{3,2A}\left(T_{2A}\left(\frac{i}{\sqrt{2}}\right)\right)
	&= \sum_{H\in\mc{H}_3} T_{2A}^{(H)}\left(\frac{i}{\sqrt{2}}\right) \\
	&= T_{2A}\left(\frac{3i}{\sqrt{2}}\right) + T_{2A}\left(\frac{i}{3\sqrt{2}}\right) + T_{2A}\left(\frac{i}{3\sqrt{2}}+\frac{1}{3}\right) + T_{2A}\left(\frac{i}{3\sqrt{2}}+\frac{2}{3}\right).
\end{align*}
Here, $\frac{3i}{\sqrt{2}}$ and $\frac{i}{3\sqrt{2}}$ belong to the same orbit under $\Fricke{2}$, while both $\frac{i}{3\sqrt{2}}+\frac{1}{3}$ and $\frac{i}{3\sqrt{2}}+\frac{2}{3}$ belong to the orbit of $\frac{i}{\sqrt{2}}$ (and recall $T_{2A}(\frac{i}{\sqrt{2}}) = 152$, from \cite{SHI}). Thus by modularity of $T_{2A}$, we have \[ T_{2A}(\tau_0) = \frac{F_{3,2A}(152) - 2\cdot 152}{2} = 614552, \] implying $t_{2B}(\tau_0) = 307264 \pm 125440\sqrt{6}$. We do not need to determine the sign here (though it is not hard to show it must be negative), since we may determine
\begin{align*} 
	t_{6A}^2(\tau_0)
	&= \frac{64}{t_{2B}(\tau_0)}+\frac{t_{2B}(\tau_0)}{64}+2 \\
	&= \frac{1}{4801\pm 1960\sqrt{6}} + (4801\pm 1960\sqrt{6}) + 2 \\
	&= (4801\mp 1960\sqrt{6}) + (4801\pm 1960\sqrt{6}) + 2 \\
	&= 9604.
\end{align*}
Since $T_{6A}$ is real-valued and positive for $\tau$ on the imaginary axis (having only non-negative $q$-coefficients), we conclude 
\begin{equation}\label{Eqn:T6Avalue} 
	T_{6A} = +\sqrt{9604} - 12 = 86.
\end{equation}

Finally, we compute $T_{3C}\left(\frac{i\sqrt{3}}{6}\right)$. We use the relation $T_{3C}(\tau) = \sqrt[3]{j(3\tau)}$ (\cite{CN}), giving 
\begin{equation}\label{Eqn:T3Cvalue} 
	T_{3C}\left(\frac{i\sqrt{3}}{6}\right) = \sqrt[3]{j\left(i\frac{\sqrt{3}}{2}\right)} = 15\sqrt[3]{4}(30-17\sqrt{3}), 
\end{equation}
using the value of $j$ from the case of $T_{2A}$ above.

\bibliographystyle{amsplain}
\bibliography{ZerosArxiv.bib}

\end{document}